\documentclass[reqno]{amsart}
\usepackage[left=1in,right=1in,top=1in,bottom=1in]{geometry}
 \usepackage{microtype}
\usepackage{hyperref}
         
\usepackage{amsmath}             
\usepackage{amsfonts}             
\usepackage{amsthm}               
\usepackage{amsbsy}
\usepackage{amssymb}
\usepackage{amsthm}
\usepackage{amsrefs}

\usepackage{enumerate}

\newtheorem{thm}{Theorem}[section]
\newtheorem{lem}[thm]{Lemma}
\newtheorem{prop}[thm]{Proposition}
\newtheorem{cor}[thm]{Corollary}

\theoremstyle{definition}
\newtheorem{defn}[thm]{Definition}
\newtheorem{nota}[thm]{Notation}
\newtheorem{rem}[thm]{Remark}
\newtheorem{exam}[thm]{Example}

\newcommand{\bC}{{\mathbb{C}}}
\newcommand{\bN}{{\mathbb{N}}}

\newcommand{\A}{{\mathcal{A}}}

\newcommand{\D}{{\mathcal{D}}}

\newcommand{\F}{{\mathcal{F}}}
\renewcommand{\H}{{\mathcal{H}}}
\newcommand{\I}{{\mathcal{I}}}

\renewcommand{\L}{{\mathcal{L}}}
\newcommand{\M}{{\mathcal{M}}}

\renewcommand{\P}{{\mathcal{P}}}

\newcommand{\X}{{\mathcal{X}}}

\renewcommand{\phi}{\varphi}

\newcommand{\fB}{{\mathfrak{B}}}
\newcommand{\fM}{{\mathfrak{M}}}
\newcommand{\fN}{{\mathfrak{N}}}

\newcommand{\qand}{\quad\text{and}\quad}
\newcommand{\qqand}{\qquad\text{and}\qquad}

\newcommand{\alg}{\mathrm{alg}}

\renewcommand{\th}{\mathrm{th}}
\newcommand{\diag}{\mathrm{diag}}
\newcommand{\op}{\mathrm{op}}

\newcommand{\bncs}{{BNC_{vs}}}
\newcommand{\inv}{{\langle -1 \rangle}}

\usepackage{tikz}
\usetikzlibrary{shapes,snakes,calc,arrows}
\usetikzlibrary{decorations.pathreplacing,shapes.geometric}
\usetikzlibrary{calc,positioning}
\tikzset{Box/.style={very thick, rounded corners}}
\tikzset{marked/.style={star, star point height = .75mm, star points =5, fill=black,minimum size=2mm, inner sep=0mm} }
\tikzset{verythickline/.style = {line width=7pt}}
\tikzset{thickline/.style = {line width=5pt}}
\tikzset{medthick/.style = {line width=3pt}}
\tikzset{med/.style = {line width=2pt}}
\tikzset{count/.style = {fill=white,circle,draw,thin, inner sep=2pt}}
\tikzset{rcount/.style = {fill=white,rectangle,draw,thin,inner sep=2pt, rounded corners}}
\tikzset{cpr/.style = {draw,fill=white,rectangle,thin, rounded corners}}

\definecolor{ggreen}{HTML}{00BB33}

\begin{document}

\nocite{*}

\title[On Operator-Valued Bi-Free Distributions]{On Operator-Valued Bi-Free Distributions}

\author{Paul Skoufranis}
\address{Department of Mathematics, Texas A\&M University, College Station, Texas, USA, 77843}
\email{pskoufra@math.tamu.edu}

\date{\today}

\begin{abstract}
In this paper, operator-valued bi-free distributions are investigated.  Given a subalgebra $D$ of a unital algebra $B$, it is established that a two-faced family $Z$ is bi-free from $(B, B^{\mathrm{op}})$ over $D$ if and only if certain conditions relating the $B$-valued and $D$-valued bi-free cumulants of $Z$ are satisfied.  Using this, we verify that a two-faced family of matrices is $R$-cyclic if and only if they are bi-free from the scalar matrices over the scalar diagonal matrices.  Furthermore, the operator-valued bi-free partial $R$-, $S$-, and $T$-transforms are constructed.  New proofs of results from free probability are developed in order to facilitate many of these bi-free results.
\end{abstract}

\maketitle

\section{Introduction}
\label{sec:intro}

The notion of bi-free pairs of faces was introduced by Voiculescu in \cite{V2014} as a theory to enable the simultaneous study of the left and right actions of algebras on free product of vector spaces.  Initially postulated by Mastnak and Nica in \cite{MN2015} and demonstrated by Charlesworth, Nelson, and the author in \cite{CNS2015-1}, the combinatorial structures for bi-free probability are bi-non-crossing partitions; partitions that are non-crossing once a specific permutation is applied.  Consequently, as the combinatorics of free and bi-free probability are intimately related, many results from free probability have immediate generalizations to the bi-free setting.  Other bi-free results, such as bi-free partial transforms (see \cites{V2015, V2013, S2015-1, S2015-2}) and bi-matrix models (see \cite{S2015-3}), require additional work.

Although briefly examined in \cite{V2014}, operator-valued bi-free probability (a generalization of bi-free probability where the scalars $\bC$ are replaced with an arbitrary unital algebra $B$) has received less attention.  In \cite{CNS2015-2}, Charlesworth, Nelson, and the author demonstrated that the combinatorics of operator-valued bi-free probability is similar to the combinatorics of operator-valued free probability, yet addition care had to be taken.

Operator-valued free probability has been incredible useful as its study enlarges the domain of mathematics where free probability techniques may be applied.  This greater framework yields its own interesting results and allows additional techniques, which simplify arguments in free probability.  Of course, the trade-off of this wider framework is that results are more difficult to obtain.  These same ideas should resonate in operator-valued bi-free probability; the techniques should yield interesting results at the expense of the arguments being greater in difficulty.  Furthermore, as bi-free probability is in its infancy, a further understanding of bi-free probability can be obtain through studying operator-valued bi-freeness as intuition from the operator-valued setting can yield results in the scalar setting (as was the case with the bi-matrix models in \cite{S2015-3}).

Unfortunately, few concrete examples of bi-free pairs of $B$-faces are in existence to derive intuition from.  The most natural example comes from considering a type II$_1$ factor $\fM$, a von Neumann subalgebra $\fN$ of $\fM$, and the conditional expectation $E_\fN : \fM \to \fN$ of $\fM$ onto $\fN$.  To construct pairs of $\fN$-faces, consider the linear maps on $\fM$, denoted $\L(\fM)$, and the expectation $E : \L(\fM) \to \fN$ defined by
\[
E(T) = E_\fN(T(1_\fM))
\]
for all $T \in \L(\fM)$; that is, apply $T$ to $1_\fM$ and take the expectation of $\fM$ onto $\fN$.  Further define $^*$-homomorphisms $L : \fM \to \L(\fM)$ and $R : \fM^\op \to \L(\fM)$ by
\[
L(X)(A) = XA \qqand R(X)(A) = AX
\]
for all $X, A \in \fM$.
For this discussion we will call $L(X)$ a left operator and $R(X)$ a right operator.  If $\fM = \fM_1 \ast_\fN \fM_2$, the amalgamated free product of von Neumann algebras $\fM_1$ and $\fM_2$ over $\fN$, then the pairs of $\fN$-faces $(L(\fM_1), R(\fM_1^\op))$ and $(L(\fM_2), R(\fM_2^\op))$ are bi-free with amalgamation over $\fN$ with respect to $E$.

Like free probability, bi-free probability is a theory which describes the joint moments of operators.  In particular, free probability may be viewed as a subcase of bi-free probability where only left operators are considered thereby providing intuition for the bi-free case.  Recall, to compute the $\fN$-distribution of elements of $\fM$, given $X_1, \ldots, X_n \in \fM$ one takes $T_1, \ldots, T_n, T'_1, \ldots, T'_n \in \fN$ and computes the joint moment of 
\[
L(T_1) L(X_1) L(T'_1), \ldots, L(T_n) L(X_n) L(T'_n),
\]
which is
\[
E(L(T_1) L(X_1) L(T'_1) \cdots L(T_n) L(X_n) L(T'_n)) = T_1 E_\fN(X_1 (T'_1 T_2) X_2 \cdots (T'_{n-1} T_n) X_n) T'_n.
\]
In particular, although the joint moment depends on all of $ X_1, \ldots, X_n, T_1, \ldots, T_n, T'_1, \ldots, T'_n$, the moment really only depends on the $\fM$-operators $X_1, \ldots, X_n$ and the $\fN$-values of $T_1$, $T'_1T_2$, $\ldots$, $T'_{n-1}T_n$, and $T'_n$. This is realized by combining $L(T'_k)$ and $L(T_{k+1})$ into a single element of $L(\fN)$.  Furthermore, one need not include $L(T_1)$ and $L(T'_n)$ in the product as one may always pull $T_1$ out on the left and $T'_n$ out on the right.

Similar arguments to those above apply if one only considers right operators.  However, things become incredible more complex if one simultaneously considers left and right operators.  For example, note that if $X \in \fM$, $Y \in \fM^\op$, and $T \in \fN$, then
\[
E(L(X) R(Y) L(T)) = E_\fN(XTY).
\]
In particular, as $X, Y$ are elements of $\fM$, the $\fN$-operator $T$ cannot be moved outside of the expectation even though $L(T)$ was the right-most operator in the product.  This is startlingly different from the free case where the first left (or right) element of $\fN$ to act (i.e. the right-most in the product in both cases) may be moved outside of the expectation.  Furthermore, notice 
\[
E(L(X) R(Y) R(T)) = E_\fN(XTY) = E(L(X) R(Y) L(T));
\]
that is, the result is the same no matter whether we view the first $\fN$-operator acting as a left operator or as a right operator.  In particular, the first left and right $\fN$-operators to act behave in a special manner.

For further illustration, consider $X \in \fM$, $Y \in \fM^{\op}$, $T, T' \in \fN$, and $S, S'\in \fN^{\op}$.  Then
\[
E(L(T) L(X) L(T') R(S) R(Y) R(S')) = T E_\fN(X (T'S') Y) S.
\]
Like before,  although the joint moment depends on all of $X, Y, T, T', S, S'$, the moment really only depends on the $\fM$-operators $S, X$ and the $\fN$-values of $T, S, T'S'$.  This implies the first left and right $\fN$ operators to act combine into a single $\fN$-operator that is special in the sense that it can be treated as either a left \emph{or} as a right operator.  

The goal of this paper is to better understand operator-valued bi-free distributions over an arbitrary algebra $B$.  In Section \ref{sec:prelims}, preliminaries on operator-valued bi-free pairs of $B$-faces common to the majority of sections of this paper are developed.  After these important definitions, constructions, and results are collected, a description of which joint moments of a pair of $B$-faces $(A_\ell, A_r)$ are required to completely describe all joint moments is provided based on the above.  In particular, if elements of both $A_\ell$ and $A_r$ are considered (which is the interesting case in bi-free probability), then it is always the case that there is precisely one $B$-operator that can be viewed as a left \emph{and} as a right $B$-operator.  It is this special $B$-operator that must be treated with some care and produces very interesting peculiarities in the theory.

In Section \ref{sec:sub}, a powerful result of Nica, Shlyakhtenko, and Speicher from \cite{NSS2002} is generalized to the bi-free setting.  Given a subalgebra $D$ of $B$, \cite{NSS2002} examined of how the $B$-valued distributions of random variables interacts with the $D$-valued distributions.  In particular, \cite{NSS2002}*{Theorem 3.5} provides conditions between the $B$-valued and $D$-valued cumulants of random variables $X_1, \ldots, X_n$ in order to completely classify when $X_1, \ldots, X_n$ were free from $B$ over $D$.  Said result is quite a useful theoretical tool and has the application that if one wants to compute the $D$-valued cumulants of $X_1, \ldots, X_n$ for some $D$, one may extend $D$ to an algebra $B$ where the cumulants may be easier to compute.  This application is prevalent in the study of $R$-cyclic matrices (see \cite{NSS2002-R} for example).

Section \ref{sec:sub} extends \cite{NSS2002}*{Theorem 3.5} to the bi-free setting.  Specifically Theorem \ref{thm:bi-free-over-D} shows $(\{Z_i\}_{i\in I}, \{Z_j\}_{j \in J})$ are bi-free from $(B, B^\op)$ with amalgamation over $D$ if and only nearly identical conditions to those in \cite{NSS2002}*{Theorem 3.5} on the $B$-valued and $D$-valued bi-free cumulants are satisfied.  However, the proof \cite{NSS2002}*{Theorem 3.5} required two techniques in free probability that do not have bi-free analogues.  First \cite{NSS2002}*{Theorem 3.5} makes use of the `alternating centred moments vanish' characterization of free independence, which does not have a bi-free analogue.  Secondly, \cite{NSS2002}*{Theorem 3.5} uses the notion of canonical random variables, which is an abstraction of the commonly known Fock space operator model from \cite{N1996}.  Although there is a bi-free operator model on a Fock space in \cite{CNS2015-1}, a complication arises in its abstraction due to the necessity of left (right) operators to commute with right (left) $B$-operators.  However, as the conditions in Theorem \ref{thm:bi-free-over-D} are related to the operator-valued bi-free cumulants, a combinatorial argument manipulating the cumulants produces the result (and a new proof of \cite{NSS2002}*{Theorem 3.5}).

In Section \ref{sec:R-cyc}, Theorem \ref{thm:bi-free-over-D} is used to produce simple examples of operator-valued bi-free pairs of $B$-faces.  In \cite{NSS2002-R}, Nica, Shlyakhtenko, and Speicher analyzed the notion of $R$-cyclic families of matrices and demonstrated these are precisely the matrices of operators that are free from the scalar matrices over the scalar diagonal matrices.    Theorem \ref{thm:R-cyc} provides the bi-free version of this result by using the same notion required for bi-matrix models in \cite{S2015-3}.  In particular, this provides additional evidence that the constructions of \cite{S2015-3} are the correct ones to consider for matrices of operators in the bi-free setting and gives concrete examples of operators that are bi-free from scalar matrix algebra over scalar diagonal matrices.

In free probability theory, the $R$- and $S$-transforms provide substantial information about the joint moments of freely independent random variables under additive and multiplicative free convolution respectively.  These transforms have analogues in the operator-valued setting: the operator-valued $R$-transform was developed by Voiculescu in \cite{V1995}, and the operator-valued $S$-transform was developed by Dykema in \cite{D2006} (also see \cite{A2004}).   The goal of Sections \ref{sec:R-Trans}, \ref{sec:T}, and \ref{sec:S} are to generalize the bi-free partial $R$-, $T$-, and $S$-transforms from \cites{V2015, V2013, S2015-1, S2015-2} to the operator-valued setting (where the $T$-transform is desired for additive free convolution in one element of the pair and multiplicative free convolution in the other element of the pair).

A peculiarity with the operator-valued bi-free partial $R$-, $T$-, and $S$-transform is the number of variables these transformations are functions of.  One may expect these to be functions of two $B$-variables; one for a left $B$-operator and one for a right $B$-operator.  However, a third $B$-variable is required to handle the special $B$-operator that may be viewed as a left and as a right $B$-operator.  This third variable is required since when one examines reduction of the operator-valued bi-free cumulant corresponding to a bi-non-crossing partition, one of these special $B$-operators is always created from the other left and right operators.  More surprisingly, the operator-valued bi-free partial $R$-, $T$-, and $S$-transforms of their respected convolutions are \emph{compositions} in this third $B$-operator of one transform in the family by the other (i.e. the bottom of a bi-non-crossing diagram behaves differently that the sides).  The definitions and formulae for the partial $R$-transform (Theorem \ref{thm:R-transform}), the partial $T$-transform (Definition \ref{defn:T-Transform} and Theorem \ref{thm:T-property}), and the partial $S$-transform (Definition \ref{defn:S-Transform} and Theorem \ref{thm:S-property}) all reduce to the same formula as in \cites{V2015, V2013, S2015-1, S2015-2} in the scalar-valued setting.  Furthermore, partial $S$-transform of a convolution is particularly nice as the equation looks symmetric, even though the operator-valued free $S$-transform of \cite{D2006} is not symmetric.

The approach take in Sections \ref{sec:R-Trans}, \ref{sec:T}, and \ref{sec:S} is the same combinatorial approach for the scalar-valued bi-free partial transforms taken in \cites{S2015-1, S2015-2}.  However, for the partial $T$- and $S$-transforms, a new proof of the operator-valued $S$-transform formula  from \cite{D2006} is needed to facilitate these arguments.  Consequently, Section \ref{sec:free-S} is devoted to developing another proof of \cite{D2006}*{Theorem 1.1} along the lines of Nica's and Speicher's ``Fourier'' transform approach to multiplicative convolution from \cite{NS1997}.

\section{Common Preliminaries}
\label{sec:prelims}

This section will recall and develop most of the preliminary results required throughout this paper.  More background will be recalled in the other sections as required.

\subsection*{Bi-Non-Crossing Partitions}

Throughout this paper, a map $\chi : \{1,\ldots, n\} \to \{\ell, r\}$ is used to designate whether the $k^{\mathrm{th}}$ operator in a sequence of $n$ operators is a left operator (when $\chi(k) = \ell$) or a right operator (when $\chi(k) = r$).  Furthermore, we will often need to index the set of left and right operators.  Consequently, if $I$ is an indexing set for left operators and $J$ is an indexing set for right operators, a map $\omega : \{1, \ldots, n\} \to I \sqcup J$ is used to designate the index in $I \sqcup J$ for each operator in the sequence.  Given $\omega : \{1, \ldots, n\} \to I \sqcup J$, the corresponding map $\chi_\omega : \{1,\ldots, n\} \to \{\ell, r\}$ is defined by
\[
\chi_\omega(k) = \left\{
\begin{array}{ll}
\ell & \mbox{if } \omega(k) \in I \\
r & \mbox{if } \omega(k) \in J
\end{array} \right. .
\]

Given $\chi: \{1, \ldots, n\} \to \{\ell, r\}$ with
\[
\chi^{-1}(\{\ell\}) = \{k_1<\cdots<k_p\} \qqand \chi^{-1}(\{r\}) = \{k_{p+1} > \cdots > k_n\},
\]
define the permutation $s_\chi$ on $\{1,\ldots, n\}$ by $s_\chi(q) = k_q$.  The only differences between the combinatorial aspects of free and bi-free probability arise from dealing with $s_\chi$.  

Using $s_\chi$, define the total ordering $\prec_\chi$ on $\{1,\ldots, n\}$ by $k_1 \prec_\chi k_2$ if and only if $s_\chi^{-1} (k_1)< s_\chi^{-1}(k_2)$.  Instead of reading $\{1,\ldots, n\}$ in the traditional order, $\prec_\chi$ corresponds to reading $\chi^{-1}(\{\ell\})$ in increasing order followed by reading $\chi^{-1}(\{r\})$ in decreasing order.  

A subset $V \subseteq \{1,\ldots, n\}$ is said to be a \emph{$\chi$-interval} if $V$ is an interval with respect to the ordering $\prec_\chi$.
In addition, $\min_{\prec_\chi}(V)$ and $\max_{\prec_\chi}(V)$ denote the minimal and maximal elements of $V$ with respect to the ordering $\prec_\chi$.

\begin{defn}
A partition $\pi \in \P (n)$ is said to be \emph{bi-non-crossing with respect to $\chi$} if the partition  $s_\chi^{-1}\cdot \pi$ (the partition formed by applying $s_\chi^{-1}$ to each entry of each block of $\pi$) is non-crossing.  Equivalently $\pi$ is bi-non-crossing if whenever there are blocks $U, V \in \pi$ with $u_1, u_2 \in U$ and $v_1, v_2 \in V$ such that 
\[
u_1 \prec_\chi v_1 \prec_\chi u_2 \prec_\chi v_2,
\]
then $U = V$.  The set of bi-non-crossing partitions with respect to $\chi$ is denoted by $BNC(\chi)$.
\end{defn}

\begin{exam}
\label{exam:FLOP}
If $\chi : \{1,\ldots, 6\} \to \{\ell, r\}$ is such that $\chi^{-1}(\{\ell\}) = \{1, 2, 3, 6\}$ and $\chi^{-1}(\{r\}) = \{4, 5\}$, then $(s_\chi(1), \ldots, s_\chi(6)) = (1, 2, 3, 6, 5, 4)$.  If $\pi = \{ \{1,4\}, \{2,5\}, \{3, 6\}\}$, then $\pi$ is crossing on $\{1,\ldots, 6\}$ yet is bi-non-crossing with respect to $\chi$.  This may be seen via the following diagrams.
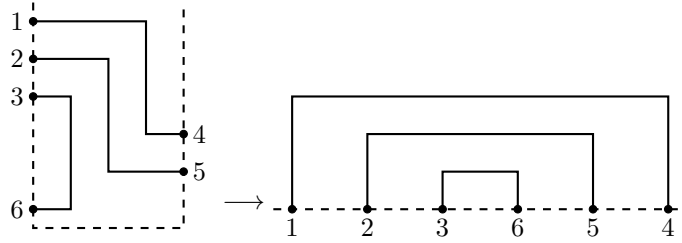
\begin{align*}
\begin{tikzpicture}[baseline]
	\draw[thick, dashed] (-1,2.75) -- (-1,-.25) -- (1,-.25) -- (1,2.75);
	\draw[thick] (-1, 1.5) -- (-.5,1.5) -- (-0.5,0) -- (-1, 0);
	\draw[thick] (-1, 2.5) -- (0.5,2.5) -- (0.5,1) -- (1, 1);
	\draw[thick] (-1, 2) -- (0,2) -- (0,.5) -- (1, .5);
	\node[right] at (1, 1) {$4$};
	\draw[black, fill=black] (1,1) circle (0.05);
	\node[right] at (1, .5) {$5$};
	\draw[black, fill=black] (1,.5) circle (0.05);
	\node[left] at (-1, 0) {$6$};
	\draw[black, fill=black] (-1,0) circle (0.05);
	\node[left] at (-1, 2.5) {$1$};
	\draw[black, fill=black] (-1,2.5) circle (0.05);	
	\node[left] at (-1, 2) {$2$};
	\draw[black, fill=black] (-1,2) circle (0.05);
	\node[left] at (-1, 1.5) {$3$};
	\draw[black, fill=black] (-1,1.5) circle (0.05);
\end{tikzpicture}
\longrightarrow
\begin{tikzpicture}[baseline]
	\draw[thick, dashed] (-.25,0) -- (5.25, 0);
	\draw[thick] (5, 0) -- (5,1.5) -- (0,1.5) -- (0, 0);
	\draw[thick] (1, 0) -- (1,1) -- (4,1) -- (4, 0);
	\draw[thick] (2, 0) -- (2,.5) -- (3,.5) -- (3, 0);
	\node[below] at (4, 0) {$5$};
	\draw[black, fill=black] (4,0) circle (0.05);	
	\node[below] at (5, 0) {$4$};
	\draw[black, fill=black] (5,0) circle (0.05);	
	\node[below] at (0, 0) {$1$};
	\draw[black, fill=black] (0,0) circle (0.05);	
	\node[below] at (1, 0) {$2$};
	\draw[black, fill=black] (1,0) circle (0.05);	
	\node[below] at (2, 0) {$3$};
	\draw[black, fill=black] (2,0) circle (0.05);	
	\node[below] at (3, 0) {$6$};
	\draw[black, fill=black] (3,0) circle (0.05);
\end{tikzpicture}
\end{align*}
\end{exam}

Note $BNC(\chi)$ inherits a lattice structure from $\P(n)$ via refinement (i.e. $\pi \leq \sigma$ if $\pi$ is a refinement of $\sigma$) and thus has minimal and maximal elements, denoted $0_\chi$ and $1_\chi$ respectively.  Furthermore, given $\pi\in BNC(\chi)$ and $p,q \in \{1, \ldots, n\}$, $p \sim_\pi q$ is used to denote that $p$ and $q$ are in the same block of $\pi$ whereas $p \nsim_\pi q$ is used to denote that $p$ and $q$ are in different blocks of $\pi$.

For $n,m\geq 0$, define $\chi_{n,m} : \{1,\ldots, n+m\} \to \{\ell, r\}$ by $\chi(k) = \ell$ if $k \leq n$ and $\chi(k) = r$ if $k > n$.  For notation purposes in later sections, it will be useful to think of $\chi_{n,m}$ as a map on 
\[
\{1_\ell, 2_\ell, \ldots, n_\ell, 1_r, 2_r, \ldots, m_r\}
\]
under the identification $k \mapsto k_\ell$ if $k \leq n$ and $k \mapsto (k-n)_r$ if $k > n$. Furthermore, denote $BNC(n,m)$ for $BNC(\chi_{n,m})$ and  $1_{n,m}$ for $1_{\chi_{n,m}}$.

\begin{defn}
\label{defn:mobius}
The \emph{bi-non-crossing M\"{o}bius function} is the function
\[
\mu_{BNC} : \bigcup_{n\geq1}\bigcup_{\chi : \{1, \ldots, n\}\to\{\ell, r\}}BNC(\chi)\times BNC(\chi) \to \bC
\]
defined such that $\mu_{BNC}(\pi, \sigma) = 0$ unless $\pi$ is a refinement of $\sigma$, and otherwise defined recursively via the formulae
\[
\sum_{\substack{\tau \in BNC(\chi) \\\pi \leq \tau \leq \sigma}} \mu_{BNC}(\tau, \sigma) = \sum_{\substack{\tau \in BNC(\chi) \\ \pi \leq \tau \leq \sigma}} \mu_{BNC}(\pi, \tau) = \left\{
\begin{array}{ll}
1 & \mbox{if } \pi = \sigma  \\
0 & \mbox{otherwise}
\end{array} \right. .
\]
\end{defn}

Due to the similarity in lattice structures, the bi-non-crossing M\"{o}bius function is related to the non-crossing M\"{o}bius function $\mu_{NC}$ by the formula 
\[
\mu_{BNC}(\pi, \sigma) = \mu_{NC}(s^{-1}_\chi \cdot \pi, s^{-1}_\chi \cdot \sigma).
\]
This implies that $\mu_{BNC}$ inherits many `multiplicative' properties that $\mu_{NC}$ has.  For more details, see \cite{CNS2015-1}*{Section 3}.

\subsection*{$\boldsymbol{B}$-$\boldsymbol{B}$-Non-Commutative Probability Spaces}  There are very specific structures that operator-valued bi-free probability apply to.  We begin by recalling the structures for operator-valued free probability first.

\begin{defn}
A \emph{$B$-non-commutative probability space} is a pair $(\A, E)$ where $\A$ is a unital algebra containing $B$ (with $1_\A= 1_B$) and $E : \A \to B$ is a unital linear map such that
\[
E(b_1 Z b_2) = b_1  E(Z) b_2
\]
for all $b_1, b_2 \in B$ and $Z \in \A$.
\end{defn}

\begin{defn}
\label{defn:BBncps}
A \emph{$B$-$B$-non-commutative probability space} is a triple $(\A, E, \varepsilon)$ where $\A$ is a unital algebra, $\varepsilon : B \otimes B^{\op} \to \A$ is a unital homomorphism such that $\varepsilon|_{B \otimes 1_B}$ and $\varepsilon|_{1_B \otimes B^{\op}}$ are injective, and $E : \A \to B$ is a linear map such that
\[
E(\varepsilon(b_1 \otimes b_2)Z) = b_1 E_{\A}(Z) b_2
\qqand
E(Z\varepsilon(b \otimes 1_B)) = E_{\A}(Z\varepsilon(1_B \otimes b))
\]
for all $b_1, b_2, b \in B$ and $Z \in \A$.  To simplify notation, $L_b$ and $R_b$ may be used in place of $\varepsilon(b \otimes 1_B)$ and $\varepsilon(1_B \otimes b)$ respectively. 

The unital subalgebras of $\A$ defined by
\begin{align*}
\A_\ell &:= \{ Z \in \A  \, \mid \, Z R_b = R_b Z \mbox{ for all }b \in B\} \text{ and}\\
\A_r &:= \{ Z \in \A  \, \mid \, Z L_b = L_b Z \mbox{ for all }b \in B\}
\end{align*}
are called the \emph{left} and \emph{right algebras of} $\A$ respectively.  
\end{defn}

For the purposes of this paper, view $B^\op$ as $B$ with the map $b \mapsto R_b$ being anti-multiplicative on $B$.

Given a $B$-$B$-non-commutative probability space $(\A, E, \varepsilon)$, one can verify that $(\A_\ell, E)$ is a $B$-non-commutative probability space with $\varepsilon(B \otimes 1_B)$ as the copy of $B$ and $(\A_r, E)$ is a $B^{\op}$-non-commutative probability space with $\varepsilon(1_B \otimes B^{\op})$ as the copy of $B^{\op}$.   Thus the bi-free structure of a $B$-$B$-non-commutative probability space is a generalization of the free structure of a $B$-non-commutative probability space. Furthermore $(\A, E, \varepsilon)$ is simply a non-commutative probability space whenever $B = \bC$.

We refer the reader to \cite{CNS2015-2}*{Section 3} for further discussions on why $B$-$B$-non-commutative probability spaces are the correct setting to perform operator-valued bi-free probability.

\subsection*{Operator-Valued Bi-Multiplicative Functions}

In order to discuss the operator-valued bi-free moment and cumulant functions, it is necessary to describe what operations one may use to reduce said functions via blocks of bi-non-crossing partitions and how one many move elements of $\varepsilon(B \otimes B^{\op})$ around.  Said operations are essential to this paper so we recall these operations in the greatest detail possible.

 To begin, given $\chi: \{1,\ldots, n\} \to \{\ell, r\}$ and a subset $V \subseteq \{1,\ldots, n\}$, let $\chi|_V : V \to \{\ell, r\}$ denote the restriction of $\chi$ to $V$.  Similarly, given an $n$-tuple of objects $(Z_1, \ldots, Z_n)$, let $(Z_1, \ldots, Z_n)|_V$ denote the $|V|$-tuple where the elements in positions not indexed by an element of $V$ are removed.  Finally, given $\pi \in BNC(\chi)$ such that $V$ is a union of blocks of $\pi$, let $\pi|_V \in BNC(\chi|_V)$ denote the bi-non-crossing partition formed by taking the blocks of $\pi$ contained in $V$.

\begin{defn}
\label{defn:bi-multiplicative}
Let $(\A, E, \varepsilon)$ be a $B$-$B$-non-commutative probability space and let 
\[
\Phi : \bigcup_{n\geq 1} \bigcup_{\chi : \{1,\ldots, n\} \to \{\ell, r\}} BNC(\chi) \times \mathcal{A}_{\chi(1)} \times \cdots \times \mathcal{A}_{\chi(n)} \to B
\]
be a function that is linear in each $\mathcal{A}_{\chi(k)}$. It is said that $\Phi$ is \emph{bi-multiplicative} if for every $\chi : \{1,\ldots, n\} \to \{\ell, r\}$, $Z_k \in \mathcal{A}_{\chi(k)}$, $b \in B$, and $\pi \in BNC(\chi)$, the following four conditions all hold:
\begin{enumerate}
\item\label{part:bi-multi-1} Let
\[
q = \max\{ k \in \{1,\ldots, n\} \, \mid \, \chi(k) \neq \chi(n)\}.
\]
If $\chi(n) = \ell$ then
\[
\Phi_{1_\chi}(Z_1, \ldots, Z_{n-1}, Z_nL_b) = \left\{
\begin{array}{ll}
\Phi_{1_\chi}(Z_1, \ldots, Z_{q-1}, Z_q R_b, Z_{q+1}, \ldots, Z_n) & \mbox{if } q \neq -\infty  \\
\Phi_{1_\chi}(Z_1, \ldots, Z_{n-1}, Z_n)b & \mbox{if } q = -\infty
\end{array} \right. .
\]
If $\chi(n) = r$ then 
\[
\Phi_{1_\chi}(Z_1, \ldots, Z_{n-1}, Z_nR_b) = \left\{
\begin{array}{ll}
\Phi_{1_\chi}(Z_1, \ldots, Z_{q-1}, Z_q L_b, Z_{q+1}, \ldots, Z_n) & \mbox{if } q \neq -\infty  \\
b\Phi_{1_\chi}(Z_1, \ldots, Z_{n-1}, Z_n) & \mbox{if } q = -\infty
\end{array} \right. .
\]
\item\label{part:bi-multi-2} Let $p \in \{1,\ldots, n\}$ and let
\[
q = \max\{ k \in \{1,\ldots, n\} \, \mid \, \chi(k) = \chi(p), k < p\}.
\]
If $\chi(p) = \ell$ then
\[
\Phi_{1_\chi}(Z_1, \ldots, Z_{p-1}, L_bZ_p, Z_{p+1}, \ldots, Z_n) = \left\{
\begin{array}{ll}
\Phi_{1_\chi}(Z_1, \ldots, Z_{q-1}, Z_qL_b, Z_{q+1}, \ldots, Z_n) & \mbox{if } q \neq -\infty  \\
b \Phi_{1_\chi}(Z_1, Z_2, \ldots, Z_n) & \mbox{if } q = -\infty
\end{array} \right. .
\]
If $\chi(p) = r$ then
\[
\Phi_{1_\chi}(Z_1, \ldots, Z_{p-1}, R_bZ_p, Z_{p+1}, \ldots, Z_n) = \left\{
\begin{array}{ll}
\Phi_{1_\chi}(Z_1, \ldots, Z_{q-1}, Z_qR_b, Z_{q+1}, \ldots, Z_n) & \mbox{if } q \neq -\infty  \\
\Phi_{1_\chi}(Z_1, Z_2, \ldots, Z_n) b & \mbox{if } q = -\infty
\end{array} \right. .
\]
\item\label{part:bi-multi-3} 
Suppose that $V_1, \ldots, V_m$ each are unions of blocks of $\pi$ and are $\chi$-intervals which partition $\{1, \ldots, n\}$. Further, suppose $V_1, \ldots, V_m$ are ordered by $\prec_\chi$. Then
\[
\Phi_\pi(Z_1, \ldots, Z_n) = \Phi_{\pi|_{V_1}}\left((Z_1, \ldots, Z_n)|_{V_1}\right) \cdots \Phi_{\pi|_{V_m}}\left((Z_1, \ldots, Z_n)|_{V_m}\right).
\]
\item\label{part:bi-multi-4} Suppose that $V$ and $W$ are unions of blocks of $\pi$ that partition $\{1,\ldots, n\}$.  Further suppose that $V$ is a $\chi$-interval and
\[
\min_{\prec_\chi}(\{1,\ldots, n\}), \max_{\prec_\chi}(\{1,\ldots, n\}) \in W.
\]
Let
\[
p = \max_{\prec_\chi}\left(\left\{k \in W  \, \mid \, k \prec_\chi \min_{\prec_\chi}(V)\right\}\right) \qquad\mbox{ and } \qquad q = \min_{\prec_\chi}\left(\left\{k \in W  \, \mid \, \max_{\prec_\chi}(V) \prec_\chi k\right\}\right).
\]
Then
\begin{align*}
\Phi_\pi(Z_1, \ldots, Z_n) &= \left\{
\begin{array}{ll}
\Phi_{\pi|_{W}}\left(\left(Z_1, \ldots, Z_{p-1}, Z_p L_{\Phi_{\pi|_{V}}\left((Z_1,\ldots, Z_n)|_{V}\right)}, Z_{p+1}, \ldots, Z_n\right)|_{W}\right)  & \mbox{if } \chi(p) = \ell \\
\Phi_{\pi|_{W}}\left(\left(Z_1, \ldots, Z_{p-1}, R_{\Phi_{\pi|_{V}}\left((Z_1,\ldots, Z_n)|_{V}\right)} Z_p, Z_{p+1}, \ldots, Z_n\right)|_{W}\right)  & \mbox{if } \chi(p) = r 
\end{array} \right. \\
&= \left\{
\begin{array}{ll}
\Phi_{\pi|_{W}}\left(\left(Z_1, \ldots, Z_{q-1},  L_{\Phi_{\pi|_{V}}\left((Z_1,\ldots, Z_n)|_{V}\right)} Z_q, Z_{q+1}, \ldots, Z_n\right)|_{W}\right)  & \mbox{if } \chi(q) = \ell  \\
\Phi_{\pi|_{W}}\left(\left(Z_1, \ldots, Z_{q-1}, Z_q R_{\Phi_{\pi|_{V}}\left((Z_1,\ldots, Z_n)|_{V}\right)}, Z_{q+1}, \ldots, Z_n\right)|_{W}\right) & \mbox{if } \chi(q) = r
\end{array} \right. .
\end{align*}
\end{enumerate}
\end{defn}

\begin{rem}
\label{rem:bi-multi-same-as-free}
Given a bi-multiplicative function $\Phi$, the four conditions in Definition \ref{defn:bi-multiplicative} demonstrate how one may reduce $\Phi_\pi$ into an expression involving only $\Phi_{1_\chi}$'s, and how one may move elements of $\varepsilon(B \otimes B^{\op})$ around, thereby allowing multiple reductions.  

Note one may understand the four conditions of Definition \ref{defn:bi-multiplicative} via the notion of a multiplicative function in free probability as follows.  Given $\pi \in BNC(\chi)$ and a bi-multiplicative map $\Phi$, each reduction property one may apply to $\Phi_\pi(Z_1, \ldots, Z_n)$ follows by
\begin{enumerate}
\item viewing the non-crossing partition $s_\chi^{-1}  \cdot \pi$,
\item rearranging the $n$-tuple $(Z_1, \ldots, Z_n)$ to $(Z_{s_\chi(1)}, \ldots, Z_{s_\chi(n)})$, 
\item replacing any occurrences of $L_bZ_j$, $Z_j L_b$, $R_b Z_j$, and $Z_j R_b$ with $bZ_j$, $Z_j b$, $Z_j b$, and $bZ_j$ respectively,
\item applying one of the properties of a multiplicative map from \cite{NSS2002}*{Section 2.2}, and
\item reversing the above identifications.
\end{enumerate}  
\end{rem}

\begin{exam}
\label{exam:bi-mult-exam-1}
For an example of how properties (\ref{part:bi-multi-1}) and (\ref{part:bi-multi-2}) of Definition \ref{defn:bi-multiplicative} apply, consider $\chi_0 : \{1,2,3,4\} \to \{\ell, r\}$ by $\chi_0^{-1}(\{\ell\}) = \{1,3\}$.  If $\Phi$ is bi-multiplicative, $Z_k \in \mathcal{A}_{\chi(k)}$, and $\{b_k\}^{5}_{k=1} \subseteq B$, then
\[
\Phi_{1_{\chi_0}}(L_{b_1} Z_1, R_{b_2}Z_2, L_{b_3} Z_3, R_{b_4}Z_4R_{b_5}) = b_1 \Phi_{1_\chi}( Z_1 L_{b_3}, Z_2 R_{b_4}, Z_3 L_{b_5}, Z_4) b_2.
\]
Diagrammatically, view the sequence $L_{b_1} Z_1, R_{b_2}Z_2, L_{b_3} Z_3, R_{b_4}Z_4R_{b_5}$ from the top down in the bi-non-crossing diagram of $\pi$.  Each $B$-operator that occurs to the left of a $Z_k$ is viewed as above the corresponding node in the bi-non-crossing diagram and each $B$-operator that occurs to the right of a $Z_k$ is viewed as below the corresponding node.  One then may move each $B$-operator along the dotted lines until one either arrives at another node (changing between $L_b$ and $R_b$ if needed), escapes on the left (as $b_1$ does), or on the right (as $b_2$ does).
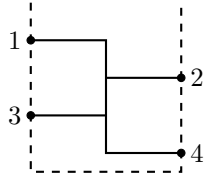
\begin{align*}
\begin{tikzpicture}[baseline]
	\draw[thick, dashed] (-1,2) -- (-1,-.25) -- (1,-.25) -- (1,2);
	\draw[thick] (-1, 1.5) -- (0,1.5) -- (-0,0) -- (1, 0);
	\draw[thick] (1, 1) -- (0,1);
	\draw[thick] (-1, .5) -- (0,.5);
	\node[left] at (-1, 1.5) {$1$};
	\draw[black, fill=black] (-1,1.5) circle (0.05);	
	\node[right] at (1, 1) {$2$};
	\draw[black, fill=black] (1,1) circle (0.05);
	\node[left] at (-1, .5) {$3$};
	\draw[black, fill=black] (-1,.5) circle (0.05);
	\node[right] at (1, 0) {$4$};
	\draw[black, fill=black] (1,0) circle (0.05);
\end{tikzpicture}
\end{align*}
\end{exam}

\begin{exam}
For an example of how properties (\ref{part:bi-multi-3}) and (\ref{part:bi-multi-4}) of Definition \ref{defn:bi-multiplicative} apply, consider $\chi : \{1,\ldots, 10\} \to \{\ell, r\}$ by $\chi^{-1}(\{\ell\}) = \{1,2, 3, 6, 7\}$ and $\pi \in BNC(\chi)$ with the following bi-non-crossing diagram:
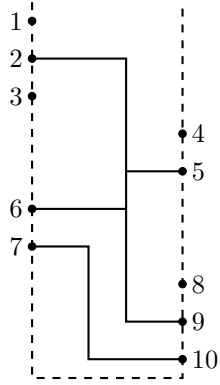
\begin{align*}
\begin{tikzpicture}[baseline]
	\draw[thick, dashed] (-1,4.75) -- (-1,-.25) -- (1,-.25) -- (1,4.75);
	\draw[thick] (-1, 4) -- (0.25,4) -- (0.25,0.5) -- (1, 0.5);
	\draw[thick] (1, 2.5) -- (.25,2.5);
	\draw[thick] (-1, 2) -- (0.25,2);
	\draw[thick] (-1, 1.5) -- (-0.25,1.5) -- (-0.25,0) -- (1, 0);
	\node[left] at (-1, 4.5) {$1$};
	\draw[black, fill=black] (-1,4.5) circle (0.05);	
	\node[left] at (-1, 4) {$2$};
	\draw[black, fill=black] (-1,4) circle (0.05);
	\node[left] at (-1, 3.5) {$3$};
	\draw[black, fill=black] (-1,3.5) circle (0.05);	
	\node[right] at (1, 3) {$4$};
	\draw[black, fill=black] (1,3) circle (0.05);
	\node[right] at (1, 2.5) {$5$};
	\draw[black, fill=black] (1,2.5) circle (0.05);
	\node[left] at (-1, 2) {$6$};
	\draw[black, fill=black] (-1,2) circle (0.05);
	\node[left] at (-1, 1.5) {$7$};
	\draw[black, fill=black] (-1,1.5) circle (0.05);	
	\node[right] at (1, 1) {$8$};
	\draw[black, fill=black] (1,1) circle (0.05);
	\node[right] at (1, .5) {$9$};
	\draw[black, fill=black] (1,.5) circle (0.05);
	\node[right] at (1, 0) {$10$};
	\draw[black, fill=black] (1,0) circle (0.05);
\end{tikzpicture}
\end{align*}
If $\Phi$ is bi-multiplicative and $Z_k \in \mathcal{A}_{\chi(k)}$, then
\[
 \Phi_\pi(Z_1, \ldots, Z_{10}) = \Phi_{1_{1,0}}(Z_1)  \Phi_{1_{\chi_0}}\left(Z_2 L_{\Phi_{1_{1,0}}(Z_3)}, Z_5 R_{\Phi_{1_{0,1}}(Z_8)}, Z_6, Z_9 R_{\Phi_{1_{1,1}}(Z_7, Z_{10})} \right)          \Phi_{1_{0,1}}(Z_4)
\]
where $\chi_0$ is as in Example \ref{exam:bi-mult-exam-1}. Furthermore, the same holds when any bi-non-crossing partition replaces the block containing $Z_7$ and $Z_{10}$ at the bottom of the bi-non-crossing diagram, when $Z_3$ is replaced with a non-crossing partition on the left-hand-side, $Z_8$ is replaced with a non-crossing partition on the right-hand-side, and $Z_1$ and $Z_4$ are replaced with non-crossing partitions on the left- and right-hand-sides respectively that do not have any common nodes.

Properties (\ref{part:bi-multi-3}) and (\ref{part:bi-multi-4}) of Definition \ref{defn:bi-multiplicative} enable one to apply $\Phi$ to a $\chi$-interval, and then move the resulting $B$-operator along the dotted lines until one either arrives at another node (yielding a $L_b$ and $R_b$ depending on the node), or escapes on the left or on the right.
\end{exam}

With the notion of bi-multiplicative complete, operator-valued bi-free moment and cumulant functions may be discussed.  Given a $B$-$B$-non-commutative probability space $(\A, E, \varepsilon)$, the \emph{operator-valued bi-free moment function}
\[
E^B : \bigcup_{n\geq 1} \bigcup_{\chi : \{1,\ldots, n\} \to \{\ell, r\}} BNC(\chi) \times \mathcal{A}_{\chi(1)} \times \cdots \times \mathcal{A}_{\chi(n)} \to B
\]
is the bi-multiplicative function such that
\[
E^B_{1_\chi}(Z_1, \ldots, Z_n) := E(Z_1 \cdots Z_n)
\]
for each $\chi : \{1, \ldots, n\} \to \{\ell, r\}$ and $Z_k \in \mathcal{A}_{\chi(k)}$.  Great lengths were taken in \cite{CNS2015-2}*{Section 5} in order to show that such a bi-multiplicative function exists.

\begin{defn}
\label{defn:kappa}
Let $(\mathcal{A}, E, \varepsilon)$ be a $B$-$B$-non-commutative probability space.  The \emph{operator-valued bi-free cumulant function} 
\[
\kappa^B : \bigcup_{n\geq 1} \bigcup_{\chi : \{1,\ldots, n\} \to \{\ell, r\}} BNC(\chi) \times \mathcal{A}_{\chi(1)} \times \cdots \times \mathcal{A}_{\chi(n)} \to B
\]
is defined by
\begin{align}
\label{eq:cumulants}
\kappa^B_\pi(Z_1,\ldots, Z_n) := \sum_{\substack{\sigma \in BNC(\chi) \\ \sigma \leq\pi}}  E^B_\sigma(Z_1, \ldots, Z_n) \mu_{BNC}(\sigma, \pi)
\end{align}
for each $\chi : \{1, \ldots, n\} \to \{\ell, r\}$, $\pi \in BNC(\chi)$, and $Z_k \in \mathcal{A}_{\chi(k)}$.  To simplify notation, define
\[
\kappa^B_\chi(Z_1,\ldots, Z_n) := \kappa^B_{1_\chi}(Z_1,\ldots, Z_n)
\]
\end{defn}

Since $E^B$ is bi-multiplicative, one obtains that $\kappa^B$ is bi-multiplicative by properties of the bi-non-crossing M\"{o}bius function (see \cite{CNS2015-2}*{Theorem 6.2.1}).  Furthermore, using M\"{o}bius inversion, obtain obtains that
\begin{align}
\label{eq:mobius}
E^B_\sigma(Z_1, \ldots, Z_n) &= \sum_{\substack{\pi \in BNC(\chi) \\ \pi \leq\sigma}}\kappa^B_\pi(Z_1,\ldots, Z_n).
\end{align}

Finally, the following result shows that, as with operator-valued free cumulants, when one inserts an element of $\varepsilon(B \otimes 1_{B})$ or $\varepsilon(1_B \otimes B^{\op})$ into a non-trivial operator-valued bi-free cumulation, one always obtains zero.
\begin{prop}[\cite{CNS2015-2}*{Proposition 6.4.1}]
\label{prop:vanishing-of-scalar-cumulants}
Let $(\mathcal{A}, E, \varepsilon)$ be a $B$-$B$-non-commutative probability space, let $\chi : \{1,\ldots, n\} \to \{\ell, r\}$ with $n \geq 2$, and let $Z_k \in \mathcal{A}_{\chi(k)}$.  If there exist $q \in \{1,\ldots, n\}$ and $b \in B$ such that $Z_q = L_b$ if $\chi(q) = \ell$ or $Z_q = R_b$ if $\chi(q) = r$, then
\[
\kappa^B_\chi(Z_1, \ldots, Z_n) = 0.
\]
\end{prop}

\subsection*{Operator-Valued Bi-Freeness}  We are now in a position to discuss the main concept of this paper.

\begin{defn}
\label{defn:pair-of-B-faces}
Let $(\mathcal{A}, E_\mathcal{A}, \varepsilon)$ be a $B$-$B$-non-commutative probability space.  A \emph{pair of $B$-faces of} $\A$ is a pair $(A_\ell, A_r)$ of unital subalgebras of $\A$ such that
\[
\varepsilon(B \otimes 1_B) \subseteq A_\ell \subseteq \A_\ell \qquad \mathrm{and}\qquad \varepsilon(1_B \otimes B^{\mathrm{op}}) \subseteq A_r \subseteq \A_r.
\]
\end{defn}

The following is an equivalent definition of when a family of $B$-faces is bi-free with amalgamation over $B$ from \cite{CNS2015-2}.  Note that the map $\epsilon : \{1,\ldots, n\} \to K$ defines a partition on $\{1,\ldots, n\}$ with blocks $\{ \epsilon^{-1}(\{k\})\}_{k \in K}$, so it makes sense to ask if a partition is a refinement of $\epsilon$.

\begin{thm}[\cite{CNS2015-2}*{Theorem 7.1.4 and Theorem 8.1.1}]
\label{thm:bifree-classifying-theorem}
Let $(\mathcal{A}, E_\mathcal{A}, \varepsilon)$ be a $B$-$B$-non-commutative probability space and let $\{(A_{k,\ell}, A_{k,r})\}_{k \in K}$ be a family of pairs of $B$-faces of $\A$.  Then $\{(A_{k,\ell}, A_{k,r})\}_{k \in K}$ are bi-free with amalgamation over $B$ if and only if for all $\chi : \{1,\ldots, n\} \to \{\ell, r\}$, $\epsilon : \{1,\ldots, n\} \to K$, and $Z_k \in A_{\epsilon(k), \chi(k)}$, the formula
\begin{align*}
E^B_{\A}(Z_1 \cdots Z_n) = \sum_{\pi \in BNC(\chi)} \left[ \sum_{\substack{\sigma\in BNC(\chi)\\\pi\leq\sigma\leq\epsilon}}\mu_{BNC}(\pi, \sigma) \right] E^B_{\pi}(Z_1,\ldots, Z_n)   
\end{align*}
holds.  Equivalently $\{(A_{k,\ell}, A_{k,r})\}_{k \in K}$ are bi-free with amalgamation over $B$ if and only if for all $\chi : \{1,\ldots, n\} \to \{\ell, r\}$, $\epsilon : \{1,\ldots, n\} \to K$, and $Z_k \in A_{\epsilon(k), \chi(k)}$, 
\[
\kappa^B_{\chi}(Z_1, \ldots, Z_n) = 0
\]
provided $\epsilon$ is not constant.
\end{thm}

\subsection*{Operator-Valued Bi-Free Cumulants of Products}
One of the most important tools on the combinatorial side of bi-free probability is the ability to expand bi-free cumulants of products of left operators and of products of right operators as cumulants in the individual operators.  To do so, we recall a result from \cite{CNS2015-2}.  To begin, given two partitions $\pi, \sigma \in BNC(\chi)$, let $\pi \vee \sigma$ denote the smallest element of $BNC(\chi)$ greater than $\pi$ and $\sigma$.

Let $m,n \in \mathbb{N}$ with $m < n$ and fix a sequence of integers 
\[
k(0) = 0 < k(1) < \cdots < k(m) = n.
\]
For $\chi : \{1,\ldots, m\} \to \{\ell, r\}$, define $\widehat{\chi} : \{1,\ldots, n\} \to \{\ell, r\}$ via
\[
\widehat{\chi}(q) = \chi(p_q)
\]
where $p_q$ is the unique element of $\{1,\ldots, m\}$ such that $k(p_q-1) < q \leq k(p_q)$.  Note there exists an embedding of $BNC(\chi)$ into $BNC(\widehat{\chi})$ via $\pi \mapsto \widehat{\pi}$ where the $p^{\mathrm{th}}$ node of $\pi$ is replaced by the block $(k(p-1)+1, \ldots, k(p))$.  Alternatively, this map can be viewed as an analogue of the map on non-crossing partitions from \cite{NSBook}*{Notation 11.9} after applying $s^{-1}_\chi$.

Using the above notation, the following bi-free analogue of \cite{NSBook}*{Theorem 11.12} was obtained.
\begin{thm}[\cite{CNS2015-2}*{Theorem 9.1.5}]
\label{thm:products}
Let $(\mathcal{A}, E, \varepsilon)$ be a $B$-$B$-non-commutative probability space, let $m,n \in \mathbb{N}$ with $m < n$, let $\chi : \{1,\ldots, m\} \to \{\ell, r\}$, and let
\[
k(0) = 0 < k(1) < \cdots < k(m) = n.
\]
If $Z_k \in \mathcal{A}_{\widehat{\chi}(k)}$ then
\[
\kappa^B_{\chi}\left(Z_1 \cdots Z_{k(1)}, Z_{k(1)+1} \cdots Z_{k(2)}, \ldots, Z_{k(m-1)+1} \cdots Z_{k(m)}\right) = \sum_{\substack{\sigma \in BNC(\widehat{\chi})\\ \sigma \vee \widehat{0_\chi} = 1_{\widehat{\chi}}}} \kappa^B_\sigma(Z_1, \ldots, Z_n).
\]
\end{thm}

Theorem \ref{thm:products} has many applications in this paper.  One application is the following result, which may also be deduced from the proof of Theorem \ref{thm:bifree-classifying-theorem}.
\begin{cor}
\label{cor:generators-vanishing-cumulants}
Let $(\A, E, \varepsilon)$ be a $B$-$B$-non-commutative probability space and let $\{(A_{k, \ell}, A_{k, r})\}_{k \in K}$ be pairs of $B$-faces.  For each $k \in K$, suppose there exist elements $\{Z_{k,i}\}_{i \in I_k} \in A_{k,\ell}$ and $\{Z_{k,j}\}_{j \in J_k} \subseteq A_{k,r}$ such that
\[
A_{k,\ell} = \alg(\{Z_{k,i}\}_{i \in I}, \varepsilon(B \otimes 1_B)) \qqand A_{k,r} = \alg(\{Z_{k,j}\}_{j \in J},  \varepsilon(1_B \otimes B^\op)).
\]
Then $\{(A_{k, \ell}, A_{k, r})\}_{k \in K}$ are bi-free with amalgamation over $B$ if and only if for all $\chi : \{1,\ldots, n\} \to \{\ell, r\}$, $\epsilon : \{1,\ldots, n\} \to K$, and 
\[
Z_k \in \left\{
\begin{array}{ll}
\{L_{b_1} Z_{\epsilon(k),i} L_{b_2} \, \mid \, i \in I_{\epsilon(k)}, b_1, b_2 \in B\} & \mbox{if } \chi(k) = \ell \\
\{R_{b_1} Z_{\epsilon(k),j} R_{b_2} \, \mid \, j \in J_{\epsilon(k)}, b_1, b_2 \in B\} & \mbox{if } \chi(k) = r
\end{array} \right.,
\]
we have
\[
\kappa^B_{\chi}(Z_1, \ldots, Z_n) = 0
\]
provided $\epsilon$ is not constant (i.e. checking that mixed bi-free cumulants vanish on the generators is enough).
\end{cor}

\subsection*{Operations on the Operator-Valued Bi-Free Cumulants}

In \cite{CNS2015-2}*{Theorem 10.2.1} it was demonstrated for families of pairs of $B$-faces with certain conditions that it suffices to show that certain algebras are freely independent with amalgamation over $B$ in order to show the family is bi-freely independent.  The following two results are quantitative realizations of these arguments.  In particular, these arguments can be used to show that, under certain conditions, only certain of the operator-valued bi-free cumulants need be considered.

\begin{lem}
\label{lem:interchange-cumulant}
Let $(\A, E, \varepsilon)$ be a $B$-$B$-non-commutative probability space, let $\chi : \{1,\ldots, n\} \to \{\ell, r\}$ be such that $\chi(k_0) = \ell$ and $\chi(k_0+1) = r$ for some $k_0 \in \{1,\ldots, n-1\}$, and let $X \in \A_\ell$ and $Y \in \A_r$ be such that $E(ZXYZ') = E(ZYXZ')$ for all $Z, Z' \in \A$.  Define $\chi' : \{1,\ldots, n\} \to \{\ell, r\}$ by 
\[
\chi'(k) = \left\{
\begin{array}{ll}
r & \mbox{if } k = k_0 \\
\ell & \mbox{if } k = k_0+1 \\
\chi(k) & \mbox{otherwise}
\end{array} \right. .
\]
Then 
\[
\kappa^B_{\chi}(Z_1, \ldots, Z_{k_0-1}, X, Y, Z_{k_0+2}, \ldots, Z_n) = \kappa^B_{\chi'}(Z_1, \ldots, Z_{k_0-1}, Y, X, Z_{k_0+2}, \ldots, Z_n)
\]
for all $Z_1, \ldots, Z_{k_0 - 1}, Z_{k_0+2}, \ldots, Z_n \in \A$ with $Z_k \in \A_{\chi(k)}$.
\end{lem}
\begin{proof}
By Definition \ref{defn:kappa} we obtain that
\begin{align*}
\kappa^B_{\chi}(Z_1, \ldots, Z_{k_0-1}, X, Y, Z_{k_0+2}, \ldots, Z_n) &= \sum_{\sigma \in BNC(\chi)}  E^B_\sigma(Z_1, \ldots, Z_{k_0-1}, X, Y, Z_{k_0+2}, \ldots, Z_n) \mu_{BNC}(\sigma, 1_\chi) \\
\kappa^B_{\chi'}(Z_1, \ldots, Z_{k_0-1}, Y, X, Z_{k_0+2}, \ldots, Z_n) &= \sum_{\sigma' \in BNC(\chi')}  E^B_{\sigma'}(Z_1, \ldots, Z_{k_0-1}, Y, X, Z_{k_0+2}, \ldots, Z_n) \mu_{BNC}(\sigma', 1_{\chi'}).
\end{align*}
Note that there is a bijection from $BNC(\chi)$ to $BNC(\chi')$ which sends a partition $\pi$ to the partition $\pi'$ obtained by interchanging $k_0$ and $k_0+1$.  Furthermore $\mu_{BNC}(\sigma, 1_\chi) = \mu_{BNC}(\sigma', 1_{\chi'})$ by the lattice structure on bi-non-crossing partitions.

To complete the proof, it suffices to show that 
\[
E^B_\sigma(Z_1, \ldots, Z_{k_0-1}, X, Y, Z_{k_0+2}, \ldots, Z_n) = E^B_{\sigma'}(Z_1, \ldots, Z_{k_0-1}, Y, X, Z_{k_0+2}, \ldots, Z_n) 
\]
for all $\sigma \in BNC(\chi)$.  If $k_0 \sim_\sigma k_0+1$, then, using bi-multiplicative properties, one may reduce 
\[
E^B_\sigma(Z_1, \ldots, Z_{k_0-1}, X, Y, Z_{k_0+2}, \ldots, Z_n)
\]
to an expression involving $E(ZXYZ')$ for some $Z, Z' \in \A$, commute $X$ and $Y$ to get $E(ZYXZ')$ in this reduction, and undo the bi-multiplicative reductions to obtain 
\[
E^B_{\sigma'}(Z_1, \ldots, Z_{k_0-1}, Y, X, Z_{k_0+2}, \ldots, Z_n).
\]
 If $k_0 \nsim_\sigma k_0+1$, then the bi-multiplicative reductions of 
 \[
 E^B_\sigma(Z_1, \ldots, Z_{k_0-1}, X, Y, Z_{k_0+2}, \ldots, Z_n) \qqand E^B_{\sigma'}(Z_1, \ldots, Z_{k_0-1}, Y, X, Z_{k_0+2}, \ldots, Z_n) 
 \]
 agree via Definition \ref{defn:bi-multiplicative}.  Consequently the proof is complete.
\end{proof}
\begin{lem}
\label{lem:swap-cumulant}
Let $(\A, E, \varepsilon)$ be a $B$-$B$-non-commutative probability space, let $\chi : \{1,\ldots, n\} \to \{\ell, r\}$ be such that $\chi(n) = \ell$, and let $X \in \A_\ell$ and $Y \in \A_r$ be such that $E(ZX) = E(ZY)$ for all $Z \in \A$.  Define $\chi' : \{1,\ldots, n\} \to \{\ell, r\}$ by 
\[
\chi'(k) = \left\{
\begin{array}{ll}
r & \mbox{if } k = n \\
\chi(k) & \mbox{otherwise}
\end{array} \right. .
\]
Then 
\[
\kappa^B_{\chi}(Z_1, \ldots, Z_{n-1}, X) = \kappa^B_{1_{\chi'}}(Z_1, \ldots, Z_{n-1}, Y)
\]
for all $Z_1, \ldots, Z_{n - 1} \in \A$ with $Z_k \in \A_{\chi(k)}$.
\end{lem}
\begin{proof}
By Definition \ref{defn:kappa}, we obtain that
\begin{align*}
\kappa^B_{\chi}(Z_1, \ldots, Z_{n-1}, X)&= \sum_{\sigma \in BNC(\chi)}  E^B_\sigma(Z_1, \ldots, Z_{n-1}, X) \mu_{BNC}(\sigma, 1_\chi) \\
\kappa^B_{\chi'}(Z_1, \ldots, Z_{n-1}, Y) &= \sum_{\sigma' \in BNC(\chi')}  E^B_{\sigma'}(Z_1, \ldots, Z_{n-1}, Y) \mu_{BNC}(\sigma', 1_{\chi'}).
\end{align*}
Note that there is a bijection from $BNC(\chi)$ to $BNC(\chi')$ which sends a partition $\pi$ to the partition $\pi'$ obtained by changing the last node from a left node to a right node.  Furthermore $\mu_{BNC}(\sigma, 1_\chi) = \mu_{BNC}(\sigma', 1_{\chi'})$ by the lattice structure on bi-non-crossing partitions.  Since Definition \ref{defn:bi-multiplicative} together with the assumptions on $X$ and $Y$ implies
\[
E^B_\sigma(Z_1, \ldots, Z_{n-1}, X)= E^B_{\sigma'}(Z_1, \ldots, Z_{n-1}, Y)
\]
for all $\sigma \in BNC(\chi)$, the proof is complete.
\end{proof}

\subsection*{Operator-Valued Bi-Free Distributions}

With the above preliminaries, we turn our attention to the main focus of this paper; operator-valued bi-free distributions.  Below it is exhibited which operator-valued bi-free moments/cumulants are required to understand the operator-valued bi-free distributions.  Many of the arguments are similar to those in the introduction via bi-multiplicativity from Definition \ref{defn:bi-multiplicative}.

Let $(\A, E, \varepsilon)$ be a $B$-$B$-non-commutative probability space and let $(A_\ell, A_r)$ be a pair of $B$-faces such that
\[
A_{\ell} = \alg(\{Z_{i}\}_{i \in I}, \varepsilon(B \otimes 1_B)) \qqand A_{r} = \alg(\{Z_{j}\}_{j \in J}, \varepsilon(1_B \otimes B^\op))
\]
for some $\{Z_i\}_{i \in I} \subseteq \A_\ell$ and $\{Z_j\}_{j \in J} \subseteq \A_r$.
Using the facts that $b \mapsto L_b$ and $b \mapsto R_b$ are a homomorphism and anti-homomorphism that commute, one obtains that  every element of $\alg(A_\ell, A_r)$ can be written as a linear combination of elements of the form
\[
C_{b_1}Z_{\omega(1)}C_{b_2} Z_{\omega(2)} \cdots C_{b_n} Z_{\omega(n)} L_{b} R_{b'} \quad \text{ where } \quad C_{b_k} = \left\{
\begin{array}{ll}
L_{b_k} & \mbox{if } \omega(k) \in I \\
R_{b_k} & \mbox{if } \omega(k) \in J
\end{array} \right. 
\]
for some $n \geq 0$ and $\omega : \{1,\ldots, n\} \to I \sqcup J$.  

Note Definition \ref{defn:BBncps} implies that
\[
E(C_{b_1}Z_{\omega(1)} \cdots C_{b_n} Z_{\omega(n)} L_{b} R_{b'})  = E(C_{b_1}Z_{\omega(1)} \cdots C_{b_n} Z_{\omega(n)} L_{b b'}) = E(C_{b_1}Z_{\omega(1)} \cdots C_{b_n} Z_{\omega(n)} R_{b b'})
\]
Moreover if $\omega(k) \in I$ for all $k$ ($\omega(k) \in J$ for all $k$), then the $L_{b} R_{b'}$ term may be removed from the product if one multiplies on the right (respectively left) of the remaining expectation by $bb'$.  However, if $\{\omega(k)\}^n_{k=1}$ intersects both $I$ and $J$, there need not be a way to remove $L_{b} R_{b'}$ from the expectation.

Furthermore, if $k_\ell = \min\{k \, \mid \, \omega(k) \in I\} \neq \infty$, then $C_{b_{k_\ell}}$ can be removed from the product if one multiplies by $b_{k_\ell}$ on the left of the remaining expectation.  Similarly if $k_r = \min\{k \, \mid \, \omega(k) \in  J\} \neq \infty$, then $C_{b_{k_r}}$ can be removed from the product if one multiplies by $b_{k_r}$ on the right of the remaining expectation.

Consequently, the joint distributions elements of $\alg(A_\ell, A_r)$ can be deduced by moments of a very specific form.   Similarly, using bi-multiplicative properties, only certain operator-valued bi-free cumulants are required to study the joint distributions of elements in $\alg(A_\ell, A_r)$.  We will make the following notation to describe the necessary operator-valued bi-free moments and cumulants:
\begin{nota}
For $Z = \{Z_i\}_{i \in I} \sqcup \{Z_j\}_{j \in J}$ with $Z_i$ and $Z_j$ as above, $n \geq 1$, $\omega : \{1,\ldots, n\} \to I \sqcup J$, and $b_1, \ldots, b_{n-1} \in B$, define the following:
\begin{itemize}
\item If $\omega(k)\in I$ for all $k$, define
\begin{align*}
\mu^B_{Z, \omega}(b_1, \ldots, b_{n-1}) &= E(Z_{\omega(1)} L_{b_1} Z_{\omega(2)} L_{b_2} Z_{\omega(3)} \cdots L_{b_{n-1}} Z_{\omega(n)}) \\
\kappa^B_{Z, \omega}(b_1, \ldots, b_{n-1}) &= \kappa^B_{\chi_\omega}(Z_{\omega(1)}, L_{b_1} Z_{\omega(2)}, L_{b_2} Z_{\omega(3)}, \ldots, L_{b_{n-1}} Z_{\omega(n)}).
\end{align*}
\item If $\omega(k)\in J$ for all $k$, define
\begin{align*}
\mu^B_{Z, \omega}(b_1, \ldots, b_{n-1}) &= E(Z_{\omega(1)} R_{b_1} Z_{\omega(2)} R_{b_2} Z_{\omega(3)} \cdots R_{b_{n-1}} Z_{\omega(n)})\\
\kappa^B_{Z, \omega}(b_1, \ldots, b_{n-1}) &= \kappa^B_{\chi_\omega}(Z_{\omega(1)}, R_{b_1} Z_{\omega(2)}, R_{b_2} Z_{\omega(3)}, \ldots, R_{b_{n-1}} Z_{\omega(n)}).
\end{align*}
\item Otherwise let $k_\ell = \min\{k \, \mid \, \omega(k) \in I\}$ and $k_r = \min\{k \, \mid \, \omega(k) \in J\}$.  Then $\{k_\ell, k_r\} = \{1, k_0\}$ for some $k_0$.  Define $\mu^B_{Z, \omega}(b_1, \ldots, b_{n-1})$ to be
\[
E\left(Z_{\omega(1)} C^{\omega(2)}_{b_1} Z_{\omega(2)} \cdots C^{\omega(k_0-1)}_{b_{k_0-2}} Z_{\omega(k_0-1)} Z_{k_0} C^{\omega(k_0+1)}_{b_{k_0-1}} Z_{\omega(k_0+1)} \cdots C^{\omega(n-1)}_{b_{n-3}} Z_{\omega(n-1)} C^{\omega(n)}_{b_{n-2}} Z_{\omega(n)} C^{\omega(n)}_{b_{n-1}}\right)
\]
and define $\kappa^B_{Z, \omega}(b_1, \ldots, b_{n-1})$ to be
\[
\kappa^B_{\chi_\omega}\left(Z_{\omega(1)}, C^{\omega(2)}_{b_1} Z_{\omega(2)}, \ldots, C^{\omega(k_0-1)}_{b_{k_0-2}} Z_{\omega(k_0-1)}, Z_{k_0}, C^{\omega(k_0+1)}_{b_{k_0-1}} Z_{\omega(k_0+1)}, \ldots, C^{\omega(n-1)}_{b_{n-3}} Z_{\omega(n-1)}, C^{\omega(n)}_{b_{n-2}} Z_{\omega(n)} C^{\omega(n)}_{b_{n-1}}\right)
\]
where
\[
C^{\omega(k)}_{b} = \left\{
\begin{array}{ll}
L_{b} & \mbox{if } \omega(k) \in I \\
R_{b} & \mbox{if } \omega(k) \in J
\end{array} \right.  .
\]
\end{itemize}
\end{nota}
Note that one may think of the above as placing exactly one $B$-operator between every pair of $\chi$-adjacent terms.  Furthermore, if $\omega : \{1,\ldots, n\} \to I \sqcup J$ is such that $\{\omega(k)\}^n_{k=1}$ intersects both $I$ and $J$, then the $b_{n-1}$-term behaves differently than the other elements of $B$ in $\kappa^B_{Z, \omega}(b_1, \ldots, b_{n-1})$ (i.e. it is the special $b$-operator that can be viewed as a left and as a right operator).  Moreover, it is clear from the above discussions that 
\[
\{\mu^B_{Z, \omega} \, \mid \, n \geq 1, \omega : \{1,\ldots, n\} \to I \sqcup J\}
\]
completely describe the joint moments of elements in $\alg(A_\ell, A_r)$.  Using equation (\ref{eq:mobius}) together with bi-multiplicative properties, one sees that 
\[
\{\kappa^B_{Z, \omega} \, \mid \, n \geq 1, \omega : \{1,\ldots, n\} \to I \sqcup J\}
\]
also completely describe the joint moments of elements in $\alg(A_\ell, A_r)$.

\section{A Characterization of Bi-Freeness over Subalgebras}
\label{sec:sub}

Let $B$ be a unital algebra and let $(\A, E, \varepsilon)$ be a $B$-$B$-non-commutative probability space.  Suppose $D \subseteq B$ is a unital subalgebra (with $1_B \in D$) such that there exists a conditional expectation $F : B \to D$ of $B$ onto $D$ (that is, $F(d) = d$ for all $d \in D$ and $F(d_1 b d_2) = d_1 F(b) d_2$ for all $d_1, d_2 \in D$ and $b \in B$).  We claim that $(\A, F \circ E, \varepsilon|_{D \otimes D^{\op}})$ is a $D$-$D$-non-commutative probability space.  Indeed $\varepsilon|_{D \otimes D^{\op}}$ is a unital homomorphism that is injective on $D \otimes 1_D$ and $1_D \otimes D^\op$, 
\[
(F \circ E)(\varepsilon(d_1 \otimes d_2)Z) = F(d_1 E(Z) d_2) = d_1 F(E(Z)) d_2
\]
for all $d_1, d_2 \in D$ and $Z \in \A$, and
\[
(F \circ E)(Z\varepsilon(d \otimes 1_D)) = F(E(Z\varepsilon(d \otimes 1_D))) =  F(E(Z\varepsilon(1_D \otimes d))) = (F \circ E)(Z\varepsilon(1_D \otimes d))
\]
for all $d \in D$ and $Z \in \A$.  Hence $(\A, F \circ E, \varepsilon|_{D \otimes D^{\op}})$ is a $D$-$D$-non-commutative probability space by Definition \ref{defn:BBncps}.

The goal of this section is to investigate the relationship between the $B$-valued and $D$-valued distributions of elements of $\A$ in the context of the previous paragraph.  Consequently, the notation in the previous paragraph will be used throughout the section.  Furthermore, $b$ will be used for an element of $B$, $d$ will be used for an element of $D$, and $\kappa^B$ and $\kappa^D$ will be used for the $B$-valued and $D$-valued cumulants respectively.

We begin with the following bi-free analogue of \cite{NSS2002}*{Theorem 3.1}.
\begin{thm}
\label{thm:cumulants-D-valued}
Let $(\A, E, \varepsilon)$ be a $B$-$B$-non-commutative probability space, let $\{Z_i\}_{i \in I} \subseteq \A_\ell$, and let $\{Z_j\}_{j \in J} \subseteq \A_r$.  If the $B$-valued cumulants of $Z = \{Z_i\}_{i\in I} \sqcup \{Z_j\}_{j \in J}$ are such that
\[
\kappa^B_{Z, \omega}(d_1, \ldots, d_{n-1}) \in D
\]
for all $n\geq 1$, $\omega : \{1,\ldots, n\} \to I \sqcup J$, and $d_1, \ldots, d_{n-1} \in D$, then the $D$-valued cumulants of $Z$ are the restriction to $D$ of the $B$-valued cumulants of $Z$; that is
\[
\kappa^D_{Z, \omega}(d_1, \ldots, d_{n-1}) = \kappa^B_{Z, \omega}(d_1, \ldots, d_{n-1})
\]
for all $n\geq 1$, $\omega : \{1,\ldots, n\} \to I \sqcup J$, and $d_1, \ldots, d_{n-1} \in D$.
\end{thm}
\begin{proof}
Let $\A_0$ be the subalgebra of $\A$ generated by $Z$ and $\varepsilon(D \otimes D^{\op})$.  Using bi-multiplicative properties, Theorem \ref{thm:products}, and the assumptions of the theorem, we obtain that the restriction of the $B$-valued cumulant function to $\A_0$ is valued in $D$.  Consequently, the moment-cumulant formula (\ref{eq:mobius}) implies that the $B$-valued moment function is $D$-valued when restricted to $\A_0$.  Thus the $B$-valued and $D$-valued moment functions agree when restricted to $\A_0$.  Hence the moment-cumulant formula (\ref{eq:cumulants}) implies that the $B$-valued and $D$-valued cumulant functions agree on $\A_0$.
\end{proof}
In \cite{NSS2002}, it was demonstrated that the sufficient condition in Theorem \ref{thm:cumulants-D-valued} is almost necessary in the free case.

The main goal of this section is to prove the following bi-free analogue of the beautiful result \cite{NSS2002}*{Theorem 3.5} to completely characterizes when a pair of $D$-faces is bi-free from $(B, B^\op)$ with amalgamation over $D$ via a relationship between the $B$-valued and $D$-valued bi-free cumulants.  The power of the following theorem in obtaining bi-free results and examples is exhibited in Section \ref{sec:R-cyc}.
\begin{thm}
\label{thm:bi-free-over-D}
Let $\{Z_i\}_{i \in I} \subseteq \A_\ell$, $\{Z_j\}_{j \in J} \subseteq \A_r$, and $Z = \{Z_i\}_{i \in I} \sqcup \{Z_j\}_{j \in J}$.   Assume that $F : B \to D$ satisfies the following faithfulness condition:
\begin{itemize}
\item if $b_1 \in B$ and $F(b_2b_1) = 0$ for all $b_2 \in B$, then $b_1 = 0$.
\end{itemize}
Then $(\alg(\varepsilon(D \otimes 1_D), \{Z_i\}_{i \in I}), \alg(\varepsilon(1_D \otimes D^\op), \{Z_j\}_{j \in J}))$ is bi-free from $(\varepsilon(B \otimes 1_B), \varepsilon(1_B \otimes B^\op))$ with amalgamation over $D$ if and only if
\begin{align}
\label{eq:cummulant-condition-1}
\kappa^B_{Z, \omega}(b_1, \ldots, b_{n-1}) = F\left( \kappa^B_{Z, \omega}(F(b_1), \ldots, F(b_{n-1})) \right)
\end{align}
for all $n\geq 1$, $\omega : \{1,\ldots, n\} \to I \sqcup J$, and $b_1, \ldots, b_{n-1} \in B$.  Alternatively, equation (\ref{eq:cummulant-condition-1}) is equivalent to
\begin{align}
\label{eq:cummulant-condition-0}
\kappa^B_{Z, \omega}(b_1, \ldots, b_{n-1}) = \kappa^D_{Z, \omega}(F(b_1), \ldots, F(b_{n-1})).
\end{align}
\end{thm}

Note the equivalence of equations (\ref{eq:cummulant-condition-1}) and (\ref{eq:cummulant-condition-0}) follows directly from Theorem \ref{thm:cumulants-D-valued}.

\begin{rem}
\label{rem:ex-D-valued}
If $Z = \{Z_i\}_{i \in I} \sqcup \{Z_j\}_{j \in J}$ satisfy equation (\ref{eq:cummulant-condition-0}), then the proof of Theorem \ref{thm:cumulants-D-valued} implies that $E$ is $D$-valued on $\alg(Z, \varepsilon(D \otimes D^\op ))$.
\end{rem}

\begin{rem}
\label{rem:property-preserved-products-by-D}
Suppose equation (\ref{eq:cummulant-condition-0}) holds for some $\{Z_i\}_{i \in I} \subseteq \A_\ell$, $\{Z_j\}_{j \in J} \subseteq \A_r$, and $Z = \{Z_i\}_{i \in I} \sqcup \{Z_j\}_{j \in J}$.  For each $i \in I$, $j \in J$, and $k \in \{1,2\}$  let $d_{i,k}, d_{j,k} \in D$, let $Z'_i = L_{d_{i,1}} Z_i L_{d_{i,2}}$, and let $Z'_j = R_{d_{j,1}} Z_j R_{d_{j,2}}$.  If $Z' = \{Z'_i\}_{i \in I} \sqcup \{Z'_j\}_{j \in J}$, then 
\[
\kappa^B_{Z', \omega}(b_1, \ldots, b_{n-1}) = \kappa^D_{Z', \omega}(F(b_1), \ldots, F(b_{n-1})).
\]
for all $n\geq 1$, $\omega : \{1,\ldots, n\} \to I \sqcup J$, and $b_1, \ldots, b_{n-1} \in B$; that is, $Z'$ also satisfies equation (\ref{eq:cummulant-condition-0}).  Indeed, this is easily verified using properties (\ref{part:bi-multi-1}) and (\ref{part:bi-multi-2}) of Definition \ref{defn:bi-multiplicative} along with the conditional expectation property of $F$.
\end{rem}

As described in the introduction, the proof of Theorem \ref{thm:bi-free-over-D} will be drastically different from the proof of \cite{NSS2002}*{Theorem 3.5}.  The approach taken here to show bi-freeness of $Z$ from $(B, B^\op)$ over $D$ is to demonstrate that mixed $D$-valued cumulants of $Z$ and $(B, B^\op)$ vanish in order to apply Theorem \ref{thm:bifree-classifying-theorem}.  This will be accomplished via induction arguments using the moment-cumulant formulae.  As the proof of Theorem \ref{thm:bi-free-over-D} ultimately must contain a proof of \cite{NSS2002}*{Theorem 3.5}, we will begin by reproving \cite{NSS2002}*{Theorem 3.5} (i.e. showing all mixed cumulants of only left operators from the pairs vanish).  The full proof of Theorem \ref{thm:bi-free-over-D} can then be seen using near identical combinatorial arguments.

In order to reprove \cite{NSS2002}*{Theorem 3.5} by showing all mixed cumulants of left operators vanish, we begin with the following base case of an inductive argument.  To simplify notation throughout the arguments in this section, define $Z(D, I) := \{ L_{d_1} Z_i L_{d_2} \, \mid \, i \in I, d_1, d_2 \in D\}$.
\begin{lem}
\label{lem:base-case-1}
Let $Z = \{Z_i\}_{i \in I} \sqcup \{Z_j\}_{j \in J}$ satisfy equation (\ref{eq:cummulant-condition-0}).  For all $n \geq 2$, for all $\chi : \{1,\ldots, n\} \to \{\ell, r\}$ such that $\chi(k) = \ell$ for all $k$ (i.e. $\chi = \chi_{n,0}$), and for all $X_1, \ldots, X_n \in \epsilon(B \otimes 1_B) \cup Z(D,I)$, we have
\[
\kappa^D_\chi(X_1, \ldots, X_n) = 0
\]
provided exactly one $X_k$ is an element of $\epsilon(B \otimes 1_B)$.
\end{lem}
\begin{proof}
The proof proceeds by induction on $n$.  In the case $n = 2$, note, using Remark \ref{rem:ex-D-valued}, that
\begin{align*}
\kappa^D_\chi(L_b, L_{d_1} Z_i L_{d_2}) &= F(E(L_bL_{d_1} Z_i L_{d_2})) - F(E(L_b))F(E(L_{d_1} Z_i L_{d_2})) \\
&= F(bE(L_{d_1} Z_i L_{d_2})) - F(b)E(L_{d_1} Z_i L_{d_2}) = F(b)E(L_{d_1} Z_i L_{d_2}) - F(b)E(L_{d_1} Z_i L_{d_2}) = 0
\end{align*}
and
\begin{align*}
\kappa^D_\chi(L_{d_1} Z_i L_{d_2}, L_b) &=F(E(L_{d_1} Z_i L_{d_2}L_b)) - F(E(L_{d_1} Z_i L_{d_2}))F(E(L_b)) \\
&= F(E(L_{d_1} Z_i L_{d_2})b) - E(L_{d_1} Z_i L_{d_2})F(b) = E(L_{d_1} Z_i L_{d_2})F(b) - E(L_{d_1} Z_i L_{d_2})F(b) = 0.
\end{align*}
Hence the base case is complete.

For the inductive step, suppose the result has been verified for $n \geq 2$.  Let $\{X_k\}^{n}_{k=1} \subseteq Z(D,I)$ and let $b \in B$.  The inductive step will now be broken into three cases depending on the position of $L_b$.  To begin, we will demonstrate that $\kappa^D_\chi(L_b, X_1, \ldots, X_n) = 0$.  Indeed note by the moment-cumulant formula (\ref{eq:mobius}) that
\begin{align*}
\kappa^D_\chi(L_b, X_1, \ldots, X_n) &= F(E(L_b X_1 \cdots X_n)) - \sum_{\substack{ \sigma \in BNC(\chi) \\ \sigma \neq 1_\chi}} \kappa^D_\sigma(L_b, X_1, \ldots, X_n).
\end{align*}
For the first term, note
\[
F(E(L_b X_1 \cdots X_n)) = F(b) E(X_1 \cdots X_n)
\]
by Remark \ref{rem:ex-D-valued}.  Furthermore, note if $\sigma \in BNC(\chi)$ is such that $\sigma \neq 1_\chi$ and $\{1\}$ is not a block of $\sigma$, then, using the bi-multiplicative properties of the bi-free cumulants to reduce to the block of $\sigma$ containing $1$ (note each $X_k$ may become $L_d X_k L_{d'}$ where $d, d' \in D$ when using these reduction properties), $\kappa^D_\sigma(L_b, X_1, \ldots, X_n) = 0$ by the inductive hypothesis.  Contemplating the following diagram may aid in the comprehension of how this argument works.
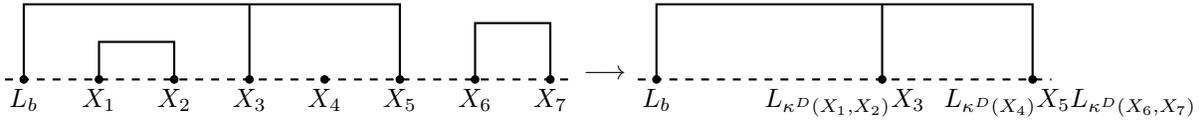
\begin{align*}
\begin{tikzpicture}[baseline]
	\draw[thick, dashed] (-.25,0) -- (7.25, 0);
	\draw[thick] (5, 0) -- (5,1) -- (0,1) -- (0, 0);
	\draw[thick] (3,0) -- (3,1);
	\draw[thick] (1, 0) -- (1,.5) -- (2,.5) -- (2, 0);
	\draw[thick] (6, 0) -- (6,.75) -- (7,.75) -- (7, 0);
	\node[below] at (0, 0) {$L_b$};
	\draw[black, fill=black] (0,0) circle (0.05);	
	\node[below] at (1, 0) {$X_1$};
	\draw[black, fill=black] (1,0) circle (0.05);	
	\node[below] at (2, 0) {$X_2$};
	\draw[black, fill=black] (2,0) circle (0.05);	
	\node[below] at (3, 0) {$X_3$};
	\draw[black, fill=black] (5,0) circle (0.05);
	\node[below] at (4, 0) {$X_4$};
	\draw[black, fill=black] (4,0) circle (0.05);	
	\node[below] at (5, 0) {$X_5$};
	\draw[black, fill=black] (3,0) circle (0.05);
	\node[below] at (6, 0) {$X_6$};
	\draw[black, fill=black] (6,0) circle (0.05);	
	\node[below] at (7, 0) {$X_7$};
	\draw[black, fill=black] (7,0) circle (0.05);
\end{tikzpicture}
\longrightarrow
\begin{tikzpicture}[baseline]
	\draw[thick, dashed] (-.25,0) -- (5.25, 0);
	\draw[thick] (5, 0) -- (5,1) -- (0,1) -- (0, 0);
	\draw[thick] (3,0) -- (3,1);
	\node[below] at (0, 0) {$L_b$};
	\draw[black, fill=black] (0,0) circle (0.05);	
	\node[below] at (2.5, 0) {$L_{\kappa^D(X_1, X_2)} X_3$};
	\draw[black, fill=black] (5,0) circle (0.05);	
	\node[below] at (5.5, 0) {$L_{\kappa^D(X_4)}X_5L_{\kappa^D(X_6, X_7)}$};
	\draw[black, fill=black] (3,0) circle (0.05);
\end{tikzpicture}
\end{align*}
Hence
\begin{align*}
\sum_{\substack{ \sigma \in BNC(\chi) \\ \sigma \neq 1_\chi}} \kappa^D_\sigma(L_b, X_1, \ldots, X_n) &= F(b) \sum_{\sigma' \in BNC(\chi|_{2, \ldots, n+1})} \kappa^D_{\sigma'} (X_1, \ldots, X_n) \\ &= F(b) F(E(X_1 \cdots X_n)) = F(b) E(X_1 \cdots X_n).
\end{align*}
Hence $\kappa^D_\chi(L_b, X_1, \ldots, X_n) = 0$.  

A similar argument show that $\kappa^D_\chi(X_1, \ldots, X_n, L_b) = 0$.  Finally, suppose $m \in \{1,\ldots, n-1\}$.  We desire to show that $\kappa^D_\chi (X_1, \ldots, X_m, L_b, X_{m+1}, \ldots, X_n)  = 0$.  Note, by Proposition \ref{prop:vanishing-of-scalar-cumulants}, 
\[
\kappa^D_\chi (X_1, \ldots, X_m, L_b, X_{m+1}, \ldots, X_n)  = 0 \Longleftrightarrow \kappa^D_\chi (X_1, \ldots, X_m, L_{b- F(b)}, X_{m+1}, \ldots, X_n)  = 0.
\]
Hence we may assume that $F(b) = 0$.  Furthermore by the moment-cumulant formula (\ref{eq:mobius}) 
\begin{align*}
\kappa^D_\chi & (X_1, \ldots, X_m, L_b, X_{m+1}, \ldots, X_n) \\
&= F(E(X_1 \cdots X_m L_b X_{m+1} \cdots X_n)) - \sum_{\substack{ \sigma \in BNC(\chi) \\ \sigma \neq 1_\chi}} \kappa^D_\sigma(X_1, \ldots, X_m, L_b, X_{m+1}, \ldots, X_n) \\
&= \sum_{\pi \in BNC(\chi)} F\left(\kappa^B_\pi(X_1, \ldots, X_m, L_b, X_{m+1}, \ldots, X_n)\right) - \sum_{\substack{ \sigma \in BNC(\chi) \\ \sigma \neq 1_\chi}} \kappa^D_\sigma(X_1, \ldots, X_m, L_b, X_{m+1}, \ldots, X_n).
\end{align*}
Note if $\pi \in BNC(\chi)$ is such that $\{m+1\}$ is not a block of $\pi$, then $\kappa^B_\pi(X_1, \ldots, X_m, L_b, X_{m+1}, \ldots, X_n) = 0$ by Proposition \ref{prop:vanishing-of-scalar-cumulants} (and bi-multiplicative properties).  Similarly, if $\sigma \in BNC(\chi)$ is such that $\{m+1\}$ is not a block of $\sigma$, then, using the bi-multiplicative properties of the bi-free cumulants, $\kappa^D_\sigma(X_1, \ldots, X_m, L_b, X_{m+1}, \ldots, X_n) = 0$ by the inductive hypothesis.   Hence
\begin{align*}
\kappa^D_\chi & (X_1, \ldots, X_m, L_b, X_{m+1}, \ldots, X_n) \\
&= \sum_{\substack{ \sigma \in BNC(\chi) \\ \{m+1\} \text{ a block of } \sigma }} F\left(\kappa^B_\sigma(X_1, \ldots, X_m, L_b, X_{m+1}, \ldots, X_n)\right) -  \kappa^D_\sigma(X_1, \ldots, X_m, L_b, X_{m+1}, \ldots, X_n).
\end{align*}
Using $0 = F(b) = \kappa^D(b)$, one obtains
\[
\sum_{\substack{ \sigma \in BNC(\chi) \\ \{m+1\} \text{ a block of } \sigma }} \kappa^D_\sigma(X_1, \ldots, X_m, L_b, X_{m+1}, \ldots, X_n) = 0
\]
as bi-multiplicative properties produce a $\kappa^D(b)$ in the reduction thereby yielding zero.  

To see that 
\begin{align}
\label{eq:sum-to-show-0}
\sum_{\substack{ \sigma \in BNC(\chi) \\ \{m+1\} \text{ a block of } \sigma }} F\left(\kappa^B_\sigma(X_1, \ldots, X_m, L_b, X_{m+1}, \ldots, X_n)\right)
\end{align}
is zero, note if $\sigma \in BNC(\chi)$ is such that $\{m+1\}$ is a block of $\sigma$ and there exist $k_1 \in \{1,\ldots, m\}$ and $k_2 \in \{m+2, \ldots, n+1\}$ with $k_1 \sim_\sigma  k_2$, then, after applying bi-multiplicative properties, one obtains $\kappa^B_{\chi'}(X'_1, \ldots, X'_q, L_{b} X'_{q+1}, \ldots, X'_p)$ where $1 \leq q < p$, $\chi' : \{1, \ldots, p\} \to \{\ell, r\}$ is such that $\chi'(k) = \ell$ for all $k$, and $X'_k \in Z(D, I)$. 
Contemplating the following diagrams may aid in the comprehension of this argument.
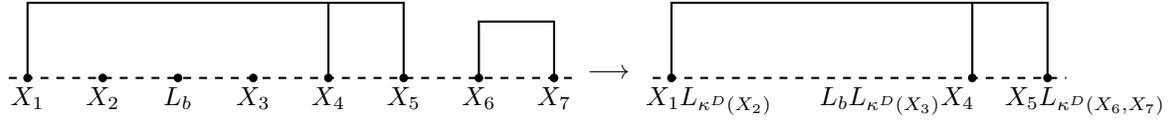
\begin{align*}
\begin{tikzpicture}[baseline]
	\draw[thick, dashed] (-.25,0) -- (7.25, 0);
	\draw[thick] (5, 0) -- (5,1) -- (0,1) -- (0, 0);
	\draw[thick] (4,0) -- (4,1);
	\draw[thick] (6, 0) -- (6,.75) -- (7,.75) -- (7, 0);
	\node[below] at (0, 0) {$X_1$};
	\draw[black, fill=black] (0,0) circle (0.05);	
	\node[below] at (1, 0) {$X_2$};
	\draw[black, fill=black] (1,0) circle (0.05);	
	\node[below] at (2, 0) {$L_b$};
	\draw[black, fill=black] (2,0) circle (0.05);	
	\node[below] at (3, 0) {$X_3$};
	\draw[black, fill=black] (5,0) circle (0.05);
	\node[below] at (4, 0) {$X_4$};
	\draw[black, fill=black] (4,0) circle (0.05);	
	\node[below] at (5, 0) {$X_5$};
	\draw[black, fill=black] (3,0) circle (0.05);
	\node[below] at (6, 0) {$X_6$};
	\draw[black, fill=black] (6,0) circle (0.05);	
	\node[below] at (7, 0) {$X_7$};
	\draw[black, fill=black] (7,0) circle (0.05);
\end{tikzpicture}
\longrightarrow
\begin{tikzpicture}[baseline]
	\draw[thick, dashed] (-.25,0) -- (5.25, 0);
	\draw[thick] (5, 0) -- (5,1) -- (0,1) -- (0, 0);
	\draw[thick] (4,0) -- (4,1);
	\node[below] at (0.5, 0) {$X_1L_{\kappa^D(X_2)}$};
	\draw[black, fill=black] (0,0) circle (0.05);	
	\node[below] at (3, 0) {$L_b L_{\kappa^D(X_3)} X_4$};
	\draw[black, fill=black] (5,0) circle (0.05);	
	\node[below] at (5.5, 0) {$X_5L_{\kappa^D(X_6, X_7)}$};
	\draw[black, fill=black] (4,0) circle (0.05);
\end{tikzpicture}
\end{align*}
Consequently, $\kappa^B_\sigma(X_1, \ldots, X_m, L_b, X_{m+1}, \ldots, X_n) = 0$ for such $\sigma$ since Remark \ref{rem:property-preserved-products-by-D} implies the $X'_k$ satisfy equation (\ref{eq:cummulant-condition-0}) so $F(b) = 0$ implies $\kappa^B_{\chi'}(X'_1, \ldots, X'_q, L_b X'_{q+1}, \ldots, X'_p) = 0$.  

Therefore, in (\ref{eq:sum-to-show-0}), one need only consider $\sigma \in BNC(\chi)$ with $\{m+1\}$ a block of $\sigma$ and $k_1 \nsim_\sigma k_2$ for all $k_1 \in \{1,\ldots, m\}$ and $k_2 \in \{m+2, \ldots, n+1\}$.  By the moment-cumulant formula (\ref{eq:mobius}), the sum (\ref{eq:sum-to-show-0}) then becomes
\[
F(E(X_1\cdots X_m) b E(X_{m+1} \cdots X_n)) = E(X_1 \cdots X_m)F(b) E(X_{m+1} \cdots X_n) = 0.
\]
Hence $\kappa^D_\chi (X_1, \ldots, X_m, L_b, X_{m+1}, \ldots, X_n) = 0$ thereby completing the inductive step of the proof.
\end{proof}

\begin{lem}
\label{lem:base-case-2}
Let $Z = \{Z_i\}_{i \in I} \sqcup \{Z_j\}_{j \in J}$ satisfy equation (\ref{eq:cummulant-condition-0}).  For all $n \geq 2$, for all $\chi : \{1,\ldots, n\} \to \{\ell, r\}$ such that $\chi(k) = \ell$ for all $k$, and for all $X_1, \ldots, X_n \in \epsilon(B \otimes 1_B) \cup Z(D, I)$, we have
\[
\kappa^D_\chi(X_1, \ldots, X_n) = 0
\]
provided at least one $X_k$ is an element of $\epsilon(B \otimes 1_B)$ and at least one $X_k$ is an element of $Z(D, I)$.
\end{lem}
\begin{proof}
Let $m$ be the number of $X_k$ that are elements of $\epsilon(B \otimes 1_B)$ (so $m \geq 1$).  We will proceed by induction on $m$ with the case $m = 1$ complete by Lemma \ref{lem:base-case-1}.

For the inductive step, suppose the result holds for some $m \geq 1$.  The inductive step will proceed by induction on $n-m \geq 1$.  Note we will demonstrate the base case and the inductive step simultaneously as the arguments are similar.

For notational purposes, let
\[
\theta(X_1, \ldots, X_n) = \{k \, \mid \, X_k \in \varepsilon(B \otimes 1_B)\}.
\]
Notice that
\begin{align*}
\kappa^D_\chi(X_1, \ldots, X_n) &= F(E(X_1 \cdots X_n)) - \sum_{\substack{ \sigma \in BNC(\chi) \\ \sigma \neq 1_\chi}} \kappa^D_\sigma(X_1, \ldots, X_n) \\
&= \sum_{\pi \in BNC(\chi)} F\left(\kappa^B_\pi(X_1, \ldots, X_n)\right) - \sum_{\substack{ \sigma \in BNC(\chi) \\ \sigma \neq 1_\chi}} \kappa^D_\sigma(X_1, \ldots, X_n).
\end{align*}

Note if $\pi \in BNC(\chi)$ then $\kappa^B_\pi(X_1, \ldots, X_n) = 0$ by Proposition \ref{prop:vanishing-of-scalar-cumulants} unless $\{k\}$ is a block of $\pi$ for each $k \in \theta(X_1, \ldots, X_n)$.  Furthermore, we claim that the induction hypotheses imply $\kappa^D_\sigma(X_1, \ldots, X_n) = 0$ whenever $\sigma \in BNC(\chi)$ is such that $\sigma \neq 1_\chi$ and there exist $k \in \theta(X_1, \ldots, X_n)$ and $k' \in \{1,\ldots, n\} \setminus\theta(X_1, \ldots, X_n)$ with $k \sim_\pi k'$.  Indeed, in the base case, $k'$ is the only index such that $X_{k'} \in Z(D, I)$, so, as $\sigma \neq 1_\chi$, the inductive hypothesis on $m$ (along with bi-multiplicative properties if necessary) imply $\kappa^D_\sigma(X_1, \ldots, X_n) = 0$ (see Lemma \ref{lem:S-lem-1} for examples of how to reduce using bi-multiplicative properties).  Otherwise, if $V$ is the block containing $k'$ and some element of $\theta(X_1, \ldots, X_n)$, then either the cardinality of $\theta(X_1, \ldots, X_n) \cap V$ is less than $m$ so the inductive hypothesis on $m$ (along with bi-multiplicative properties if necessary) imply $\kappa^D_\sigma(X_1, \ldots, X_n) = 0$, or the cardinality of $\theta(X_1, \ldots, X_n) \cap V$ is $m$ and there are less than $n-m$ elements of $\{1,\ldots, n\} \setminus \theta(X_1, \ldots, X_n)$ in $V$ so the inductive hypothesis on $n-m$ (along with bi-multiplicative properties if necessary) imply $\kappa^D_\sigma(X_1, \ldots, X_n) = 0$. 

 Consequently 
\begin{align*}
\kappa^D_\chi(X_1, \ldots, X_n) 
&= \sum_{\pi \in \Theta(X_1, \ldots, X_n)} F\left(\kappa^B_\pi(X_1, \ldots, X_n)\right) - \sum_{\sigma \in \Omega(X_1, \ldots, X_n)} \kappa^D_\sigma(X_1, \ldots, X_n).
\end{align*}
where
\begin{align*}
\Theta(X_1, \ldots, X_n) &:= \{\pi \in BNC(\chi) \, \mid \, \{k\} \text{ a block of }\chi \text{ for all } k \in \theta(X_1, \ldots, X_n) \} \quad \text{and}\\
\Omega(X_1, \ldots, X_n) &:= \{\sigma \in BNC(\chi) \, \mid \,  k \nsim_\sigma k' \text{ for all } k \in \theta(X_1, \ldots, X_n) \text{ and } k' \notin \theta(X_1, \ldots, X_n)\}.
\end{align*}
For all $\pi \in \Theta(X_1, \ldots, X_n)$, the partitions $\sigma \in \Omega(X_1, \ldots, X_n)$ such that 
\[
\sigma|_{\{1,\ldots, n\} \setminus  \theta(X_1, \ldots, X_n)} = \pi|_{\{1,\ldots, n\} \setminus \theta(X_1, \ldots, X_n)}
\]
will be called the partitions of $\Omega(X_1, \ldots, X_n)$ corresponding to $\pi$.

The proof will proceed to show 
\begin{align}
&\sum_{\pi \in \Theta(X_1, \ldots, X_n)} F\left(\kappa^B_\pi(X_1, \ldots, X_n)\right) \quad \text{ and}\label{sum:1} \\
&\sum_{\sigma \in \Omega(X_1, \ldots, X_n)} \kappa^D_\sigma(X_1, \ldots, X_n)\label{sum:2}
\end{align}
are equal.
This will be accomplished via three operations.  
Each operation will enable us to restrict the partitions one needs to consider in sums (\ref{sum:1}) and (\ref{sum:2}) at the cost of having sums of (\ref{sum:1}) and (\ref{sum:2}) with these restricted partitions and at the cost of having to change the sequence $(X_1, \ldots, X_n)$ to a new sequence $(X'_1, \ldots, X_{n'})$ for each sum.  
The proof will eventually reach the case of sums of identical $B$-values obtained by the third operation with sums of (\ref{sum:1}) and (\ref{sum:2}) where $(X'_1, \ldots, X'_{n'})$ has the property that $X'_k \in \varepsilon(B \otimes 1_B)$ for all $k$ (with the full $\Theta(X'_1, \ldots, X'_{n'})$ and $\Omega(X'_1, \ldots, X'_n)$).  The proof will then be complete since if $(X'_1, \ldots, X'_{n'}) = (L_{b'_1}, \ldots, L_{b'_{m'}})$, then
\begin{align*}
\sum_{\pi \in \Theta\left(L_{b'_1}, \ldots, L_{b'_{m'}}\right)} F\left(\kappa^B_\pi\left(L_{b'_1}, \ldots, L_{b'_{m'}}\right)\right) = F\left(E\left(L_{b'_1}\right) \cdots E\left(L_{b'_{m'}}\right)\right) = F(b'_1 \cdots b'_{m'})
\end{align*}
as $\Theta\left(L_{b'_1}, \ldots, L_{b'_{m'}}\right) = \{0_{m'}\}$ whereas
\begin{align*}
\sum_{\sigma \in \Omega\left(L_{b'_1}, \ldots, L_{b'_{m'}}\right)} \kappa^D_\sigma\left(L_{b'_1}, \ldots, L_{b'_{m'}}\right) = F\left(E\left(L_{b'_1} \cdots L_{b'_{m'}}\right)\right) =  F(b'_1 \cdots b'_{m'})
\end{align*}
as $\Omega\left(L_{b'_1}, \ldots, L_{b'_{m'}}\right) = BNC(m',0)$.

The first operation will enable us to restrict the partitions in (\ref{sum:1}) and (\ref{sum:2}) to partitions with the property that each interval of entries from $Z(D, I)$ that is surrounded by elements from $\epsilon(B \otimes 1_B)$ (or the ends of the sequence) has one element that is equivalent via the partition to an element outside of this interval.  The following diagram illustrates an example of a partition this first operation remove from sums (\ref{sum:1}) and (\ref{sum:2}) where $\tau$ and $\pi'$ are some bi-non-crossing partitions and where the $L_b$ and $L_{b'}$ may be connected in sum (\ref{sum:2}) (the goal is to exterminate $\tau$ leaving only $\pi'$, $L_b$, and $L_{b'}$):
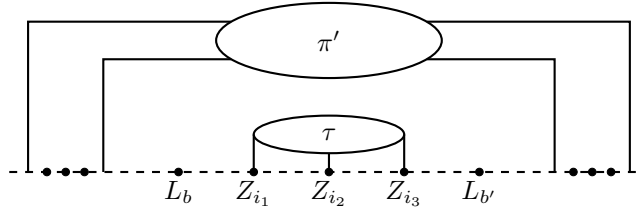
\begin{align*}
\begin{tikzpicture}[baseline]
	\draw[thick, dashed] (-.25,0) -- (8.25, 0);
	\node[below] at (2, 0) {$L_b$};
	\draw[black, fill=black] (2,0) circle (0.05);	
	\node[below] at (3, 0) {$Z_{i_1}$};
	\draw[black, fill=black] (3,0) circle (0.05);
	\node[below] at (4, 0) {$Z_{i_2}$};
	\draw[black, fill=black] (4,0) circle (0.05);	
	\node[below] at (5, 0) {$Z_{i_3}$};
	\draw[black, fill=black] (5,0) circle (0.05);
	\node[below] at (6, 0) {$L_{b'}$};
	\draw[black, fill=black] (6,0) circle (0.05);	
	\draw[thick] (4,.5) ellipse (1cm and .25cm);
	\draw[thick] (3, 0) -- (3, .5);
	\draw[thick] (5, 0) -- (5, .5);
	\draw[thick] (4, 0) -- (4, .25);
	\node[below] at (4, .7) {$\tau$};
	\node[below] at (4, 2) {$\pi'$};
	\draw[thick] (4,1.75) ellipse (1.5cm and .5cm);
	\draw[thick] (1,0) -- (1, 1.5) -- (2.7, 1.5);
	\draw[thick] (7,0) -- (7, 1.5) -- (5.3, 1.5);
	\draw[thick] (0,0) -- (0, 2) -- (2.7, 2);
	\draw[thick] (8,0) -- (8, 2) -- (5.3, 2);
	\draw[black, fill=black] (.25,0) circle (0.05);	
	\draw[black, fill=black] (.5,0) circle (0.05);	
	\draw[black, fill=black] (.75,0) circle (0.05);	
	\draw[black, fill=black] (7.25,0) circle (0.05);	
	\draw[black, fill=black] (7.5,0) circle (0.05);	
	\draw[black, fill=black] (7.75,0) circle (0.05);	
\end{tikzpicture}
\end{align*}

To begin the process of this first operation, choose $p, q \in \theta(X'_1, \ldots, X'_{n'}) \cup \{0, n'+1\}$ such that $p+1 < q$ and $\{p+1, \ldots, q-1\} \cap \theta(X'_1, \ldots, X'_{n'}) = \emptyset$.  For each 
\[
\pi' \in \left\{
\begin{array}{ll}
\Theta(X'_1, \ldots, X'_p, X'_{q},\ldots, X'_{n'}) & \mbox{if } p \neq 0, q \neq n'+1  \\
\Theta(X'_{q},\ldots, X'_{n'}) & \mbox{if } p = 0, q \neq n'+1 \\
\Theta(X'_1, \ldots, X'_p) & \mbox{if } p \neq 0, q = n'+1
\end{array} \right.
\]
(the case $p = 0$ and $q=n'+1$ does not occur), consider $\pi'$ as a bi-non-crossing partition on $\{1, \ldots, p, q+1, \ldots, n'\}$ in the natural way.  We claim we can reduce parts of the sums (\ref{sum:1}) and (\ref{sum:2}) by adding over all $\pi \in \Theta(X'_1, \ldots, X'_{n'})$ and $\sigma \in \Omega(X'_1, \ldots, X'_{n'})$ such that $k \nsim_\pi k'$ and $k \nsim_\sigma k'$ for all $k \in \{p+1, \ldots, q-1\}$ and $k' \notin \{p+1, \ldots, q-1\}$, and such that
\[
\pi|_{\{1, \ldots, p, q+1, \ldots, n'\}} = \sigma|_{\{1, \ldots, p, q+1, \ldots, n'\} \setminus \theta(X'_1, \ldots, X'_{n'})} = \pi'.
\]
Indeed, using the moment-cumulant equation (\ref{eq:mobius}), summing in (\ref{sum:1}) over such $\pi \in \Theta(X'_1, \ldots, X'_{n'})$ produces
\begin{align*}
&F\left(\kappa^B_{\pi'}\left(X'_1, \ldots, X'_p, L_{E(X'_{p+1} \cdots X'_{q-1})} X'_q, X'_{q+1}, \ldots, X_{n'}\right)\right) & & \text{ if } p \neq 0, q \neq n'+1 \\
&F\left(\kappa^B_{\pi'}\left(L_{E(X'_{p+1} \cdots X'_{q-1})} X'_q, X'_{q+1}, \ldots, X'_{n'}\right)\right) & & \text{ if } p =0, q \neq n'+1, \text{ and }\\
&F\left(\kappa^B_{\pi'}\left(X'_1, \ldots, X'_p L_{E(X'_{p+1} \cdots X'_{q-1})}\right)\right) & & \text{ if } p \neq 0, q = n'+1,
\end{align*}
and, if $V = \{1, \ldots, n - (q-p)\} \setminus \theta(X'_1, \ldots, X'_p, X'_{q},\ldots, X'_n)$, summing in (\ref{sum:2}) over such $\sigma \in \Omega(X'_1, \ldots, X'_{n'})$ produces
\begin{align*}
&\sum_{\substack{\sigma' \in \Omega(X'_1, \ldots, X'_p, X_{q},\ldots, X'_n) \\ \sigma'|_{V} = \pi'}} \kappa^D_{\sigma'}\left(X'_1, \ldots, X'_p, L_{F(X'_{p+1} \cdots X'_{q-1})} X'_q, X'_{q+1}, \ldots, X'_{n'}\right) & & \text{ if } p \neq 0, q \neq n'+1 \\
&\sum_{\substack{\sigma' \in \Omega(X'_{q},\ldots, X'_{n'}) \\ \sigma'|_{\{1, \ldots, n' - (q-p)\} \setminus \theta(X'_{q},\ldots, X'_{n'})} = \pi'}} \kappa^D_{\sigma'}\left(L_{F(X'_{p+1} \cdots X'_{q-1})} X'_q, X'_{q+1}, \ldots, X'_{n'}\right) & & \text{ if } p = 0, q \neq n'+1,  \text{ and}\\
&\sum_{\substack{\sigma' \in \Omega(X'_1, \ldots, X'_p) \\ \sigma'|_{\{1, \ldots, n' - (q-p)\} \setminus \theta(X'_1, \ldots, X'_p)} = \pi'}} \kappa^D_{\sigma'}\left(X'_1, \ldots, X'_p L_{F(X'_{p+1} \cdots X'_{q-1})}\right) & &  \text{ if } p \neq 0, q = n'+1.
\end{align*}
Since $X'_k \in Z(D, I)$ for all $k \in \{p+1, \ldots, q-1\}$ as $\{p+1, \ldots, q-1\} \cap \theta(X'_1, \ldots, X'_{n'}) = \emptyset$, 
\[
F(X'_{p+1} \cdots X'_{q-1}) = E(X'_{p+1} \cdots X'_{q-1})
\]
by Remark \ref{rem:ex-D-valued}.  Consequently, since 
\[
L_{E(X'_{p+1} \cdots X'_{q-1})} X'_q, X'_p L_{E(X'_{p+1} \cdots X'_{q-1})} \in \varepsilon(B \otimes 1_B),
\]
since this holds for all $\pi'$, and since distinct $\pi'$ require distinct $\pi$ and $\sigma$, by progressively applying this operation (i.e. applying it to each such interval one at a time) it suffices to show sums (\ref{sum:1}) and (\ref{sum:2}) agree when one restricts to those $\pi \in \Theta(X'_1, \ldots, X'_{n'})$ and $\sigma \in \Omega(X'_1, \ldots, X'_{n'})$ where if $p, q \in \theta(X'_1, \ldots, X'_{n'}) \cup \{0, n'+1\}$ are such that $p+1 < q$ and $\{p+1, \ldots, q-1\} \cap \theta(X'_1, \ldots, X'_{n'}) = \emptyset$, then there exist  $k_1, k_2 \in \{p+1, \ldots, q-1\}$ and $k'_1, k'_2 \notin \{p+1, \ldots, q-1\}$ such that $k_1 \sim_\pi k'_1$ and $k_2 \sim_\sigma k'_2$.  Note this operation will complete the base case of this inductive step (i.e. when there is exactly one element of $Z(D,I)$ in the sequence).

The second operation used to restrict the partitions that need to be considered in sums (\ref{sum:1}) and (\ref{sum:2}) will enable us to reduce to partitions $P$ where if 
\[
\theta(X'_1, \ldots, X'_{n'}) = \{k_1 < k_2 < \cdots < k_{m'}\}
\]
and
\[
V_P = \{k \in \{1,\ldots, n'\} \setminus \theta(X'_1, \ldots, X'_{n'}) \, \mid \, \text{ if }k_q < k < k_{q+1} \text{ then there exists }k' \notin [k_q, k_{q+1}]\text{ with }k' \sim_P k \}
\]
(where $k_0 = 0$ and $k_{m'+1} = n'+1$), then $V_P = \{1,\ldots, n'\} \setminus \theta(X'_1, \ldots, X'_{n'})$.  The following diagram illustrates an example of a partition this second operation removes from sums (\ref{sum:1}) and (\ref{sum:2}) where $\tau_1$, $\tau_2$, and $\pi'$ are some bi-non-crossing partitions (the goal is to exterminate $\tau_1$ and $\tau_2$; this also includes the case that $L_b$ and $\tau_2$ are absent):
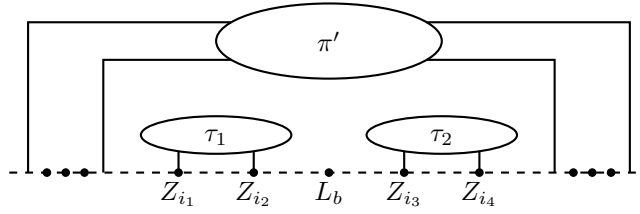
\begin{align*}
\begin{tikzpicture}[baseline]
	\draw[thick, dashed] (-.25,0) -- (8.25, 0);
	\node[below] at (2, 0) {$Z_{i_1}$};
	\draw[black, fill=black] (2,0) circle (0.05);	
	\node[below] at (3, 0) {$Z_{i_2}$};
	\draw[black, fill=black] (3,0) circle (0.05);
	\node[below] at (4, 0) {$L_b$};
	\draw[black, fill=black] (4,0) circle (0.05);	
	\node[below] at (5, 0) {$Z_{i_3}$};
	\draw[black, fill=black] (5,0) circle (0.05);
	\node[below] at (6, 0) {$Z_{i_4}$};
	\draw[black, fill=black] (6,0) circle (0.05);	
	\draw[thick] (2.5,.5) ellipse (1cm and .25cm);
	\draw[thick] (5.5,.5) ellipse (1cm and .25cm);
	\draw[thick] (2, 0) -- (2, .29);
	\draw[thick] (3, 0) -- (3, .29);
	\draw[thick] (5, 0) -- (5, .29);
	\draw[thick] (6, 0) -- (6, .29);
	\node[below] at (2.5, .7) {$\tau_1$};
	\node[below] at (5.5, .7) {$\tau_2$};
	\node[below] at (4, 2) {$\pi'$};
	\draw[thick] (4,1.75) ellipse (1.5cm and .5cm);
	\draw[thick] (1,0) -- (1, 1.5) -- (2.7, 1.5);
	\draw[thick] (7,0) -- (7, 1.5) -- (5.3, 1.5);
	\draw[thick] (0,0) -- (0, 2) -- (2.7, 2);
	\draw[thick] (8,0) -- (8, 2) -- (5.3, 2);
	\draw[black, fill=black] (.25,0) circle (0.05);	
	\draw[black, fill=black] (.5,0) circle (0.05);	
	\draw[black, fill=black] (.75,0) circle (0.05);	
	\draw[black, fill=black] (7.25,0) circle (0.05);	
	\draw[black, fill=black] (7.5,0) circle (0.05);	
	\draw[black, fill=black] (7.75,0) circle (0.05);	
\end{tikzpicture}
\end{align*}

To describe this second operation, for a fixed $\pi_0 \in \Theta(X'_1, \ldots, X'_{n'})$ with $V_{\pi_0} \neq \{1,\ldots, n'\} \setminus \theta(X'_1, \ldots, X'_{n'})$, if we sum over all $\pi \in \Theta(X'_1, \ldots, X'_{n'})$ with $V_\pi = V_{\pi_0}$, we obtain $F(\kappa^B_{\pi'}(X''_1, \ldots, X''_{n''}))$ where 
\[
\pi' = \pi_0|_{V_{\pi_0} \cup \theta(X'_1, \ldots, X'_{n'})} \in \Theta(X''_1, \ldots, X''_{n''})
\]
and $(X''_1, \ldots, X''_{n''})$ is obtained from $(X'_1, \ldots, X'_{n'})$ by multiplying certain $X'_k  \in Z(D, I)$ by $L_{E(X'_{p+1} \cdots X'_{q-1})}$ for some $X'_{p+1}, \ldots, X'_{q-1} \in Z(D, I)$ (which are then elements of $Z(D, I)$).  Similarly, if we sum over all $\sigma \in \Omega(X'_1, \ldots, X'_{n'})$ with $V_\sigma = V_{\pi_0}$, we obtain the sum of all $\kappa^D_{\sigma'}(X''_1, \ldots, X''_{n''})$ where 
\[
\sigma' = \sigma|_{V_{\pi_0} \cup \theta(X'_1, \ldots, X'_{n'})} \in \Omega(X''_1, \ldots, X''_{n''})
\]
since $L_{E(X'_{p+1} \cdots X'_{q-1})} = L_{F(X'_{p+1} \cdots X'_{q-1})}$ for all $X'_{p+1}, \ldots, X'_{q-1} \in Z(D, I)$ that one need consider (by applying the same bi-multiplicative properties in the same orders).  One then sees that 
\[
V_{\pi'} = \{1, \ldots, n''\} \setminus \theta(X''_1, \ldots, X''_{n''})
\]
 and all $\sigma'$ corresponding to $\pi'$ occur (precisely once as we vary $\pi'$).   Consequently, the second operation is complete.

Using the above two operations, in sums (\ref{sum:1}) and (\ref{sum:2}) we need only consider $\pi \in \Theta(X'_1, \ldots, X'_{n'})$ and $\sigma \in \Omega(X'_1, \ldots, X'_{n'})$ such that $V_\pi = V_\sigma = \{1,\ldots, n'\} \setminus \theta(X'_1, \ldots, X'_{n'})$. Our third and final operation will be progressively used to complete the proof.  For a fixed $\pi \in \Theta(X'_1, \ldots, X'_{n'})$, there must be a block $W_\pi = \{k_1 < k_2 < \cdots < k_q\}$ of $\pi$ such that $X'_{k_p} \in  Z(D, I)$ for all $p$ and if $k_p < k < k_{p+1}$, then $\{k\}$ is a block of $\pi$ and $X'_k = L_{b_k}$ for some $b_k \in B$ (i.e. a full partition on an interval of elements of $Z(D,I)$ with singleton elements of $\varepsilon(B \otimes 1_B)$ inserted). 
The following diagram illustrates an example of a partition this third operation reduces in sums (\ref{sum:1}) and (\ref{sum:2}) where $\pi'$ is some bi-non-crossing partition.  The goal is remove the full partition on $Z_{i_1}, \ldots, Z_{i_6}$ with any $L_b$ in between. Note the $L_{b_k}$ may be absent and $L_{b_1}$ and $L_{b_2}$ may be connected in the $\Omega(X'_1, \ldots, X'_{n'})$ term.  Furthermore, $\pi'$ may only connect to one side or be absent.  If $\pi'$ is absent, this operation will complete the proof of equality for this portion of the sums.
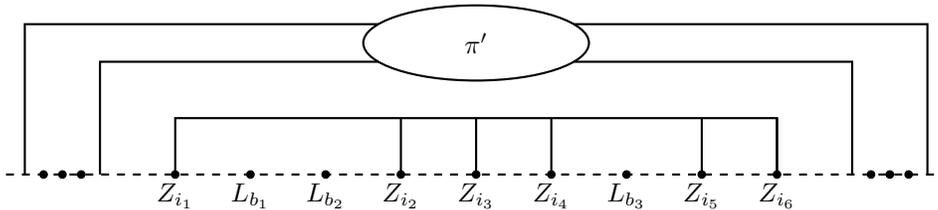
\begin{align*}
\begin{tikzpicture}[baseline]
	\draw[thick, dashed] (-.25,0) -- (12.25, 0);
	\node[below] at (2, 0) {$Z_{i_1}$};
	\draw[black, fill=black] (2,0) circle (0.05);	
	\node[below] at (3, 0) {$L_{b_1}$};
	\draw[black, fill=black] (3,0) circle (0.05);
	\node[below] at (4, 0) {$L_{b_2}$};
	\draw[black, fill=black] (4,0) circle (0.05);	
	\node[below] at (5, 0) {$Z_{i_2}$};
	\draw[black, fill=black] (5,0) circle (0.05);
	\node[below] at (6, 0) {$Z_{i_3}$};
	\draw[black, fill=black] (6,0) circle (0.05);	
	\node[below] at (7, 0) {$Z_{i_4}$};
	\draw[black, fill=black] (7,0) circle (0.05);	
	\node[below] at (8, 0) {$L_{b_3}$};
	\draw[black, fill=black] (8,0) circle (0.05);	
	\node[below] at (9, 0) {$Z_{i_5}$};
	\draw[black, fill=black] (9,0) circle (0.05);	
	\node[below] at (10, 0) {$Z_{i_6}$};
	\draw[black, fill=black] (10,0) circle (0.05);	
	\node[below] at (6, 2) {$\pi'$};
	\draw[thick] (6,1.75) ellipse (1.5cm and .5cm);
	\draw[thick] (1,0) -- (1, 1.5) -- (4.7, 1.5);
	\draw[thick] (11,0) -- (11, 1.5) -- (7.3, 1.5);
	\draw[thick] (0,0) -- (0, 2) -- (4.7, 2);
	\draw[thick] (12,0) -- (12, 2) -- (7.3, 2);
	\draw[black, fill=black] (.25,0) circle (0.05);	
	\draw[black, fill=black] (.5,0) circle (0.05);	
	\draw[black, fill=black] (.75,0) circle (0.05);	
	\draw[black, fill=black] (11.25,0) circle (0.05);	
	\draw[black, fill=black] (11.5,0) circle (0.05);	
	\draw[black, fill=black] (11.75,0) circle (0.05);	
	\draw[thick] (2,0) -- (2, .75) -- (10, .75) -- (10, 0);
	\draw[thick] (5,0) -- (5, .75);
	\draw[thick] (6,0) -- (6, .75);
	\draw[thick] (7,0) -- (7, .75);
	\draw[thick] (9,0) -- (9, .75);
	\draw[thick] (10,0) -- (10, .75);
\end{tikzpicture}
\end{align*}

Note if $W_1 = \{k_1, k_1 + 1, \ldots, k_q\}$, then $W_1$ must be union of blocks of $\pi$ and
\[
\kappa^B_{\pi|_{W_1}}((X'_1, \ldots, X'_{n'})|_{W_1}) = \kappa^B_{\chi_{q,0}}(X'_{k_1}, L_{b'_1} X'_{k_2}, \ldots, L_{b'_{q-1}} X'_{k_q})
\]
where (if $X_k = L_{b_k}$) $b'_p = b_{k_p + 1} b_{k_p+2} \cdots b_{k_{p+1}-1} \in B$ (or $b'_p = 1$ when $k_p +1 = k_{p+1}$).  However, by the assumptions of the lemma along with Remark \ref{rem:property-preserved-products-by-D}, we obtain that
\[
\kappa^B_{W_1}((X'_1, \ldots, X'_{n'})|_{W_1}) = \kappa^D_{\chi_{q,0}}\left(X'_{k_1}, L_{F(b'_1)} X'_{k_2}, \ldots, L_{F(b'_{q-1})} X'_{k_q}\right)
\]
as $X'_{k_p} \in  Z(D, I)$.  Consequently, we get a single $B$-value if $W_1 = \{1, \ldots, n'\}$ and otherwise we may replace $\kappa^B_\pi(X'_1, \ldots, X'_{n'})$ in sum (\ref{sum:1}) with a new sequence $(X''_1, \ldots, X''_{n''})$ where indices corresponding to $W_1$ are removed and one operator is multiplied by 
\[
L_{\kappa^D_{\chi_{q,0}}(X'_{k_1}, L_{F(b'_1)} X'_{k_2}, \ldots, L_{F(b'_{q-1})} X'_{k_q})}.
\]

For the other the sum, if we consider all $\sigma \in \Omega(X'_1, \ldots, X'_{n'})$ with
\[
\sigma|_{\{1,\ldots, n'\} \setminus  \theta(X'_1, \ldots, X'_{n'})} = \pi|_{\{1,\ldots, n'\} \setminus \theta(X'_1, \ldots, X_{n'})}
\]
then for a fixed $\sigma_0$ in this collection, if we sum $\kappa^B_{\sigma|_{W_1}}((X'_1, \ldots, X'_{n'})|_{W_1})$ over all $\sigma$ with $\sigma|_{\{1,\ldots, n'\} \setminus W_1} = \sigma_0|_{\{1,\ldots, n'\} \setminus W_1}$, then we obtain
\[
\kappa^D_{\chi_{q,0}}\left(X'_{k_1}, L_{F(b'_1)} X'_{k_2}, \ldots, L_{F(b'_{q-1})} X'_{k_q}\right)
\]
as
\[
\sum_{\tau \in BNC(k_{p+1} - k_p - 1, 0)} \kappa^D_\tau( b_{k_p + 1},  b_{k_p+2},  \cdots,  b_{k_{p+1}-1}) = F(b'_p).
\]
Hence, in sum (\ref{sum:2}), the correct $B$-value is obtained if $W_1 = \{1, \ldots, n'\}$ and otherwise correct terms are produced to get the sum over all $\sigma \in \Omega(X''_1, \ldots, X''_{n''})$ corresponding to $\pi|_{\{1,\ldots, n'\} \setminus W_1}$ of $\kappa^D_\sigma(X''_1, \ldots, X''_{n''})$.

After using the first two operations to reduce the partitions in (\ref{sum:1}) and (\ref{sum:2}) (at the cost of having sums of sums and of having different sequences), one may apply the third operation a finite number of times to eventually reduced down to sums of sum of the form (\ref{sum:1}) and (\ref{sum:2}) with $(X'_1, \ldots, X'_{n'})$ having the property that $X'_k \in Z(D, I)$ for all $k$, or $X'_k \in \varepsilon(B \otimes 1_B)$ for all $k$.  Consequently, early arguments complete the proof.
\end{proof}

\begin{lem}
\label{lem:one-direction}
Let $Z$ and $F$ be as in Theorem \ref{thm:bi-free-over-D} and suppose 
\begin{align*}
\kappa^B_{Z, \omega}(b_1, \ldots, b_{n-1}) = \kappa^D_{Z, \omega}(F(b_1), \ldots, F(b_{n-1})).
\end{align*}
for all $n\geq 1$, $\omega : \{1,\ldots, n\} \to I \sqcup J$, and $b_1, \ldots, b_{n-1} \in B$. 
Then 
\[
(\alg(\varepsilon(D \otimes 1_D), \{Z_i\}_{i \in I}), \alg(\varepsilon(1_D \otimes D^\op), \{Z_j\}_{j \in J})) \text{ is bi-free from }(\varepsilon(B \otimes 1_B), \varepsilon(1_B \otimes B^\op))
\]
with amalgamation over $D$.
\end{lem}
\begin{proof}
Combining Corollary \ref{cor:generators-vanishing-cumulants} along with Lemmata \ref{lem:interchange-cumulant} and \ref{lem:swap-cumulant} (which apply since $Z_i$ and $R_b$ commute if $i \in I$, $Z_j$, $L_b$ commute if $j \in J$, and $F(E(Z'L_b)) = F(E(Z'R_b))$ for all $Z' \in \A$ and $b \in B$), it suffices to demonstrate that:
\begin{enumerate}
\item for all $Z'_1, \ldots, Z'_n \in \epsilon(B \otimes 1_B) \cup \{ L_{d_1} Z_i L_{d_2} \, \mid \, i \in I, d_1, d_2 \in D\}$ where at least one $Z'_k$ is an element of $\epsilon(B \otimes 1_B)$ and at least one $Z'_k$ is an element of $\{ L_{d_1} Z_i L_{d_2} \, \mid \, i \in I, d_1, d_2 \in D\}$ (and $\chi(k) = \ell$ for all $k$), we have $\kappa^D_\chi(Z'_1, \ldots, Z'_n) = 0$,
\item for all $Z'_1, \ldots, Z'_n \in \epsilon(1 \otimes B^\op) \cup \{ R_{d_1} Z_j R_{d_2} \, \mid \, j \in J, d_1, d_2 \in D\}$ where at least one $Z'_k$ is an element of $\epsilon(1_B \otimes B^\op)$ and at least one $Z'_k$ is an element of $\{ R_{d_1} Z_j R_{d_2} \, \mid \, j \in J, d_1, d_2 \in D\}$ (and $\chi(k) = r$ for all $k$) we have $\kappa^D_\chi(Z'_1, \ldots, Z'_n) = 0$, and
\item for all sequences $(Z'_1, \ldots, Z'_n)$ which are concatenations of sequences of the form 
\begin{itemize}
\item $L_{b_1}, L_{b_2},\ldots, L_{b_k}, L_{d_1} Z_i L_{d_2}$ where $k\geq 0$, $b_1,\ldots, b_k \in B$, $i \in I$, $d_1, d_2 \in D$, and 
\item $R_{b_1}, R_{b_2},\ldots, R_{b_k}, R_{d_1} Z_j R_{d_2}$ where $k \geq 0$, $b_1,\ldots, b_k \in B$, $j \in J$, $d_1, d_2 \in D$
\end{itemize}
with at least one of each occurring, followed by concatenating $L_{b_1}, L_{b_2},\ldots, L_{b_k}$ on the right-hand-side where $k\geq 0$, $b_1,\ldots, b_k \in B$, we have $\kappa^D_\chi(Z'_1, \ldots, Z'_n) = 0$ (where $\chi$ is determined by the sequence) provided at least one $L_b$ or $R_b$ occurs in the sequence.
\end{enumerate}
Note (1) is true by Lemma \ref{lem:base-case-2} (note this reproves one direction of \cite{NSS2002}*{Theorem 3.5}). Similarly, (2) is true by identical arguments as used in Lemmata \ref{lem:base-case-1} and \ref{lem:base-case-2}.  In fact, (3) is also true by identical arguments as those used in Lemmata \ref{lem:base-case-1} and \ref{lem:base-case-2} once `interval' is replaced with `$\chi$-interval' and `$<$' is replaced with `$\prec_\chi$'.  Indeed, Lemmata \ref{lem:base-case-1} and \ref{lem:base-case-2} only requires combinatorial arguments together with multiplicative properties of the operator-valued free cumulant functions all of which hold when one considers the combinatorics of operator-valued bi-freeness are identical to those of operator-valued freeness by the identification illustrated in Example \ref{exam:FLOP} together with Remark \ref{rem:bi-multi-same-as-free} (and one can change the bottom most $R_b$ to an $L_b$ and vice versa).  Consequently, it will be left to the reader to observe Lemmata \ref{lem:base-case-1} and \ref{lem:base-case-2} generalize.
\end{proof}

With one direction of Theorem \ref{thm:bi-free-over-D} complete, we turn our attention to the other direction.  In order to complete the proof, we will require a method for constructing a pair of $B$-faces with any operator-valued bi-free cumulants we want.  The following does the trick.
\begin{lem}
\label{lem:get-the-cumulants-you-want}
Let $I$ and $J$ be non-empty, disjoint index sets.  For every $n \geq 1$ and $\omega : \{1,\ldots, n\} \to I \sqcup J$ let $\Theta_\omega : B^{n-1} \to B$ be complex linear in each coordinate of $B^{n-1}$.  There exist a $B$-$B$-non-commutative probability space $(\A, E, \varepsilon)$ and elements $\{Z_i\}_{i \in I} \subseteq \A_\ell$ and $\{Z_j\}_{j \in J} \subseteq \A_r$ such that if $Z = \{Z_i\}_{i \in I} \sqcup \{Z_j\}_{j \in J}$, then
\[
\kappa^B_{Z, \omega}(b_1, \ldots, b_{n-1}) = \Theta_\omega(b_1, \ldots, b_{n-1})
\]
for all $n\geq 1$, $\omega : \{1,\ldots, n\} \to I \sqcup J$, and $b_1, \ldots, b_{n-1} \in B$.
\end{lem}
\begin{proof}
Let $\A_0$ be the universal unital free algebra generated by symbols 
\[
\{\bC 1_\A\} \sqcup \{Z_i\}_{i \in I} \sqcup \{Z_j\}_{j \in J} \sqcup \{L_b \, \mid \, b \in B\} \sqcup \{R_b \, \mid \, b \in B\}.
\]
Let $\A$ be the unital algebra  $\A_0$ modulo the two-sided ideal $\I$ generated by 
\begin{gather*} 
\{L_{z 1_B} - z 1_\A, R_{z 1_B} - z 1_\A \, \mid \, z \in \bC\}, \qquad \{L_{zb + b'} - zL_{b} - L_{b'}, R_{zb + b'} - zR_{b} - R_{b'} \, \mid \, b, b' \in B, z \in \bC\}, \\
\{L_{bb'} - L_b L_{b'}, R_{bb'} - R_{b'} R_b, L_b R_{b'} - R_{b'} L_b \, \mid \, b,b' \in B\}, \\
\{Z_i R_b - R_b Z_i \, \mid \, b \in B, i \in I\}, \qqand \{Z_j L_b - L_b Z_j \, \mid \, b \in B, j \in J\}.
\end{gather*} 
Clearly $\A$ is a unital algebra such that if $\varepsilon : B \otimes B^{\op} \to \A$ is defined via $\varepsilon(b_1 \otimes b_2) = L_{b_1} R_{b_2}$, then $\varepsilon$ is a unital homomorphism such that $\varepsilon|_{B \otimes 1_B}$ and $\varepsilon|_{1_B \otimes B^\op}$ are injective.  Furthermore, by construction $\{Z_i\}_{i \in I} \subseteq \A_\ell$ and $\{Z_j\}_{j \in J} \subseteq \A_r$ (where $\A_\ell$ and $\A_r$ are as in Definition \ref{defn:BBncps}).

To defined an expectation $E : \A \to B$ so that $(\A, E, \varepsilon)$ is a $B$-$B$-non-commutative probability space, first note that every element in $\A$ is a linear combination of elements of the form
\[
C_{b_1} Z_{k_1} C_{b_2} Z_{k_2} \cdots C_{b_n} Z_{k_n} L_b R_{b'} + \I
\]
$n\geq 0$, $k_q\in I \sqcup J$, $b_1, \ldots, b_n, b, b' \in B$, and 
\[
C_{b_q} = \left\{
\begin{array}{ll}
L_{b_k} & \mbox{if } k_q \in I \\
R_{b_k} & \mbox{if } k_q \in J
\end{array} \right.  
\]
Furthermore, a linear combination of such elements is unique up to taking linear combinations of each $B$-term (i.e. one of the above forms can be written a non-trivial linear combination of the others only if the same sequence of $Z_k$'s are used and one can take linear combinations of the $B$-terms).

We will use these representations along with $\Theta$ to defined the correct expectation $E$ for any element of in $\A$.  To do this, note the properties of bi-multiplicative functions give a way (well, actually many ways) to reduce $\kappa^B_\pi$ to expressions involving only $\kappa^B_{1_\chi}$ for various $\chi$.  We will use $\Theta$ to define the values of $\kappa^B_{1_\chi}$ in the correct way and define the expectation to be the sum of the corresponding reduced $\kappa^B_\pi$ expressions.

To begin, for $n\geq 1$, $\omega : \{1, \ldots, n\} \to I \sqcup J$, and $b_1, \ldots, b_{n+1} \in B$, define the following:
\begin{enumerate}
\item If $n = 1$, define $\hat{\Theta}_{1_{\chi_\omega}}(L_{b_1}) = b_1$ when $\omega(1) \in I$ and $\hat{\Theta}_{1_{\chi_\omega}}(R_{b_1}) = b_1$ when $\omega(1) \in J$.
\item if $\omega(k) \in I$ for all $k$, define
\[
\hat{\Theta}_{1_{\chi_\omega}}(L_{b_1} Z_{\omega(1)}, L_{b_2} Z_{\omega(2)}, \ldots, L_{b_{n-1}} Z_{\omega(n-1)}, L_{b_n} Z_{\omega(n)} L_{b_{n+1}}) = b_1 \Theta_\omega(b_2, b_3, \ldots, b_n) b_{n+1}.
\]
\item if $\omega(k) \in J$ for all $k$, define
\[
\hat{\Theta}_{1_{\chi_\omega}}(R_{b_1} Z_{\omega(1)}, R_{b_2} Z_{\omega(2)}, \ldots, R_{b_{n-1}} Z_{\omega(n-1)}, R_{b_n} Z_{\omega(n)} R_{b_{n+1}}) = b_{n+1} \Theta_\omega(b_2, b_3, \ldots, b_n) b_{1}.
\]
\item Otherwise let $k_\ell = \min\{k \, \mid \, \omega(k) \in I\}$ and $k_r = \min\{k \, \mid \, \omega(k) \in J\}$.  Then $\{k_\ell, k_r\} = \{1, k_0\}$ for some $k_0$.  If
\[
C^{\omega(k)}_{b} = \left\{
\begin{array}{ll}
L_{b} & \mbox{if } \omega(k) \in I \\
R_{b} & \mbox{if } \omega(k) \in J
\end{array} \right.  ,
\]
define
\begin{align*}
\hat{\Theta}_{1_{\chi_\omega}} & \left(C^{\omega(1)}_{b_1}Z_{\omega(1)}, C^{\omega(2)}_{b_2} Z_{\omega(2)}, \ldots, C^{\omega(n-1)}_{b_{n-1}} Z_{\omega(n-1)}, C^{\omega(n)}_{b_{n}} Z_{\omega(n)} C^{\omega(n)}_{b_{n+1}}\right)\\ 
&= \left\{
\begin{array}{ll}
b_1 \Theta_\omega((b_2, b_3, \ldots, b_{n-1})|_{\{2,\ldots, n\} \setminus \{k_0\}}) b_{k_0} & \mbox{if } k_\ell = 1 \\
b_{k_0} \Theta_\omega((b_2, b_3, \ldots, b_{n-1})|_{\{2,\ldots, n\} \setminus \{k_0\}}) b_{1}& \mbox{if } k_r = 1
\end{array} \right. .
\end{align*}
\end{enumerate}

Subsequently, for each $\omega : \{1,\ldots, n\} \to I \sqcup J$ and for each $\pi \in BNC(\chi_\omega)$, choose one method of reduction so that if $\Phi$ is a bi-multiplicative function then $\Phi_\pi$ may be reduced to a nested expression using only $\Phi_{1_\chi}$'s.  Then for
\[
C^{\omega(k)}_{b} = \left\{
\begin{array}{ll}
L_{b} & \mbox{if } \omega(k) \in I \\
R_{b} & \mbox{if } \omega(k) \in J
\end{array} \right.  ,
\]
define
\[
\hat{\Theta}_{\pi} \left(C^{\omega(1)}_{b_1}Z_{\omega(1)}, C^{\omega(2)}_{b_2} Z_{\omega(2)}, \ldots, C^{\omega(n-1)}_{b_{n-1}} Z_{\omega(n-1)}, C^{\omega(n)}_{b_{n}} Z_{\omega(n)} C^{\omega(n)}_{b_{n+1}}\right)
\]
using the nested expression and the definitions of $\hat{\Theta}_{1_\chi}$.

Define $E : \A \to B$ by
\[
E(L_b R_{b'} + \I) = bb'
\]
for all $b,b' \in B$, and for $n\geq 1$, $\omega : \{1,\ldots, n\} \to I \sqcup J$, and
\[
C_{b_k} = \left\{
\begin{array}{ll}
L_{b_k} & \mbox{if }\omega(k) \in I \\
R_{b_k} & \mbox{if }\omega(k) \in J
\end{array} \right.  ,
\]
define
\[
E(C_{b_1} Z_{\omega(1)} \cdots C_{b_n} Z_{\omega(n)} L_b R_{b'} + \I) = \sum_{\pi \in BNC(\chi_\omega)} \hat{\Theta}_\pi(C_{b_1} Z_{\omega(1)}, \ldots, C_{b_{n-1}} Z_{\omega(n-1)}, C_{b_n} Z_{\omega(n)} C_{bb'})
\]
(where $C_{bb'} = L_{bb'}$ if $\omega(n) \in I$ and $C_{bb'} = R_{bb'}$ if $\omega(n) \in J$), and extend $E$ by linearity.  We note $E$ is well-defined on $\A$ as the elements it has been defined on are unique up to linear combinations of the $B$-terms and one easily sees there is no issue in the definition of $E$ as the definition of $\hat{\Theta}_\pi$ implies linearity in each $B$-term (i.e. $E$ is well-defined under each of the generators of $\I$; alternatively, take a $\bC$-basis $\fB$ for $B$, define $E$ using only $b \in \fB$, and extend by linearity).  

By construction of $E$ and $\Theta_\pi$ and by commutation in $\A$, one can verify that 
\[
E(L_b R_{b'} Z') = b E(Z') b' \qqand E(Z'L_b) = E(Z'R_b)
\]
for all $b, b' \in B$ and $Z' \in \A$.  Finally, as the moment-cumulant formula (\ref{eq:cumulants}) completely determines the operator-valued bi-free cumulants via the operator-valued bi-free moments, and since $E$ is bi-multiplicative, M\"{o}bius inversion implies that if $Z = \{Z_i\}_{i \in I} \sqcup \{Z_j\}_{j \in J}$, then
\[
\kappa^B_{Z, \omega}(b_1, \ldots, b_{n-1}) = \Theta_\omega(b_1, \ldots, b_{n-1})
\]
for all $n\geq 1$, $\{1,\ldots, n\} \to I \sqcup J$, and $b_1, \ldots, b_{n-1} \in B$.
\end{proof}

\begin{proof}[Proof of Theorem \ref{thm:bi-free-over-D}]
By Lemma \ref{lem:one-direction}, one direction of the proof is complete.

Conversely, suppose $(\alg(\varepsilon(D \otimes 1_D), \{Z_i\}_{i \in I}), \alg(\varepsilon(1_D \otimes D^\op), \{Z_j\}_{j \in J}))$ is bi-free from $(\varepsilon(B \otimes 1_B), \varepsilon(1_B \otimes B^\op))$ with amalgamation over $D$.  Using Lemma \ref{lem:get-the-cumulants-you-want}, there exist a $B$-$B$-non-commutative probability space $(\A', E', \varepsilon')$ and elements $\{Z'_i\}_{i \in I} \subseteq \A'_\ell$ and $\{Z'_j\}_{j \in J} \subseteq \A'_r$ such that if $Z' = \{Z'_i\}_{i \in I} \sqcup \{Z'_j\}_{j \in J}$, then
\[
\kappa^B_{Z', \omega}(b_1, \ldots, b_{n-1}) = \kappa^D_{Z, \omega}(F(b_1), \ldots, F(b_{n-1}))
\]
for all $n\geq 1$, $\omega : \{1,\ldots, n\} \to I \sqcup J$, and $b_1, \ldots, b_{n-1} \in B$.  Consequently, the first part of the proof implies that $(\alg(\varepsilon'(D \otimes 1_D), \{Z'_i\}_{i \in I}), \alg(\varepsilon'(1_D \otimes D^\op), \{Z'_j\}_{j \in J}))$ is bi-free from $(\varepsilon'(B \otimes 1_B), \varepsilon'(1_B \otimes B^\op))$ with amalgamation over $D$.  

Since $\kappa^B_{Z', \omega}(d_1, \ldots, d_{n-1}) \in D$ for all $d_1, \ldots, d_{n-1} \in D$, Theorem \ref{thm:cumulants-D-valued} implies that
\[
\kappa^D_{Z', \omega}(d_1, \ldots, d_{n-1}) = \kappa^B_{Z', \omega}(d_1, \ldots, d_{n-1}) = \kappa^D_{Z, \omega}(F(d_1), \ldots, F(d_{n-1})) = \kappa^D_{Z, \omega}(d_1, \ldots, d_{n-1}).
\]
Consequently, $(\alg(\varepsilon(D \otimes 1_D), \{Z_i\}_{i \in I}), \alg(\varepsilon(1_D \otimes D^\op), \{Z_j\}_{j \in J}))$ and $(\alg(\varepsilon'(D \otimes 1_D), \{Z'_i\}_{i \in I}), \alg(\varepsilon'(1_D \otimes D^\op), \{Z'_j\}_{j \in J}))$ have the same $D$-valued distributions.  Since both are bi-free from copies of $(B, B^{\op})$ (which have the same distribution), the $D$-valued distributions of
\[
\alg( \varepsilon(B \otimes B^{\op}), Z) \text{ with respect to }F \circ E \qqand \alg(\varepsilon(B \otimes B^{\op}), Z') \text{ with respect to } F \circ E'
\]
are equal.  However, since
\[
F(E(L_{b} T)) = F(b E(T)) \qqand F(E'(L_b T')) = F(b E'(T'))
\]
for all $T \in \alg( \varepsilon(B \otimes B^{\op}), Z)$ and the corresponding $T' \in \alg(\varepsilon(B \otimes B^{\op}), Z')$, the faithfulness condition on $F$ implies that $E(T) = E'(T')$.  Hence $\alg( \varepsilon(B \otimes B^{\op}), Z)$ and $\alg(\varepsilon(B \otimes B^{\op}), Z') $ have the same $B$-valued distributions and thus the same $B$-valued cumulants; that is
\[
\kappa^B_{Z, \omega}(b_1, \ldots, b_{n-1}) = \kappa^B_{Z', \omega}(b_1, \ldots, b_{n-1}) = \kappa^D_{Z, \omega}(F(b_1), \ldots, F(b_{n-1}))
\]
as desired.
\end{proof}

\section{$R$-Cyclic Pairs of Matrices}
\label{sec:R-cyc}

In this section, we will analyze an application of Theorem \ref{thm:bi-free-over-D}.  One use of Theorem \ref{thm:bi-free-over-D} is that if one knows $Z$ is bi-free from $(B, B^\op)$ over $D$, then one can deduce one of the $B$-valued or $D$-valued cumulants from the other.  An example of the opposite direction will be given; that is, we will establish condition (\ref{eq:cummulant-condition-1}) holds for some $Z$, $B$, and $D$ thereby enabling us to conclude bi-freeness of $Z$ from $(B, B^\op)$ over $D$.

Our particular example will relate to pairs of matrices.  To be specific, Theorem \ref{thm:R-cyc} demonstrates a condition to determine when a pair of matrices is bi-free from the scalar matrices over the scalar diagonal matrices.  The construction in bi-free probability used to consider pairs of matrices is one described below and is the same as the one used for the bi-matrix models in \cite{S2015-3}.  This thereby provides additional evidence that the following construction is the correct way to view pairs of matrices in the bi-free setting.  Note we will use $[a_{i,j}]$ to denote the matrix whose $(i,j)^\th$ entry is $a_{i,j}$.

Let $(\A, \varphi)$ be a non-commutative probability space and let $\M_d(\A)$ denote the algebra of $d \times d$ matrices with entries in $\A$.  For all $[T_{i,j}] \in \M_d(\A)$, define 
\[
\varphi_d([T_{i,j}]) = [\varphi(T_{i,j})].
\]
As in \cite{S2015-3}, one can turn the linear maps on $\M_d(\A)$, denoted $\L(\M_d(\A))$, into a $\M_d(\bC)$-$\M_d(\bC)$-non-commutative probability space in such a way that analogues of results from free probability hold in the bi-free setting.  Indeed, for $[a_{i,j}] \in \M_d(\bC)$, define
\[
L_{[a_{i,j}]}([T_{i,j}]) = \left[\sum^d_{k=1} a_{i,k} T_{k,j}   \right] \qqand R_{[a_{i,j}]}([T_{i,j}]) = \left[\sum^d_{k=1} a_{k,j} T_{i,k}   \right].
\]
Thus if $\varepsilon : \M_d(\bC) \otimes \M_d(\bC)^{\op} \to \L(\M_d(\A))$ is defined via 
\[
\varepsilon([a_{i,j}] \otimes [a'_{i,j}]) = L_{[a_{i,j}]} R_{[a'_{i,j}]},
\]
and $E_d : \L(\M_d(\A)) \to \M_d(\bC)$ is defined via
\[
E_d(Z) = \varphi_d(Z(I_d)),
\]
where $I_d$ is the $d \times d$ identity matrix, then $(\L(\M_d(\A)), E_d, \varepsilon)$ is a $\M_d(\bC)$-$\M_d(\bC)$-non-commutative probability space.

There are natural ways to embed $\M_d(\A)$ into both $\L(\M_d(\A))_\ell$ and $\L(\M_d(\A))_r$. Indeed, if $[Z_{i,j}] \in \M_d(\A)$, we define $L([Z_{i,j}]), R([Z_{i,j}]) \in \L(\M_d(\A))$ via
\[
L([Z_{i,j}]) [T_{i,j}] = \left[ \sum^d_{k=1} Z_{i,k}T_{k,j}\right] \qand R([Z_{i,j}]) [T_{i,j}] = \left[ \sum^d_{k=1} Z_{k,j} T_{i,k}\right].
\]
It is elementary to shows that $L([Z_{i,j}]) \in \L(\M_d(\A))_\ell$, $R([Z_{i,j}]) \in \L(\M_d(\A))_r$, $([Z_{i,j}] \in \M_d(\A)) \mapsto L([Z_{i,j}])$ is a unital homomorphism, and $([Z_{i,j}] \in \M_d(\A^\op)^\op) \mapsto R([Z_{i,j}])$ is a unital homomorphism (that preserve adjoints if $\A$ is a $*$-non-commutative probability space).  Furthermore 
\[
E_d(L([Z_{i,j}])) = E_d(R([Z_{i,j}])) = \varphi_d([Z_{i,j}]).
\]

The motivation for the above construction was derived from proving the following result.

\begin{thm}[specific case of \cite{S2015-1}*{Theorem 6.3.1}]
\label{thm:bi-free-with-amalgamation-over-matrix-algebra}
Let $(\A, \varphi)$ be a non-commutative probability space and let $\{(A_{\ell, k}, A_{r,k})\}_{k \in K}$ be bi-free pairs of faces with respect to $\varphi$.  Then $\{(L(\M_d(A_{\ell,k})), R(\M_N((A_{r,k})^{\op})^\op))\}_{k \in K}$ are bi-free with amalgamation over $\M_d(\bC)$ with respect to $E_d$.
\end{thm}

Essential to the proof of Theorem \ref{thm:bi-free-with-amalgamation-over-matrix-algebra} is the following lemma that we will make substantial use of.  For the remainder of this section, let $E_{i,j} \in \M_d(\bC)$ denote the $d \times d$ matrix unit with a $1$ in the $(i,j)^\th$ entry and zeros elsewhere.
\begin{lem}[\cite{S2015-3}*{Lemma 3.7}]
\label{lem:expanding-moments-of-matrices}
Let $(\A, \varphi)$ be a non-commutative probability space, let $\chi : \{1, \ldots, n\} \to \{\ell, r\}$, let $\{[Z_{k; i,j}]\}_{k=1}^n \subseteq \M_d(\A)$, and let $Z_k = L([Z_{k; i,j}])$ if $\chi(k) = \ell$ and $Z_k = R([Z_{k; i,j}])$ if $\chi(k) = r$.  Then
\[
E_d(Z_1 \cdots Z_n) = \sum^d_{\substack{i_1, \ldots, i_n = 1 \\ j_1, \ldots, j_n = 1}} \varphi(Z_{1; i_1, j_1}  \cdots Z_{n; i_n, j_n}) E_\chi((i_1, \ldots, i_n), (j_1, \ldots, j_n))
\]
where
\[
E_\chi((i_1, \ldots, i_n), (j_1, \ldots, j_n)) := E_{i_{s_\chi(1)}, j_{s_\chi(1)}} \cdots E_{i_{s_\chi(n)}, j_{s_\chi(n)}} \in \M_d(\bC).
\]
\end{lem}

As an immediate corollary of Lemma \ref{lem:expanding-moments-of-matrices}, we obtain the following formula for the $\M_d(\bC)$-valued cumulants for pairs of matrices via the $\bC$-valued cumulants of the entries.

\begin{cor}
\label{cor:cumulant-for-R-cyclic}
Let $(\A, \varphi)$ be a non-commutative probability space, let $\chi : \{1, \ldots, n\} \to \{\ell, r\}$, $\{[Z_{k; i,j}]\}_{k=1}^n \subseteq \M_d(\A)$, and let $Z_k = L([Z_{k; i,j}])$ if $\chi(k) = \ell$ and $Z_k = R([Z_{k; i,j}])$ if $\chi(k) = r$.  Then
\[
\kappa^{\M_d(\bC)}_\chi (Z_1,  \ldots, Z_n) = \sum^d_{\substack{i_1, \ldots, i_n = 1 \\ j_1, \ldots, j_n = 1}} \kappa^\bC_\chi (Z_{1; i_1, j_1}, \ldots,  Z_{n; i_n, j_n}) E_\chi((i_1, \ldots, i_n), (j_1, \ldots, j_n)).
\]
\end{cor}
\begin{proof}
Using Lemma \ref{lem:expanding-moments-of-matrices} and bi-multiplicative properties, one can see that if $\pi \in BNC(\chi)$, then 
\[
(E_d)_\pi(Z_1, \ldots, Z_n) = \sum^d_{\substack{i_1, \ldots, i_n = 1 \\ j_1, \ldots, j_n = 1}} \varphi_\pi(Z_{1; i_1, j_1}, \ldots, Z_{n; i_n, j_n}) E_\chi((i_1, \ldots, i_n), (j_1, \ldots, j_n)).
\]
Hence the result follows from the mobius-cumulant formula (\ref{eq:cumulants}).
\end{proof}

In order to obtain examples of pairs of matrices that are bi-free from the scalar matrices over the diagonal scalar matrices, we make the following definition (which enforces specific constraints on the $\bC$-valued cumulants of the entries of the matrices).

\begin{defn}\label{defn:R-cyclic}
Let $I$ and $J$ be disjoint index sets and let $\{[Z_{k; i,j}]\}_{k \in I} \cup \{[Z_{k; i,j}]\}_{k \in J} \subseteq \M_d(\A)$.  The pair $(\{[Z_{k; i,j}]\}_{k \in I},  \{[Z_{k; i,j}]\}_{k \in J})$ is said to be \emph{$R$-cyclic} if for every $n \geq 1$, $\omega : \{1, \ldots, n\} \to I \sqcup J$, and $1 \leq i_1, \ldots, i_n, j_1, \ldots, j_n \leq d$,
\[
\kappa^{\bC}_{\chi_\omega}(Z_{\omega(1); i_1, j_1}, Z_{\omega(2); i_2, j_2}, \ldots, Z_{\omega(n); i_n, j_n}) = 0
\]
whenever at least one of $j_{s_\chi(1)} = i_{s_\chi(2)}$, $j_{s_\chi(2)} = i_{s_\chi(3)}$, $\ldots$, $j_{s_\chi(n-1)} = i_{s_{\chi(n)}}$, $j_{s_\chi(n)} = i_{s_\chi(1)}$ fail.
\end{defn}

\begin{exam}
Let $(\A, \varphi)$ be a non-commutative probability space and consider the diagonal matrices $Z^\ell = \diag(X_1, \ldots, X_d), Z^r = \diag(Y_1, \ldots, Y_d) \in \M_d(\A)$.  Then the pair $(Z^\ell, Z^r)$ can be verified to be $R$-cyclic if and only if $\{(X_k, Y_k)\}_{k=1}^d$ is a bi-free family of pairs of faces (via Theorem \ref{thm:bifree-classifying-theorem}).
\end{exam}

\begin{exam}
\label{exam:creation-example}
Let $\F(\H)$ denote the Fock space of a Hilbert space $\H$.  For each $h \in \H$, let $l(h)$ and $r(h)$ denote the left and right creation operators corresponding to $h$ respectively and let $l^*(h)$ and $r^*(h)$ denote the left and right annihilation operators corresponding to $h$ respectively.  Let $\{h_{k;i,j} \, \mid \, 1 \leq i,j \leq d, k \in K\}$ be an orthonormal subset of $\H$ and consider the pair
\[
\left(\{[l(h_{k;i,j})], [l^*(h_{k;j,i})]\}, \{[r(h_{k;i,j})], [r^*(h_{k;j,i})]\}    \right).
\]
This pair is $R$-cyclic.  Indeed, as 
\[
\{(\{l(h_{k;i,j}), l^*(h_{k;i,j})\, \mid \, 1 \leq i,j \leq d\}, \{r(h_{k;i,j}), r^*(h_{k;i,j})\, \mid \, 1 \leq i,j \leq d\})\}_{k \in K}
\]
are bi-free and each pair is a bi-free central limit distribution (see \cite{V2014}*{Section 7}), all cumulants of order 1 or of order at least 3 vanish.  In addition, the only second order cumulants that are non-zero must be of the form
\[
\kappa^{\bC}(l^*(h), l(h')), \quad \kappa^{\bC}(r^*(h), l(h')), \quad \kappa^{\bC}(l^*(h), r(h')), \quad \kappa^{\bC}(r^*(h), r(h')), 
\]
with $\langle h, h'\rangle_\H \neq 0$.  Thus, for cumulants with entries from the matrices in the pair 
\[
\left(\{[l(h_{k;i,j})], [l^*(h_{k;j,i})]\}, \{[r(h_{k;i,j})], [r^*(h_{k;j,i})]\}    \right)
\]
to be non-zero, one requires $h = h_{k;i,j} = h'$ for some $1 \leq i,j \leq d$ and $k \in K$.  Since $l^*(h_{k;i,j})$ and $r^*(h_{k;i,j})$ both occur only in the $(j,i)$-entry of a matrix, and since $l(h_{k;i,j})$ and $r(h_{k;i,j})$ both occur only in the $(i,j)$-entry of a matrix, Definition \ref{defn:R-cyclic} is indeed verified for this pair.
\end{exam}

\begin{exam}
Let $(\A, \varphi)$ be a $^*$-non-commutative probability space and let $(X, Y)$ be a pair of elements of $\A$.  We will say that the pair $(X, Y)$ is \emph{$R$-diagonal} if all odd order  $\bC$-valued bi-free cumulants involving $(\{X, X^*\}, \{Y, Y^*\})$ are zero and
\[
\kappa^\bC_\chi(Z_1, \ldots, Z_{2n}) = 0
\]
unless $\left(Z_{s_\chi(1)}, \ldots, Z_{s_{\chi(2n)}}\right)$ is of one of the following forms:
\begin{itemize}
\item $(Z, Z^*, Z, Z^*,  \ldots, Z, Z^*)$ with $Z = X$ or $Z = Y$, 
\item $(Z^*, Z, Z^*, Z, \ldots, Z^*, Z)$ with $Z = X$ or $Z = Y$, 
\item $(X, X^*, X, X^*,  \ldots, X, X^*, Y, Y^*, Y, Y^*, \ldots, Y, Y^*)$,
\item $(X, X^*, X, X^*, \ldots, X, X^*, X, Y^*, Y, Y^*, Y, Y^*, \ldots, Y, Y^*)$,
\item $( X^*, X, X^*, X, \ldots, X^*, X,   Y^*, Y, Y^*, Y, \ldots,  Y^*, Y)$, or
\item $(X^*, X, X^*, X, \ldots, X^*, X, X^*, Y, Y^*, Y, Y^*, Y, \ldots, Y^*, Y)$
\end{itemize}
(i.e. alternate between $^*$-terms and non-$^*$-terms, with any number of $X$ terms followed by any number of $Y$ terms).  In particular, if $(X, Y)$ is $R$-diagonal, then $X$ is $R$-diagonal and $Y$ is $R$-diagonal as defined in \cite{NSS2002-R}.

It is not difficult to check that if $(X, Y)$ is $R$-diagonal, then the pair
\[
\left( \begin{bmatrix}
0 & X \\ X^* & 0
\end{bmatrix}, \begin{bmatrix}
0 & Y \\ Y^* & 0
\end{bmatrix}\right)
\]
is $R$-cyclic.
\end{exam}

Many more examples of $R$-cyclic pairs of matrices may be constructed from know examples of $R$-cyclic families of matrices (see \cite{NSS2002-R} for examples).  All such examples arise from placing specific patterns on the $\bC$-valued cumulants of matrices in the pair.

With the definition of $R$-cyclic pairs complete, we shift our attention to showing that $R$-cyclic pairs are precisely those pairs that are bi-free from the scalar matrices over the diagonal matrices.  To begin, let $\D_d$ denoted the subalgebra of $\M_d(\bC)$ consisting of all diagonal matrices and let $F : \M_d(\bC) \to \D_d$ be the conditional expectation onto the diagonal.

In order to invoke Theorem \ref{thm:bi-free-over-D}, a method for computing the $\D_d$-valued bi-free cumulants is required.  We note the following which is the bi-free analogue of \cite{NSS2002-R}*{Theorem 7.2} (and is proved using similar techniques).
\begin{lem}
\label{lem:cumulant-for-R-cyclic-diagonal}
Let $(\A, \varphi)$ be a non-commutative probability space, let $\chi : \{1, \ldots, n\} \to \{\ell, r\}$, let $\{[Z_{k; i,j}]\}_{k=1}^n \subseteq \M_d(\A)$, and let $Z_k = L([Z_{k; i,j}])$ if $\chi(k) = \ell$ and $Z_k = R([Z_{k; i,j}])$ if $\chi(k) = r$.  Suppose for all $n\geq 1$, $\omega : \{1,\ldots, n\} \to I \sqcup J$, $\chi = \chi_\omega$, and $1 \leq i_1, \ldots, i_n, j_1, \ldots, j_n\leq d$ with 
\[
j_{s_\chi(1)} = i_{s_\chi(2)}, j_{s_\chi(2)} = i_{s_\chi(3)}, \ldots, j_{s_\chi(n-1)} = i_{s_{\chi(n)}}
\]
that
\begin{align}
j_{s_\chi(n)} \neq i_{s_\chi(1)} \quad \text{ implies } \quad \kappa^\bC_{\chi}(Z_{\omega(1); i_1, j_1}, \ldots, Z_{\omega(n); i_n, j_n}) = 0. \label{hyp:cumulant}
\end{align}
Then for all $n\geq 1$, $\omega : \{1,\ldots, n\} \to I \sqcup J$, $\chi = \chi_\omega$, and
\[
\left\{\Lambda_k = \diag \left(\lambda^{(k)}_1, \ldots, \lambda^{(k)}_{d}\right)\right\}^{n}_{k=1}  \cup \left\{\Gamma_k = \diag\left(\gamma^{(k)}_1, \ldots, \gamma^{(k)}_{d}\right)\right\}^{n}_{k=1}\subseteq \D_d,
\]
if
\[
Z'_{\omega, k} =  \left\{
\begin{array}{ll}
L(\Lambda_k) Z_{\omega(k)} L(\Gamma_k)  & \mbox{if } \omega(k) \in I  \\
R(\Gamma_k) Z_{\omega(k)} R(\Lambda_k)  & \mbox{if } \omega(k) \in J
\end{array} \right.
\]
then one has
\[
\kappa^{\D_d}_\chi (Z'_{\omega, 1}, \ldots, Z'_{\omega, n}) = \sum^d_{\substack{i_1, \ldots, i_n = 1 \\ j_1, \ldots, j_n = 1 \\ j_{s_\chi(k)} = i_{s_\chi(k+1)} \forall k \\ j_{s_\chi(n)} = i_{s_{\chi(1)}}}} \left( \prod^n_{q=1} \lambda^{(q)}_{i_q}\right)   \left( \prod^n_{q=1} \gamma^{(q)}_{j_q}\right) \kappa^\bC_{\chi}(Z_{\omega(1); i_1, j_1}, \ldots, Z_{\omega(n); i_n, j_n})  E_{i_{s_\chi(1)}, i_{s_\chi(1)}}  .
\]
\end{lem}
\begin{proof}
Let $\X$ denote the complex linear span of
\[
\{L_{b_1} Z_i L_{b_2} \, \mid \, i \in I, b_1, b_2 \in \D_d\} \cup \{R_{b_1} Z_j R_{b_2} \, \mid \, j \in J, b_1, b_2 \in \D_d\}.
\]
For every $n \geq 1$, $\chi : \{1,\ldots, n\} \to \{\ell, r\}$, and $\pi \in BNC(\chi)$ define $\bC$-multi-linear functionals $f_\pi, g_\pi : \X^n \to \D_d$ as follows:  if $\omega : \{1,\ldots, n\} \to I \sqcup J$ is such that $\chi_\omega = \chi$, if
\[
\left\{\Lambda_k = \diag\left(\lambda^{(k)}_1, \ldots, \lambda^{(k)}_{d}\right)\right\}^{n}_{k=1}  \cup \left\{\Gamma_k = \diag\left(\gamma^{(k)}_1, \ldots, \gamma^{(k)}_{d}\right)\right\}^{n}_{k=1}\subseteq \D_d,
\]
and if
\[
Z'_{\omega, k} =  \left\{
\begin{array}{ll}
L(\Lambda_k) Z_{\omega(k)} L(\Gamma_k)  & \mbox{if } \omega(k) \in I  \\
R(\Gamma_k) Z_{\omega(k)} R(\Lambda_k)  & \mbox{if } \omega(k) \in J
\end{array} \right.
\]
then
\[
f_\pi(Z'_{\omega, 1}, \ldots, Z'_{\omega, n}) = \kappa^{\D_d}_\pi (Z'_{\omega,1}, \ldots, Z'_{\omega, n})
\]
and
\[
g_\pi(Z'_{\omega, 1}, \ldots, Z'_{\omega, n}) = \sum^d_{\substack{i_1, \ldots, i_n = 1 \\ j_1, \ldots, j_n = 1 \\ j_{s_\chi(k)} = i_{s_\chi(k+1)} \forall k \\ j_{s_\chi(n)} = i_{s_{\chi(1)}}}} \left( \prod^n_{q=1} \lambda^{(q)}_{i_q}\right)   \left( \prod^n_{q=1} \gamma^{(q)}_{j_q}\right) \kappa^\bC_{\pi}(Z_{\omega(1); i_1, j_1}, \ldots, Z_{\omega(n); i_n, j_n})  E_{i_{s_\chi(1)}, i_{s_\chi(1)}}.
\]

Clearly $f_\pi$ is a $\D_d$-valued bi-multiplicative function.  Furthermore, it is possible to verify that $g_\pi$ has the properties of a $\D_d$-valued bi-multiplicative function when restricted to entries of the form $L(\Lambda_k) Z_{\omega(k)} L(\Gamma_k)$ and $R(\Gamma_k) Z_{\omega(k)} R(\Lambda_k)$.  The verification of parts (\ref{part:bi-multi-1}) and   (\ref{part:bi-multi-2}) of Definition \ref{defn:bi-multiplicative} follows immediately once one realizes that each possible operation either moves an element of $\D_d$ from begin an element where the $j_{s_\chi(k)}$-term along the diagonal contributes in the product to an element where the $i_{s_\chi(k+1)}$-term along the diagonal contributes in the product (or vice versa), or leaves the contributing term alone.  For parts (\ref{part:bi-multi-3}) and (\ref{part:bi-multi-4}) of Definition \ref{defn:bi-multiplicative}, note that if $V$ is a union of blocks of $\pi$ and a $\chi$-interval, then we may write $V = \{s_\chi(m), s_\chi(m+1), \ldots, s_\chi(m+m')\}$ for some $m, m'$.  If $j_{\chi(m+m')} \neq i_{\chi(m+m')}$ then
\[
\kappa^\bC_{\pi}(Z_{\omega(1); i_1, j_1}, \ldots, Z_{\omega(n); i_n, j_n}) = 0
\]
by hypothesis (\ref{hyp:cumulant}) together with the bi-multiplicative properties of $\kappa^\bC$.  Thus such terms may be removed from the expression of $g_\pi(Z'_{\omega, 1}, \ldots, Z'_{\omega, n})$.    Consequently, the computation of $g_\pi(Z'_{\omega, 1}, \ldots, Z'_{\omega, n})$ will agree with each of the additional expressions in part (\ref{part:bi-multi-3}) and  (\ref{part:bi-multi-4}) of Definition \ref{defn:bi-multiplicative} as 
\[
\kappa^\bC_{\pi}(Z_{\omega(1); i_1, j_1}, \ldots, Z_{\omega(n); i_n, j_n}) = \kappa^\bC_{\pi|_{V^c}}((Z_{\omega(1); i_1, j_1}, \ldots, Z_{\omega(n); i_n, j_n})|_{V^c})\kappa^\bC_{\pi|_V}((Z_{\omega(1); i_1, j_1}, \ldots, Z_{\omega(n); i_n, j_n})|_V)
\]
(and, in part (\ref{part:bi-multi-4}), in each of those expressions one of the $\Lambda_k$ or $\Gamma_k$ is multiplied by $g_{\pi|_V}((Z'_{\omega, 1}, \ldots, Z'_{\omega, n})|_{V})$ and thus will appear in the $\left( \prod^n_{q=1} \lambda^{(q)}_{i_q}\right)   \left( \prod^n_{q=1} \gamma^{(q)}_{j_q}\right)$ term) yields the result.

The final result will follows by M\"{o}bius inversion and linearity provided it can be demonstrated that for all $n\geq 1$, $\omega : \{1,\ldots, n\} \to I \sqcup J$, $\chi = \chi_\omega$, and $1 \leq i_1, \ldots, i_n, j_1, \ldots, j_n \leq d$ that if
\[
Z'_{\omega, k} =  \left\{
\begin{array}{ll}
L(E_{i_k, i_k}) Z_{\omega(k)} L(E_{j_k, j_k})  & \mbox{if } \omega(k) \in I  \\
R(E_{j_k, j_k}) Z_{\omega(k)} R(E_{i_k, i_k})  & \mbox{if } \omega(k) \in J
\end{array} \right.
\]
then 
\[
\sum_{\pi \in BNC(\chi)} g_\pi(Z'_{\omega, 1}, \ldots, Z'_{\omega, n}) = \sum_{\pi \in BNC(\chi)} f_\pi(Z'_{\omega, 1}, \ldots, Z'_{\omega, n}).
\]
If $\delta_{i,j}$ denote the Kronecker delta function, then by Lemma \ref{lem:expanding-moments-of-matrices} together with the definition of $g_\pi$ and $f_\pi$ we obtain that
\begin{align*}
\sum_{\pi \in BNC(\chi)} & g_\pi(Z'_{\omega, 1}, \ldots, Z'_{\omega, n}) \\
&= \delta_{j_{s_\chi(1)}, i_{s_\chi(2)}} \cdots  \delta_{j_{s_\chi(n-1)}, i_{s_\chi(n)}}   \delta_{j_{s_\chi(n)}, i_{s_\chi(1)}}  \left( \sum_{\pi \in BNC(\chi)}  \kappa^\bC_\pi (Z_{\omega(1); i_1, j_1}, \ldots,  Z_{\omega(n); i_n, j_n}) \right)   E_{i_{s_\chi(1)}, i_{s_\chi(1)}} \\
&= \delta_{j_{s_\chi(1)}, i_{s_\chi(2)}} \cdots  \delta_{j_{s_\chi(n-1)}, i_{s_\chi(n)}}   \delta_{j_{s_\chi(n)}, i_{s_\chi(1)}}  \varphi(Z_{\omega(1); i_1, j_1} \cdots  Z_{\omega(n); i_n, j_n})    E_{i_{s_\chi(1)}, i_{s_\chi(1)}} \\
&= (F \circ E_d)(Z'_{\omega, 1} \cdots Z'_{\omega, n}) \\
&= \sum_{\pi \in BNC(\chi)} \kappa^{\D_d}_\pi(Z'_{\omega, 1}, \ldots, Z'_{\omega, n}) \\
&= \sum_{\pi \in BNC(\chi)} f_\pi(Z'_{\omega, 1}, \ldots, Z'_{\omega, n})
\end{align*}
as desired.
\end{proof}

Using Lemma \ref{lem:cumulant-for-R-cyclic-diagonal}, we are enable to construct examples of matrices that are bi-free from scalar matrices over the diagonal scalar matrices via the following result.  This result will make use of Theorem \ref{thm:bi-free-over-D} and is the bi-free analogue of \cite{NSS2002-R}*{Theorem 8.2}.

\begin{thm}
\label{thm:R-cyc}
Let $(\A, \varphi)$ be a non-commutative probability space and let $\{[Z_{k; i,j}]\}_{k \in I} \cup \{[Z_{k; i,j}]\}_{k \in J} \subseteq \M_d(\A)$.  Then  $(\{[Z_{k; i,j}]\}_{k \in I},  \{[Z_{k; i,j}]\}_{k \in J})$ is $R$-cyclic if and only if $(\{L([Z_{k; i,j}])\}_{k \in I},  \{R([Z_{k; i,j}])\}_{k \in J})$ is bi-free from $(L(\M_d(\bC)), R(\M_d(\bC)^\op))$ with amalgamation over $\D_d$ with respect to $F \circ E_d$.
\end{thm}
\begin{proof}
For $k \in I$  let $Z_k = L([Z_{k;i,j}])$ and for $k \in J$ let $Z_k = R([Z_{k;i,j}])$.  Let $Z = \{Z_k\}_{k \in I} \sqcup \{Z_k\}_{k \in J}$.

For one direction, suppose $(\{[Z_{k;i,j}]\}_{k \in I},  \{[Z_{k;i,j}]\}_{k \in J})$ is $R$-cyclic.
Thus it suffices by Theorem \ref{thm:bi-free-over-D} to show that
\[
\kappa^{\M_d(\bC)}_{Z, \omega}(b_1, \ldots, b_{n-1}) =  \kappa^{\D_d}_{Z, \omega}(F(b_1), \ldots, F(b_{n-1}))  
\]
for all $n\geq 1$, $\omega : \{1,\ldots, n\} \to I \sqcup J$, and $b_1, \ldots, b_{n-1} \in \M_d(\bC)$.   Only the case where $\omega(\{1,\ldots, n\}) \cap I \neq \emptyset$ and $\omega(\{1,\ldots, n\}) \cap J \neq \emptyset$ will be shown here as the proofs of the other cases are within (or, alternatively, hold by \cite{NSS2002-R}*{Theorem 8.2}).

Fix $n\geq 1$, $\omega : \{1,\ldots, n\} \to I \sqcup J$, and $b_1, \ldots, b_{n-1} \in \M_d(\bC)$.  Let $k_\ell = \min\{k \, \mid \, \omega(k) \in I\}$ and $k_r = \min\{k \, \mid \, \omega(k) \in J\}$.  Then $\{k_\ell, k_r\} = \{1, \hat{k}\}$ for some $\hat{k}$.  We must show that  if 
\[
C^{\omega(k)}_{b} = \left\{
\begin{array}{ll}
L_{b} & \mbox{if } \omega(k) \in I \\
R_{b} & \mbox{if } \omega(k) \in J
\end{array} \right. ,
\]
then
\[
\kappa^{\M_d(\bC)}_{\chi_\omega}\left(Z_{\omega(1)}, C^{\omega(2)}_{b_1} Z_{\omega(2)}, \ldots, C^{\omega(\hat{k}-1)}_{b_{\hat{k}-2}} Z_{\omega(\hat{k}-1)}, Z_{\hat{k}}, C^{\omega(\hat{k}+1)}_{b_{\hat{k}-1}} Z_{\omega(\hat{k}+1)}, \ldots, C^{\omega(n-1)}_{b_{n-3}} Z_{\omega(n-1)}, C^{\omega(n)}_{b_{n-2}} Z_{\omega(n)} C^{\omega(n)}_{b_{n-1}}\right)
\]
and
\[
\kappa^{\D_d}_{\chi_\omega}\left(Z_{\omega(1)}, C^{\omega(2)}_{F(b_1)} Z_{\omega(2)}, \ldots, \ldots, C^{\omega(n-1)}_{F(b_{n-3})} Z_{\omega(n-1)}, C^{\omega(n)}_{F(b_{n-2})} Z_{\omega(n)} C^{\omega(n)}_{F(b_{n-1})}\right)
\]
agree.  Note a formula for the later expressions was obtained in Lemma \ref{lem:cumulant-for-R-cyclic-diagonal}.

Let $Z'_k$ denote the $k^\th$ term in the $\kappa^{\M_d(\bC)}_{\chi_\omega}$-expression and let $\chi = \chi_\omega$.  By Corollary \ref{cor:cumulant-for-R-cyclic}, we have  that if $Z'_k = L([Z'_{k; i,j}])$ if $\chi(k) = \ell$ and $Z'_k = R([Z'_{k; i,j}])$ if $\chi(k) = r$, then
\[
\kappa^{\M_d(\bC)}_\chi (Z'_1,  \ldots, Z'_n) = \sum^d_{\substack{i_1, \ldots, i_n = 1 \\ j_1, \ldots, j_n = 1}} \kappa^\bC_\chi (Z'_{1; i_1, j_1}, \ldots,  Z'_{n; i_n, j_n}) E_\chi((i_1, \ldots, i_n), (j_1, \ldots, j_n)).
\]
However, notice $E_\chi((i_1, \ldots, i_n), (j_1, \ldots, j_n)) = 0$ unless $j_{s_\chi(1)} = i_{s_\chi(2)}$, $j_{s_\chi(2)} = i_{s_\chi(3)}$, $\ldots$, $j_{s_\chi(n-1)} = i_{s_{\chi(n)}}$, so 
\[
\kappa^{\M_d(\bC)}_\chi (Z'_1,  \ldots, Z'_n) = \sum^d_{\substack{i_1, \ldots, i_n = 1 \\ j_1, \ldots, j_n = 1 \\ j_{s_\chi(1)} = i_{s_\chi(2)}, \ldots, j_{s_\chi(n-1)} = i_{s_{\chi(n)}}}} \kappa^\bC_\chi (Z'_{1; i_1, j_1}, \ldots,  Z'_{n; i_n, j_n}) E_{i_{s_\chi(1)}, j_{s_{\chi(n)}}}.
\]

For a fixed sequence $1 \leq i_1, \ldots, i_n, j_1, \ldots, j_n\leq d$ such that $j_{s_\chi(1)} = i_{s_\chi(2)}, \ldots, j_{s_\chi(n-1)} = i_{s_{\chi(n)}}$, we desire to describe 
\[
\kappa^\bC_\chi (Z'_{1; i_1, j_1}, \ldots,  Z'_{n; i_n, j_n}) E_{i_{s_\chi(1)}, j_{s_{\chi(n)}}}.
\]

Suppose $Z'_k = L_b Z_q$ for some $q \in I$ occurs in the $\kappa_\chi^{\M_d(\bC)}$-cumulant. Thus $s^{-1}_\chi(k) > 1$ (i.e. if $s^{-1}_\chi(k) = 1$, then $Z'_k$ is the first left element in the sequence and would need to be $Z_q$ for $q \in I$).  Write $b = [a_{i,j}]$ so that $Z'_k = L([\sum^d_{m=1} a_{i,m} Z_{q; m,j}])$.  In addition, we have that $k_0 = s_\chi(s^{-1}_\chi(k) - 1)$ exists  and $Z'_{k_0}$ is of the form $Z_p$ for some $p \in I$ or of the form $L_{b'} Z_p$ for some $p \in I$.  Consequently, expanding out $Z'_k$ in the $\kappa^{\M_d(\bC)}_\chi$-cumulant produces elements of the form $a_{i_k,m} Z_{q; m, j_k}$ in the $k^\th$ entry of the $\kappa^\bC_{\chi}$-cumulant and expanding out $Z'_{k_0}$ produces scalar multiplies of $Z_{p; m', j_{k_0}}$ in the $k_0^\th$ entry of the $\kappa^\bC_{\chi}$-cumulant.  Since $j_{k_0} = j_{s_\chi(s^{-1}_\chi(k) - 1)} = i_{s_\chi(s^{-1}_\chi(k) - 1 + 1)} = i_k$, the $R$-cyclic condition would imply the only non-zero terms occur when $m = i_k$.  Thus $Z_{q; i_k, j_k}$ will be left in the $\kappa^\bC_\chi$-cumulant and the complex scalar $a_{i_k, i_k}$ will be pulled out of the $\kappa^\bC_{\chi}$-cumulant.

Similarly, suppose $Z'_k = R_b Z_q$ for some $q \in J$ occurs in the $\kappa_\chi^{\M_d(\bC)}$-cumulant. Thus $s^{-1}_\chi(k) < n$ (i.e. if $s^{-1}_\chi(k) = n$, then $Z'_k$ is the first right element in the sequence and would need to be $Z_q$ for $q \in J$).  Write $b = [a_{i,j}]$ so that $Z'_k = R([a_{i,j}])R([Z_{q; i,j}]) =  R([\sum^d_{m=1} Z_{q; i,m} a_{m,j}])$.  In addition, we have that $k_0 = s_\chi(s^{-1}_\chi(k) + 1)$ exists and $Z'_{k_0}$ is of the form $Z_p$ for some $p \in J$ or of the form $R_{b'} Z_p$ for some $p \in I$.  Consequently, expanding out $Z'_k$ in the $\kappa^{\M_d(\bC)}_{\chi}$-cumulant produces elements of the form $Z_{q; i_k,m} a_{m,j_k}$ in the $k^\th$ entry of the $\kappa^{\bC}_{\chi}$-cumulant  and expanding out $Z'_{k_0}$ produces scalar multiplies of $Z_{p; i_{k_0}, m'}$ in the $k_0^\th$ entry of the $\kappa^{\bC}_{\chi}$-cumulant.  Since $i_{k_0} = i_{s_\chi(s^{-1}_\chi(k) + 1)} = j_{s_\chi(s^{-1}_\chi(k) + 1 - 1)} = j_k$, the $R$-cyclic condition would imply the only non-zero terms occur when $m = j_k$.  Thus $Z_{q; i_k, j_k}$ will be left in the $\kappa^\bC_\chi$-cumulant and the complex scalar $a_{j_k, j_k}$ will be pulled out of the $\kappa^\bC_{\chi}$-cumulant.

To deal with a term of the form $L_{b_1} Z_q L_{b_2}$ for some $q \in I$, use the same arguments as in the  $L_b Z_q$ proof to deal with $b_1$ and use the same arguments as in the $R_b Z_q$ proof to deal with $b_2$.  Similarly, to deal with a term of the form $R_{b_1} Z_q R_{b_2}$ for some $q \in J$, use the same arguments as in the  $L_b Z_q$ proof to deal with $b_2$ and use the same arguments as in the $R_b Z_q$ proof to deal with $b_1$.  Consequently, one obtains that if $b_k = [a_{k;i,j}]$ for $k \in \{2, \ldots, n+1\} \setminus \{\hat{k}\}$, then 
\[
\kappa^{\M_d(\bC)}_{\chi_\omega}\left(Z_{\omega(1)}, C^{\omega(2)}_{b_1} Z_{\omega(2)}, \ldots, C^{\omega(\hat{k}-1)}_{b_{\hat{k}-2}} Z_{\omega(\hat{k}-1)}, Z_{\hat{k}}, C^{\omega(\hat{k}+1)}_{b_{\hat{k}-1}} Z_{\omega(\hat{k}+1)}, \ldots, C^{\omega(n-1)}_{b_{n-3}} Z_{\omega(n-1)}, C^{\omega(n)}_{b_{n-2}} Z_{\omega(n)} C^{\omega(n)}_{b_{n-1}}\right)
\]
equals
\begin{align} \label{sum:annoying-R-cyclic-sum}
\sum^d_{\substack{i_1, \ldots, i_n = 1 \\ j_1, \ldots, j_n = 1 \\ j_{s_\chi(1)} = i_{s_\chi(2)}, \ldots, j_{s_\chi(n-1)} = i_{s_{\chi(n)}}}} b^\chi_{((i_1, \ldots, i_n), (j_1, \ldots, j_n))} \kappa^\bC_\chi \left(Z_{\omega(1); i_1, j_1}, \ldots,  Z_{\omega(n); i_n, j_n} \right)E_{i_{s_\chi(1)}, j_{s_{\chi(n)}}}
\end{align}
where
\[
b^\chi_{((i_1, \ldots, i_n), (j_1, \ldots, j_n))} = \left\{
\begin{array}{ll}
a_{n+1; j_n, j_n}  \left(\prod_{k, \chi(k) = \ell, k \neq n} a_{k; i_k, i_k} \right)\left(\prod_{k, \chi(k) = r} a_{k; j_k, j_k} \right) & \mbox{if } \chi(n) = \ell  \\
a_{n+1; j_n, j_n}   \left(\prod_{k, \chi(k) = \ell} a_{k; i_k, i_k} \right) \left(\prod_{k, \chi(k) = r, k \neq n} a_{k; j_k, j_k} \right)& \mbox{if } \chi(n) = r
\end{array} \right. .
\]
Furthermore, (\ref{sum:annoying-R-cyclic-sum}) must equal
\[
\sum^d_{\substack{i_1, \ldots, i_n = 1 \\ j_1, \ldots, j_n = 1 \\ j_{s_\chi(1)} = i_{s_\chi(2)}, \ldots, j_{s_\chi(n-1)} = i_{s_{\chi(n)}} \\ j_{s_\chi(n)} = i_{s_\chi(1)} }} b^\chi_{((i_1, \ldots, i_n), (j_1, \ldots, j_n))}  \kappa^\bC_\chi \left(Z_{\omega(1); i_1, j_1}, \ldots,  Z_{\omega(n); i_n, j_n} \right) E_{i_{s_\chi(1)}, i_{s_{\chi(1)}}}
\]
due to the $R$-cyclic condition.  Since this expression agrees with the expression in Lemma \ref{lem:cumulant-for-R-cyclic-diagonal} under closer examination, this direction of the proof is complete.

For the other direction, suppose $(\{L([Z_{k;i,j}])\}_{k \in I},  \{R([Z_{k;i,j}])\}_{k \in J})$ is bi-free from $(L(\M_d(\bC)), R(\M_d(\bC)^\op))$ with amalgamation over $\D_d$ with respect to $F \circ E_d$.

Note it is always possible to construct operators in some non-commutative probability space with any desired $(\ell, r)$-cumulants since one can always abstractly define pairs of families, or one can appeal to the bi-free operator model in \cite{CNS2015-1}*{Section 6}.
Consequently, we may construct another non-commutative probability space $(\A_0, \varphi_0)$ with operators 
\[
\{W_{k;i,j} \, \mid \, 1 \leq i,j \leq d, k \in I \sqcup J\} \subseteq \A_0
\]
such that for every $n \geq 1$, every $\omega : \{1,\ldots, n\} \to I \sqcup J$, and every $1 \leq i_1, \ldots, i_n, j_1, \ldots, j_n\leq d$, we have
\[
\kappa^\bC_{\chi_\omega}(W_{\omega(1); i_1, j_1}, \ldots, W_{\omega(n); i_n, j_n}) = 0
\]
if at least one of $j_{s_\chi(1)} = i_{s_\chi(2)}$, $j_{s_\chi(2)} = i_{s_\chi(3)}$, $\ldots$, $j_{s_\chi(n-1)} = i_{s_{\chi(n)}}$, $j_{s_\chi(n)} = i_{s_\chi(1)}$ fails, and otherwise if 
\[
Z'_{\omega, k} =  \left\{
\begin{array}{ll}
L(E_{i_k, i_k}) Z_{\omega(k)} L(E_{j_k, j_k})  & \mbox{if } \omega(k) \in I  \\
R(E_{j_k, j_k}) Z_{\omega(k)} R(E_{i_k, i_k})  & \mbox{if } \omega(k) \in J
\end{array} \right.
\]
then 
\[
\kappa^\bC_{\chi_\omega}(W_{\omega(1); i_1, j_1}, \ldots, W_{\omega(n); i_n, j_n}) := (i_{s_\chi(1)}, i_{s_\chi(1)})\text{-entry of } \kappa^{\D_d}_{\chi_\omega}(Z'_{\omega, 1}, \ldots, Z'_{\omega, n}).
\]
Note the occurrences of $E_{i_{s_\chi(1)}, i_{s_\chi(1)}}$ and $E_{j_{s_\chi(n)}, j_{s_\chi(n)}}$ in the definition of the $Z'_{\omega, k}$ may be removed without affecting the value of the $(i_{s_\chi(1)}, i_{s_\chi(1)})$-entry.  Furthermore, if one uses properties (\ref{part:bi-multi-1}) and (\ref{part:bi-multi-2}) of Definition \ref{defn:bi-multiplicative}, one may pair up the remaining $2n - 2$ elements of the form $E_{i,i}$ and $E_{j,j}$ to obtain $n-1$ elements of $\D_d$ in such a way that 
\[
\kappa^{\D_d}_{\chi_\omega}(Z'_{\omega, 1}, \ldots, Z'_{\omega, n}) = E_{i_{s_\chi(1)}, i_{s_\chi(1)}} \kappa^{\D_d}_{Z, \omega}(b_1, \ldots, b_{n-1})    E_{j_{s_\chi(n)}, j_{s_\chi(n)}}
\]
for some matrix units $b_1, \ldots, b_{n-1} \in \D_d$ (i.e. if $b_m$ multiplies $Z_{\omega(k)}$ on the left in the definition of $\kappa^{\D_d}_{Z, \omega}(b_1, \ldots, b_{n-1})   $ with $\omega(k) \in I$ (respect $\omega(k) \in J$), then $b_m = E_{i_k, i_k}$ (respectively $b_m = E_{j_k, j_k})$), and if $b_m$ multiplies $Z_{\omega(n)}$ on the right with $\omega(k) \in I$ (respect $\omega(k) \in J$), then $b_m = E_{j_n, j_n}$ (respectively $b_m = E_{i_n, i_n})$)).

By construction $(\{[W_{k;i,j}]\}_{k \in I}, \{W_{k;i,j}\}_{k \in J})$ is an $R$-cyclic pair.  For $k \in I$  let $W_k = L([Z_{k;i,j}])$ and for $k \in J$ let $W_k = R([Z_{k;i,j}])$.  If $W = \{W_k\}_{k \in I} \sqcup \{W_k\}_{k \in J}$, then for all $n \geq 1$, $\omega : \{1,\ldots, n\} \to I \sqcup J$, and for all $b_1, \ldots, b_{n-1} \in \D_d$, 
\[
\kappa^{\D_d}_{W, \omega}(b_1, \ldots, b_{n-1}) = \kappa^{\D_d}_{Z, \omega}(b_1, \ldots, b_{n-1}).
\]
Indeed, it suffices to prove the above equation when each $b$ is a diagonal matrix unit by linearity.  Using Lemma \ref{lem:cumulant-for-R-cyclic-diagonal} gives an expression for $\kappa^{\D_d}_{W, \omega}(b_1, \ldots, b_{n-1})$  in terms of various $\kappa^\bC_{\chi_\omega}(W_{\omega(1); i_1, j_1}, \ldots, W_{\omega(n); i_n, j_n})$.  Subsequently, adding over all possible $i_{s_\chi(1)}$ produces the diagonal matrix $ \kappa^{\D_d}_{Z, \omega}(b_1, \ldots, b_{n-1})$.  Hence $W$ and $Z$ have the same $\D_d$-valued distribution.

By the first part of this proof, $(\{L([W_{k;i,j}])\}_{k \in I},  \{R([W_{k;i,j}])\}_{k \in J})$ is bi-free from $(L(\M_d(\bC)), R(\M_d(\bC)^\op))$ with amalgamation over $\D_d$.  Hence $W$ and $Z$ have the same $\M_d(\bC)$-valued distributions and thus the same $\M_d(\bC)$-valued cumulants.

Finally, for $n \geq 1$, $\omega : \{1,\ldots, n\} \to I \sqcup J$, and $1 \leq i_1, \cdots, i_n, j_1, \ldots, j_n \leq d$, we note by Corollary \ref{cor:cumulant-for-R-cyclic} that if
\[
Z'_{\omega, k} =  \left\{
\begin{array}{ll}
L(E_{1, i_k}) Z_{\omega(k)} L(E_{j_k, 1})  & \mbox{if } \omega(k) \in I  \\
R(E_{1, j_k}) Z_{\omega(k)} R(E_{i_k, 1})  & \mbox{if } \omega(k) \in J
\end{array} \right.  \qand
W'_{\omega, k} =  \left\{
\begin{array}{ll}
L(E_{1, i_k}) W_{\omega(k)} L(E_{j_k, 1})  & \mbox{if } \omega(k) \in I  \\
R(E_{1, j_k}) W_{\omega(k)} R(E_{i_k, 1})  & \mbox{if } \omega(k) \in J
\end{array} \right. ,
\]
then Corollary \ref{cor:cumulant-for-R-cyclic} implies
\begin{align*}
\kappa^{\bC}_{\chi_\omega}(Z_{\omega(1); i_1, j_1}, \ldots, Z_{\omega(n); i_n, j_n}) &=
 (1, 1)\text{-entry of }\kappa^{\M_d(\bC)}_{\chi_\omega}(Z'_{\omega, 1}, \ldots, Z'_{\omega, n}) \\
 &=   (1, 1)\text{-entry of }\kappa^{\M_d(\bC)}_{\chi_\omega}(W'_{\omega, 1}, \ldots, W'_{\omega, n})   \\
 &= \kappa^{\bC}_{\chi_\omega}(W_{\omega(1); i_1, j_1}, \ldots, W_{\omega(n); i_n, j_n}).
\end{align*}
Therefore $(\{[Z_{k;i,j}]\}_{k \in I}, \{Z_{k;i,j}\}_{k \in J})$ is an $R$-cyclic pair as $(\{[W_{k;i,j}]\}_{k \in I}, \{W_{k;i,j}\}_{k \in J})$ is an $R$-cyclic pair by construction.
\end{proof}

\section{The Operator-Valued Partial $R$-Transform}
\label{sec:R-Trans}

As the free operator-valued transforms have been quite useful for analytic arguments, we shall begin a study of operator-valued bi-free partial transformations in this section.  In particular, this section will develop the operator-valued bi-free partial $R$-transform (Theorem \ref{thm:R-transform}) to relate bi-free additive convolution in each entry of the pair of $B$-faces.  However, as alluded to in the introduction, such transformations are not functions of two $B$-elements, but three $B$-elements.  

The reason for this can be seen via the combinatorial proof of the bi-free partial $R$-transformation given in \cite{S2015-1} that will be emulated in this section.  Indeed the proof proceeds by breaking up a joint moment of left and right operators into a sum of cumulants.  One then desires to proceed by summing over certain blocks of bi-non-crossing partition.  When one requires to sum over a block that contains both left and right nodes, the bi-multiplicative properties do not allow for the resulting $B$-operator below the block in the bi-non-crossing diagram to escape (i.e. be multiplied on one side of the result) and the resulting $B$-operator is of a different form than the $B$-operators formed on the left and right sides.   Consequently, a third $B$-variable is required to describe the $B$-operator produced by the bottom of each bi-non-crossing diagram.  Furthermore, when one sums over all such partitions, one obtains the desired expression in this third $B$-variable resulting in the $R$-transform requiring a composition.  Of course, said transformation reduces to the known transformation in the scalar-valued setting, and contains within it the operator-valued free $R$-transforms.

For the remainder of the paper,  $B$ denotes a Banach algebra.  Furthermore $b \in B$ is used for left $B$-operators whereas $d \in B$ is used for right $B$-operators (since, symbolically, $b$ and $d$ are opposites).  Recall we have taken the view that $B^\op$ is really just elements of $B$ where we take $d \mapsto R_d$ is antimultiplicative.  Finally, $c \in B$ will be used for the special third $B$-operator (that is both on the left and the right, so it is sort of $c$entral).

\begin{defn}
Let $B$ be a Banach algebra.  A \emph{Banach $B$-$B$-non-commutative probability space} is a $B$-$B$-non-commutative probability space $(\A, E, \varepsilon)$ such that $\A$ is a Banach algebra, and $E : \A \to B$,  $\varepsilon|_{B \otimes 1_B}$, and $\varepsilon|_{1_B \otimes B^\op}$ are bounded.
\end{defn}

Let $(\A, E, \varepsilon)$ be a  Banach $B$-$B$-non-commutative probability space.  For $X \in \A_\ell$ and $b \in B$, we recall from \cite{V1995} and \cite{S1998} the series:
\begin{align*}
G^\ell_X(b) &:= \sum_{n\geq 0} E((L_bX)^nL_b), \\
R^\ell_X(b) &:= \sum_{n\geq 0} \kappa^B_{\chi_{n+1, 0}}(X, \underbrace{L_b X, \ldots, L_b X}_{n \text{ entries }}),    \\
M^\ell_X(b) &:= 1 + \sum_{n\geq 1} E((L_bX)^n), \text{ and}\\
C^\ell_X(b) &:= 1 + \sum_{n\geq 1} \kappa^B_{\chi_{n, 0}}(L_b X, \ldots, L_b X).
\end{align*}
\begin{rem}
\label{rem:convergence}
Since the above four series are infinite series with entries in $B$, we must discuss their convergence.  Notice if $\left\|b\right\|$ is sufficiently small so that $\left\|L_b\right\| < 1$ (which is the case as $\varepsilon|_{B \otimes 1_B}$ is bounded), then $G^\ell_X(b)$ and $M^\ell_X(b)$ converge absolutely.  Furthermore, the moment-cumulant formula (\ref{eq:cumulants}) imply $C^\ell_X(b)$ and $R^\ell_X(b)$ converge absolutely if $\left\|b\right\|$ is sufficiently small.  Indeed the cumulant corresponding to $\chi : \{1,\ldots, n\} \to \{\ell, r\}$ is a sum of at most $|BNC(\chi)|$ terms of the form $\mu_{BNC}(\pi, 1_\chi)$ for some $\pi\in BNC(\chi)$ times $E^B_\pi$ of a tuple (each of which is bounded by a multiple of $\left\|X\right\|^n \left\|L_b\right\|^n$).  Due to the lattice structure, $|\mu_{BNC}(\pi, 1_\chi)|$ is at most $|BNC(\chi)|$.  Since $|BNC(\chi)| = |NC(n)|$ is the $n^\th$ Catalan number $c_n$ and since $c_n \leq 4^n$, each of the above sums converges absolutely near $b = 0$.
\end{rem}
\begin{rem}
\label{rem:left-formula}
Note the following relations between the above series from \cite{V1995} and \cite{S1998}*{Section 4}:
\begin{align*}
G^\ell_X(b) &= M^\ell_X(b) b,\\
C^\ell_X(b) &= 1 + bR^\ell_X(b), \text{ and}\\
M^\ell_X(b) &= C^\ell_X(M^\ell_X(b)b).
\end{align*}
\end{rem}
Similarly, for $Y \in \A_r$ and $d \in B$, define 
\begin{align*}
G^r_Y(d) &:= \sum_{n\geq 0} E((R_dY)^nR_d), \\
R^r_Y(d) &:= \sum_{n\geq 0} \kappa^B_{\chi_{0, n+1}}(Y, \underbrace{R_d Y, \ldots, R_d Y}_{n \text{ entries }}),    \\
M^r_Y(d) &:= 1 + \sum_{n\geq 1} E((R_dY)^n), \text{ and}\\
C^r_Y(d) &:= 1 + \sum_{n\geq 1} \kappa^B_{\chi_{0, n}}(R_d Y, \ldots, R_d Y).
\end{align*}
\begin{rem}
By the same arguments as in Remark \ref{rem:convergence}, the above series converge absolutely near $d = 0$. Note we are viewing these as series with entries from $B$ with values in $B$ even though $d \mapsto R_d$ is anti-multiplicative.  Consequently, one can verify that
\begin{align*}
G^r_Y(d) &= dM^r_Y(d), \\
C^r_Y(d) &= 1 + R^r_Y(d)d,  \text{ and }\\
M^r_Y(d) &= C^r_Y(d M^r_Y(d))
\end{align*}
(i.e. the same formula as in Remark \ref{rem:left-formula} hold once one uses the opposite multiplication in $B$).
\end{rem}

Finally, for $b, c, d \in B$, $X \in \A_\ell$, and $Y \in \A_r$, define
\begin{align*}
M_{X, Y}(b, c, d) &:= \sum_{n,m\geq 0} E((L_bX)^n (R_dY)^m R_c) \quad \text{ and}\\
C_{X, Y}(b,c,d) &:= c  + \sum_{n \geq 1} \kappa^B_{\chi_{n, 0}}(\underbrace{L_b X, \ldots, L_b X}_{n-1 \text{ entries }}, L_b X L_c)     + \sum_{\substack{m\geq 1 \\ n \geq 0}} \kappa_{\chi_{n,m}}(  \underbrace{L_b X, \ldots, L_b X}_{n \text{ entries }}, \underbrace{R_d Y, \ldots, R_d Y}_{m-1 \text{ entries }}, R_d Y R_c).
\end{align*}

\begin{rem}
Similar arguments to those used in Remark \ref{rem:convergence} show for all $c$ in a bounded set, the above series converge absolutely whenever $b$ and $d$ are sufficiently small.  Furthermore, notice 
\begin{align*}
M_{X, Y}(b, c, 0) &=  M^\ell_X(b)c, \\
M_{X, Y}(0, c, d) &= c M^r_Y(d), \\
C_{X, Y}(b, c, 0) &= C^\ell_X(b) c, \text{ and}\\
C_{X, Y}(0, c, d) &= cC^r_Y(d).
\end{align*}

Finally, if $(X_1, Y_1)$ and $(X_2, Y_2)$ are bi-free over $B$, then, by Theorem \ref{thm:bifree-classifying-theorem}, 
\[
(C_{X_1+X_2, Y_1+Y_2}(b,c,d) - c) = (C_{X_1, Y_1}(b,c,d) - c) + (C_{X_2, Y_2}(b,c,d) - c);
\]
that is, $R_{X, Y}(b,c,d) = C_{X, Y}(b,c,d) - c$ is a  operator-valued bi-free partial $R$-transform.
\end{rem}

The following is the simplest formula we could find to compare these series.

\begin{thm}
\label{thm:R-transform}
Let $(\A, E, \varepsilon)$ be a  Banach $B$-$B$-non-commutative probability space, let $X \in \A_\ell$, and let $Y \in \A_r$.  Then
\begin{align}
M^\ell_X(b)M_{X, Y}(b,c,d) + M_{X, Y}(b,c,d)M^r_Y(d) = M^\ell_X(b)cM^r_Y(d) + C_{X, Y}(M^\ell_X(b)b, M_{X, Y}(b, c, d), d M^r_Y(d)).  
\end{align}
\end{thm}
\begin{rem}
The above produces precisely the bi-free partial $R$-transforms in \cite{V2013} and \cite{S2015-1} in the case that $B = \bC$ by letting $b = z$, $d = w$, and $c = 1$.  The above expression may be simplified in the case that $X \in \A_\ell \cap \A_r$, $Y \in \A_\ell \cap \A_r$, and/or $B$ is commutative.

Note if $d = 0$ is substituted into the above expression, we obtain
\[
M^\ell_X(b)M^\ell_X(b)c + M^\ell_X(b)c = M^\ell_X(b)c + C_{X, Y}(M^\ell_X(b)b, M^\ell_X(b)c, 0).
\]
By analyzing the last term, we see this implies
\[
M^\ell_X(b)M^\ell_X(b)c =  C^\ell_X(M^\ell_X(b)b) M^\ell_X(b)c.
\]
Consequently, taking $c = 1$ and using the fact that $M^\ell_X(b)$ is invertible for sufficiently small $b$, we obtain that $M^\ell_X(b) = C^\ell_X(M^\ell_X(b)b)$; that is, the formula from Remark \ref{rem:left-formula} are recovered from Theorem \ref{thm:R-transform}.
\end{rem}

\begin{proof}[Proof of Theorem \ref{thm:R-transform}]
For $\chi : \{1,\ldots, n\} \to \{\ell, r\}$, let $\bncs(\chi)$ denote the collection of all $\pi \in BNC(\chi)$ such that if $\chi(p) = \ell$ and $\chi(q) = r$, then $p \nsim_\pi q$ (that is, all bi-non-crossing partitions that can be split via a vertical line into two non-crossing partitions; one on the left and one on the right).

For $n, m\geq 1$, using bi-multiplicative properties, we obtain that
\begin{align*}
&E((L_bX)^n (R_dY)^m L_c)\\
 &= \sum_{\pi \in \bncs(\chi_{n,m})} \kappa^B_{\pi}(  \underbrace{L_b X, \ldots, L_b X}_{n \text{ entries }}, \underbrace{R_d Y, \ldots, R_d Y}_{m-1 \text{ entries }}, R_d Y R_c)  \\
& \quad + \sum_{\substack{\pi \in BNC(\chi_{n,m}) \\ \pi \notin \bncs(\chi_{n,m})}} \kappa^B_{\pi}(  \underbrace{L_b X, \ldots, L_b X}_{n \text{ entries }}, \underbrace{R_d Y, \ldots, R_d Y}_{m-1 \text{ entries }}, R_d Y R_c) \\
&= E((L_bX)^n) c E((R_dY)^m) + \sum_{\substack{\pi \in BNC(\chi_{n,m}) \\ \pi \notin \bncs(\chi_{n,m})}} \kappa^B_{\pi}(  \underbrace{L_b X, \ldots, L_b X}_{n \text{ entries }}, \underbrace{R_d Y, \ldots, R_d Y}_{m-1 \text{ entries }}, R_d Y R_c).
\end{align*}
Let 
\[
\Theta_{n,m}(b,c,d) = \sum_{\substack{\pi \in BNC(\chi_{n,m}) \\ \pi \notin \bncs(\chi_{n,m})}} \kappa^B_{\pi}(  \underbrace{L_b X, \ldots, L_b X}_{n \text{ entries }}, \underbrace{R_d Y, \ldots, R_d Y}_{m-1 \text{ entries }}, R_d Y R_c) .
\]
Notice every partition $\pi \in BNC(\chi_{n,m}) \setminus \bncs(\chi_{n,m})$ must have a block $W$ such that 
\[
W \cap \{1, \ldots, n\} \neq \emptyset \qqand W \cap \{n+1, \ldots, n+m\} \neq \emptyset.
\]
Let $V_\pi$ denote the block of $\pi$ with both left and right indices such that $\min(V_\pi)$ is smallest among all blocks $W$ of $\pi$ with both left and right indices.  

Rearrange the summing in $\Theta_{n,m}(b,c,d)$ (which may be done as it converges absolutely) by first choosing $t \in \{1,\ldots, n\}$, $s \in \{1,\ldots, m\}$, and $V \subseteq \{1,\ldots, n+m\}$ such that
\begin{align*}
V_\ell &:= V \cap \{1, \ldots, n\}=\{u_1 < u_2 < \cdots < u_t\} \qand\\
V_r &:= V \cap \{n+1, \ldots, n+m\} = \{v_1 < v_2 < \cdots < v_s\},
\end{align*}
and then sum over all $\pi \in BNC(\chi_{n,m}) \setminus \bncs(\chi_{n,m})$ such that $V_\pi = V$.  If one defines $u_0 = 0$, $v_0=n$, $u_{t+1} = n+1$, and $v_{s+1} = n+m+1$, the fact that $\pi \in BNC(\chi_{n,m}) \setminus \bncs(\chi_{n,m})$ implies if $V_\pi = V$ then no block of $\pi$ contains indices from both intervals $(u_{k_1}, u_{k_1+1})$ and $(u_{k_2}, u_{k_2+1})$ when $k_1 \neq k_2$, from both intervals $(v_{k_1}, v_{k_1+1})$ and $(v_{k_2}, v_{k_2+1})$ when $k_1 \neq k_2$, and from both intervals $(u_{k_1}, u_{k_1+1})$ and $(v_{k_2}, v_{k_2+1})$ unless $k_1 = t$ and $k_2 = s$.  In particular, examining all $\pi$ such that $V_\pi = V$, each $(t+s+1)$-tuple consisting of bi-non-crossing partitions on each of the sets $(u_k, u_{k+1})$ for $k \in \{0, 1, \ldots, t-1\}$, $(v_k, v_{k+1})$ for $k \in \{0,1, \ldots, s-1\}$, and $(u_t, u_{t+1}) \cup (v_s, v_{s+1})$ occurs precisely once.  The following diagram illustrates and example when $V_\pi = \{3_\ell, 6_\ell, 3_r, 6_r\}$.  Note $\pi_1, \ldots, \pi_5$ can be any bi-non-crossing diagram on their nodes.  Further, we really should draw all of the left nodes above all of the right nodes.
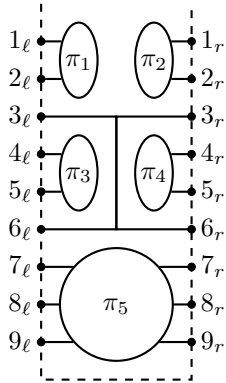
\begin{align*}
\begin{tikzpicture}[baseline]
	\draw[thick, dashed] (-1,4.5) -- (-1,-.5) -- (1,-.5) -- (1,4.5);
	\draw[thick] (-1, 3) -- (0,3) -- (-0,1.5) -- (-1, 1.5);
	\draw[thick] (1, 3) -- (0,3) -- (-0,1.5) -- (1, 1.5);
	\node[left] at (-1, 4) {$1_\ell$};
	\draw[black, fill=black] (-1,4) circle (0.05);	
	\node[left] at (-1, 3.5) {$2_\ell$};
	\draw[black, fill=black] (-1,3.5) circle (0.05);
	\node[left] at (-1, 3) {$3_\ell$};
	\draw[black, fill=black] (-1,3) circle (0.05);
	\node[left] at (-1, 2.5) {$4_\ell$};
	\draw[black, fill=black] (-1,2.5) circle (0.05);
	\node[left] at (-1, 2) {$5_\ell$};
	\draw[black, fill=black] (-1,2) circle (0.05);
	\node[left] at (-1, 1.5) {$6_\ell$};
	\draw[black, fill=black] (-1,1.5) circle (0.05);	
	\node[left] at (-1, 1) {$7_\ell$};
	\draw[black, fill=black] (-1,1) circle (0.05);
	\node[left] at (-1, .5) {$8_\ell$};
	\draw[black, fill=black] (-1,.5) circle (0.05);
	\node[left] at (-1, 0) {$9_\ell$};
	\draw[black, fill=black] (-1,0) circle (0.05);
	\node[right] at (1, 4) {$1_r$};
	\draw[black, fill=black] (1,4) circle (0.05);	
	\node[right] at (1, 3.5) {$2_r$};
	\draw[black, fill=black] (1,3.5) circle (0.05);
	\node[right] at (1, 3) {$3_r$};
	\draw[black, fill=black] (1,3) circle (0.05);
	\node[right] at (1, 2.5) {$4_r$};
	\draw[black, fill=black] (1,2.5) circle (0.05);
	\node[right] at (1, 2) {$5_r$};
	\draw[black, fill=black] (1,2) circle (0.05);
	\node[right] at (1, 1.5) {$6_r$};
	\draw[black, fill=black] (1,1.5) circle (0.05);	
	\node[right] at (1, 1) {$7_r$};
	\draw[black, fill=black] (1,1) circle (0.05);
	\node[right] at (1, .5) {$8_r$};
	\draw[black, fill=black] (1,.5) circle (0.05);
	\node[right] at (1, 0) {$9_r$};
	\draw[black, fill=black] (1,0) circle (0.05);
	\draw[thick] (-.5,3.75) ellipse (.25cm and .5cm);
	\node[below] at (-.5, 3.95) {$\pi_1$};
	\draw[thick] (-1,4) -- (-.73, 4);
	\draw[thick] (-1,3.5) -- (-.73, 3.5);
	\draw[thick] (.5,3.75) ellipse (.25cm and .5cm);
	\node[below] at (.5, 3.95) {$\pi_2$};
	\draw[thick] (1,4) -- (.73, 4);
	\draw[thick] (1,3.5) -- (.73, 3.5);
	\draw[thick] (-.5,2.25) ellipse (.25cm and .5cm);
	\node[below] at (-.5, 2.45) {$\pi_3$};
	\draw[thick] (-1,2.5) -- (-.73, 2.5);
	\draw[thick] (-1,2) -- (-.73, 2);
	\draw[thick] (.5,2.25) ellipse (.25cm and .5cm);
	\node[below] at (.5, 2.45) {$\pi_4$};
	\draw[thick] (1,2.5) -- (.73, 2.5);
	\draw[thick] (1,2) -- (.73, 2);
	\draw[thick] (0,.5) ellipse (.75cm and .75cm);
	\node[below] at (0, .7) {$\pi_5$};
	\draw[thick] (-1, 1) -- (-.57, 1);
	\draw[thick] (-1, 0) -- (-.57, 0);
	\draw[thick] (-1, .5) -- (-.75, .5);
	\draw[thick] (1, 1) -- (.57, 1);
	\draw[thick] (1, 0) -- (.57, 0);
	\draw[thick] (1, .5) -- (.75, .5);
\end{tikzpicture}
\end{align*}

Rearranging the sum  produces
\begin{align*}
\Theta_{n,m}(b,c,d) &= \sum^n_{t=1} \sum^m_{s=1} \sum_{\substack{V \\V_\ell = \{u_1 < u_2 < \cdots < u_t\} \\
V_r = \{v_1 < v_2 < \cdots < v_s\}   }}\sum_{\substack{\pi \in BNC(\chi_{n,m}) \\ \pi \notin \bncs(\chi_{n,m}) \\ V_\pi = V} } \kappa^B_{\pi}(  \underbrace{L_b X, \ldots, L_b X}_{n \text{ entries }}, \underbrace{R_d Y, \ldots, R_d Y}_{m-1 \text{ entries }}, R_d Y R_c).
\end{align*}
Furthermore, using bi-multiplicative properties, the right-most sum in this expression is precisely
\[
\kappa^B_{\chi_{t,s}}\left(L_{b_1}X, \ldots, L_{b_t}X, R_{d_1}  Y, \ldots, R_{d_{s-1}}  Y, R_{d_s}  Y R_{E((L_bX)^{n-u_t}(R_dY)^{m-v_s}R_c) }  \right)
\]
where
\[
b_k = E((L_bX)^{u_k - u_{k-1}-1})b \qqand d_k = d E((R_dY)^{v_k - v_{k-1}-1})
\]
(that is, each $L_bX$ gets multiplied on the left by a $L_{E((L_bX)^{u_k - u_{k-1}-1})}$ for the terms in the interval $(u_{k-1}, u_k)$, each $R_dY$ gets multiplied on the left by $R_{ E((R_dY)^{v_k - v_{k-1}-1})}$ (thereby producing $R_{dE((R_dY)^{v_k - v_{k-1}-1})}$) from the terms in the interval $(v_{k-1}, v_k)$, and all the terms in the intervals $(u_t, u_{t+1})$ and $(v_s, v_{s+1})$ combine to produce a single element of $B$ (and this one will always contain $c$) as all bi-non-crossing partitions are permitted on these indices).  Consequently, we obtain that $\Theta_{n,m}(b,c,d)$ equals
\begin{align}
\sum_{\substack{t \in \{1,\ldots, n\}, s \in \{1,\ldots, m\} \\i_0, i_1, \ldots, i_t \in \{0, 1, \ldots, n\} \\j_0, j_1,\ldots, j_s \in \{0, 1, \ldots, m\}  \\ i_0 + i_1 + \cdots + i_t = n-t \\ j_0 + j_1 + \cdots + j_s = m-s}}  \kappa^B_{\chi_{t,s}}\left(L_{f(i_1)}X, \ldots, L_{f(i_t)}X, R_{g(j_1)}  Y, \ldots, R_{g(j_{s-1})}  Y, R_{g(j_s)}  Y R_{E((L_bX)^{i_0}(R_dY)^{j_0}R_c) }  \right) \label{eq:messy-R}
\end{align}
where
\[
f(k) = E((L_bX)^{k})b \qqand g(k) = d E((R_dY)^{k}).
\]
Note
\[
\sum_{k\geq 0} f(k) = M^\ell_X(b) b \qqand \sum_{k \geq 0} g(k) = d M^r_Y(d).
\]

To complete the proof, expand $M_{X, Y}(b, c, d)$ (using the fact everything converges absolutely so we can rearrange sums as we would like) to obtain
\begin{align*}
M_{X, Y}(b,c,d) &= c + \sum_{n\geq 1} E((L_bX)^n R_c) + \sum_{m\geq 1} E((R_dY)^m R_c) + \sum_{n,m \geq 1} E((L_bX)^n (R_dY)^m R_c) \\
&= \sum_{n,m\geq 0} E((L_bX)^n)cE((R_dY)^m)  + \sum_{n,m\geq 1} \Theta_{n,m}(b,c,d) \\
&= M^\ell_X(b)cM^r_Y(d)  + \sum_{n,m\geq 1} \Theta_{n,m}(b,c,d).
\end{align*}
By rearranging the remaining sum involving $\Theta_{n,m}(b,c,d)$ to sum over all fixed $t, s$ in equation (\ref{eq:messy-R}), (so we sum $f(k)$ and $g(k)$ over all $k\geq 0$ in each entry) we obtain
\begin{align*}
\sum_{n,m\geq 1}& \Theta_{n,m}(b,c,d) \\
&= \sum_{s,t\geq 1} \kappa^B_{\chi_{t, s}}     (  \underbrace{L_{M^\ell_X(b)b} X, \ldots, L_{M^\ell_X(b)b} X}_{t \text{ entries }}, \underbrace{R_{dM^r_Y(d)} Y, \ldots, R_{dM^r_Y(d)} Y}_{s-1 \text{ entries }}, R_{dM^r_Y(d)} Y R_{M_{X, Y}(b,c,d)})     \\
&=  C_{X, Y}(M^\ell_X(b)b, M_{X, Y}(b, c, d), d M^r_Y(d)) \\
& \quad - M_{X, Y}(b, c, d) C^r_Y(d M^r_Y(d)) - C^\ell_X(M^\ell_X(b)b)M_{X, Y}(b,c,d) + M_{X, Y}(b,c,d)        \\
&= C_{X, Y}(M^\ell_X(b)b, M_{X, Y}(b, c, d), d M^r_Y(d))- M_{X, Y}(b, c, d) M^r_Y(d)  - M^\ell_X(b)M_{X, Y}(b,c,d) + M_{X, Y}(b,c,d)
\end{align*}
The result follows by combining these equations.
\end{proof}

\section{The Operator-Valued $S$-Transform Revisited}
\label{sec:free-S}

The goal of this section is to develop a different proof of the operator-valued free $S$-transform formula first produced in \cite{D2006}.   This proof is more along the lines of the `Fourier' transform arguments of \cite{NS1997} and the technology developed will be essential in the coming sections.

Let $(\A, E, \varepsilon)$ be a Banach $B$-$B$-non-commutative probability space, let $X \in \A_\ell$, and $Y \in \A_r$ be such that $E(X)$ and $E(Y)$ are invertible in $B$.  Define
\begin{align*}
\Phi_{\ell, X}(b) &:= C^\ell_X(b) - 1 = \sum_{n\geq 1} \kappa^B_{\chi_{n, 0}}(L_b X, \ldots, L_b X) \\
\Phi_{r, Y}(d) &:= C^r_Y(d) - 1 = \sum_{n\geq 1} \kappa^B_{\chi_{0, n}}(R_d Y, \ldots, R_d Y).
\end{align*}
Again, arguments in Remark \ref{rem:convergence} imply $\Phi_{\ell, X}(b)$ and $\Phi_{r, Y}(d)$ converge absolutely to elements of $B$ when $b$ and $d$ are near 0.  Note that $\Phi_{\ell, X}$ and $\Phi_{r, Y}$, viewed as functions on $B$, are Fr\'{e}chet differentiable on their domains (i.e. they are analytic there).  Clearly $\Phi_{\ell, X}(0) = \Phi_{r, Y}(0) = 0$ and the Fr\'{e}chet differential of $\Phi_{\ell, X}$ and $\Phi_{r, Y}$ at $0$ are the bounded linear maps
\[
c \mapsto c E(X) \qqand c \mapsto E(Y)c
\]
respectively (as the first order cumulants and moments agree).  Since $E(X)$ and $E(Y)$ are invertible in $B$, the Fr\'{e}chet differential of $\Phi_{\ell, X}$ and $\Phi_{r, Y}$ at zero have bounded inverses.  Consequently, the usual Banach space inverse function theorem implies there are neighbourhoods $U_X$, $U_Y$, $V_X$, and $V_Y$ of zero such that $U_X$ is in the domain of $\Phi_{\ell, X}$, $U_Y$ is in the domain of $\Phi_{r, Y}$, $\Phi_{\ell, X}$ is a homeomorphism from $U_X$ to $V_X$, and $\Phi_{r, Y}$ is a homeomorphism from $U_Y$ to $V_Y$.  Therefore there exist functions $\Phi^\inv_{\ell, X}$ and $\Phi^\inv_{r, Y}$ defined on neighbourhoods of $0$ in $B$ that are the compositional inverses of $\Phi_{\ell, X}$ and $\Phi_{r, Y}$ near zero respectively.  

We desire to consider the functions $b \mapsto b^{-1} \Phi^\inv_{\ell, X}(b)$ and $d \mapsto \Phi^\inv_{r, Y}(d) d^{-1}$ when $b$ and $d$ are near zero.  Of course $b^{-1}$ and $d^{-1}$ do not exist near zero, but the following (which is pretty much a carbon copy of \cite{D2006}*{Lemma 2.1}) implies these functions do indeed make sense.  In the remainder of the paper, there are on occasions when we will write $b^{-1}$ for $b \in B$ not necessarily invertible.  All such expressions will be valid via argument similar to the following, or because the $b^{-1}$ is really multiplied by $b$ in disguise and we are writing $b^{-1}$ only for notational convenience.  
\begin{lem}
\label{lem:divide}
Assuming $E(X)$ is invertible, there exists an open neighbourhood of $0$ in $B$ and a unique analytic $B$-valued function $\theta_X$ defined there  such that $\Phi^\inv_{\ell, X}(b) = b \theta_X(b)$.  Similarly, assuming $E(Y)$ is invertible, there exists an open neighbourhood of $0$ in $B$ and a unique analytic $B$-valued function $\theta_Y$ defined there such that  $\Phi^\inv_{r, Y}(d) = \theta_Y(d) d$.
\end{lem}
\begin{proof}
Note uniqueness is clear by taking power series expansions about zero.  To construct $\theta_X$, note that 
\[
\Phi_{\ell, X}(b) = b \phi_{\ell, X}(b) \quad \text{ where } \quad  \phi_{\ell, X}(b) = \sum_{n\geq 1} \kappa^B_{\chi_{n, 0}}(X, \underbrace{L_b X, \ldots, L_b X}_{n-1 \text{ copies}}).
\]
Hence we desire $\theta_X$ such that $b = \Phi_{\ell, X}(b \theta_X(b)) = b \theta_X(b) \phi_{\ell, X}(b \theta_X(b))$.   Hence it suffices to find $\theta_X$ so that
\begin{align}
\theta_X(b) \phi_{\ell, X}(b \theta_X(b)) = 1.   \label{eq:check-can-divide-by-b}
\end{align}
The existence of $\theta_X$ then follows from an application of the implicit function theorem for functions between Banach spaces (see \cite{HG1927} and \cite{G1935}*{page 655}) to the function $f(b_1, b_2) = b_2 \varphi_{\ell, X}(b_1b_2) - 1$.  Indeed $\theta_X(0) = E(X)^{-1}$ is a solution of (\ref{eq:check-can-divide-by-b}) at $b = 0$ and the Fr\'{e}chet differential of the function $x \mapsto x \phi_{\ell, X}(bx)$ at $b = 0$ is the map $c \mapsto c E(X)$, which has bounded inverse.  

The proof for $Y$ is near identical (just multiply in two elements of $B$ using the opposite order).
\end{proof}

Using the above, we many now define the operator-valued free $S$-transform.

\begin{defn}
\label{defn:op-free-S}
Let $(\A, E, \varepsilon)$ be a Banach $B$-$B$-non-commutative probability space, let $X \in \A_\ell$, and $Y \in \A_r$ be such that $E(X)$ and $E(Y)$ are invertible in $B$.  The \emph{left $S$-transform of $X$} and the \emph{right $S$-transform of $Y$} are the $B$-valued analytic functions defined on a neighbourhood of zero by
\[
S^\ell_X(b) = b^{-1} \Phi^\inv_{\ell, X}(b) \qqand S^r_Y(d) = \Phi^\inv_{r, Y}(d) d^{-1}
\]
respectively.
\end{defn}

A priori, it is not clear that Definition \ref{defn:op-free-S} agrees with \cite{D2006}*{Definition 2.6}.  Indeed, in \cite{D2006}, one instead considers the functions
\[
\Psi_{\ell, X}(b) = \sum_{n\geq 1} E((L_b X)^n) \qqand \Psi_{r, Y}(d) = \sum_{n\geq 1} E((R_d Y)^n),
\]
shows $b \mapsto b^{-1} \Psi^{\inv}_{\ell, X}(b)$ and $d \mapsto \Psi^\inv_{r, Y}(d) d^{-1}$ are analytic $B$-valued functions near zero, and defines
\[
S^\ell_X(b) = (1+b)b^{-1} \Psi^\inv_{\ell, X}(b) \qqand S^r_Y(d) = \Phi^\inv_{r, Y}(d) d^{-1} (1+d).
\]
The following demonstrates these definitions agree.
\begin{prop}
Let $(\A, E, \varepsilon)$ be a Banach $B$-$B$-non-commutative probability space, let $X \in \A_\ell$, and $Y \in \A_r$ be such that $E(X)$ and $E(Y)$ are invertible in $B$.  Then
\[
b^{-1} \Phi^\inv_{\ell, X}(b) = (1+b)b^{-1} \Psi^\inv_{\ell, X}(b) \qqand \Phi^\inv_{r, Y}(d) d^{-1} = \Phi^\inv_{r, Y}(d) d^{-1} (1+d).
\]
\end{prop}
\begin{proof}
To show $b^{-1} \Phi^\inv_{\ell, X}(b) = (1+b)b^{-1} \Psi^\inv_{\ell, X}(b)$ it suffices to show
\[
\Phi^\inv_{\ell, X}(b) = (1+b) \Psi^\inv_{\ell, X}(b)
\]
by uniqueness of power series expansions.  By replacing $b$ with $\Psi_{\ell, X}(b)$ in the above equation (which makes sense as these are analytic functions that take zero to zero), it suffices to show that
\[
\Phi^\inv_{\ell, X}(\Psi_{\ell, X}(b)) = (1+\Psi_{\ell, X}(b))b
\]
or, equivalently,
\[
\Psi_{\ell, X}(b) = \Phi_{\ell, X}((1+\Psi_{\ell, X}(b))b).
\]

To see the above equation, note
\begin{align}
\Psi_{\ell, X}(b) = \sum_{n\geq 1} E((L_b X)^n) = \sum_{n \geq 1} \sum_{\pi \in BNC(\chi_{n,0})} \kappa^B_{\pi}(L_b X, \ldots, L_b X). \label{eq:S-1}
\end{align}
For each $n \geq 1$ and $\pi \in BNC(\chi_{n,0})$, let $V_\pi$ denote the block of $\pi$ containing $n$.  For each set $V = \{1 \leq k_1 < k_2 < \cdots < k_t = n\}$, note, using bi-multiplicative properties, that
\[
\sum_{\substack{\pi \in BNC(\chi_{n,0}) \\ V_\pi = V}} \kappa^B_{\pi}(L_b X, \ldots, L_b X) = \kappa^B_{\chi_{t, 0}}\left( L_{E((L_bX)^{k_1-k_0 - 1})b} X, \ldots, L_{E((L_bX)^{k_t - k_{t-1}-1})b} X\right) 
\]
where $k_0 = 0$.  Below is a bi-non-crossing diagram to aid the reader in seeing how the bi-multiplicative properties are applied.
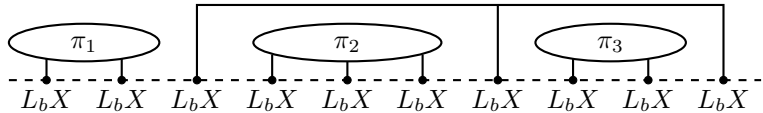
\begin{align*}
\begin{tikzpicture}[baseline]
	\draw[thick, dashed] (-.5,0) -- (9.5, 0);
	\node[below] at (0, 0) {$L_b X$};
	\draw[black, fill=black] (0,0) circle (0.05);	
	\node[below] at (1, 0) {$L_b X$};
	\draw[black, fill=black] (1,0) circle (0.05);	
	\node[below] at (2, 0) {$L_b X$};
	\draw[black, fill=black] (2,0) circle (0.05);	
	\node[below] at (3, 0) {$L_b X$};
	\draw[black, fill=black] (3,0) circle (0.05);	
	\node[below] at (4, 0) {$L_b X$};
	\draw[black, fill=black] (4,0) circle (0.05);	
	\node[below] at (5, 0) {$L_b X$};
	\draw[black, fill=black] (5,0) circle (0.05);	
	\node[below] at (6, 0) {$L_b X$};
	\draw[black, fill=black] (6,0) circle (0.05);	
	\node[below] at (7, 0) {$L_b X$};
	\draw[black, fill=black] (7,0) circle (0.05);	
	\node[below] at (8, 0) {$L_b X$};
	\draw[black, fill=black] (8,0) circle (0.05);	
	\node[below] at (9, 0) {$L_b X$};
	\draw[black, fill=black] (9,0) circle (0.05);	
	\draw[thick] (2, 0) -- (2, 1) -- (9, 1) -- (9,0);
	\draw[thick] (6, 0) -- (6, 1);
	\draw[thick] (4,.5) ellipse (1.25cm and .25cm);
	\node[below] at (4, .7) {$\pi_2$};
	\draw[thick] (4, 0) -- (4, .25);
	\draw[thick] (3, 0) -- (3, .34);
	\draw[thick] (5, 0) -- (5, .34);
	\draw[thick] (.5,.5) ellipse (1cm and .25cm);
	\node[below] at (.5, .7) {$\pi_1$};
	\draw[thick] (0,0) -- (0, .29);
	\draw[thick] (1,0) -- (1, .29);
	\draw[thick] (7.5,.5) ellipse (1cm and .25cm);
	\node[below] at (7.5, .7) {$\pi_3$};
	\draw[thick] (7,0) -- (7, .29);
	\draw[thick] (8,0) -- (8, .29);
\end{tikzpicture}
\end{align*}

Since the sum in (\ref{eq:S-1}) converges absolutely, (\ref{eq:S-1}) may be rearranged to sum over all bi-non-crossing partitions (for any $n$) where $|V_\pi| = t$.  In doing so, the $t$-tuple $(k_1 - k_0 - 1, \ldots, k_t - k_{t-1} - 1)$ takes on all values in $(\bN \cup \{0\})^t$ exactly once so we obtain that
\begin{align*}
\Psi_{\ell, X}(b) &= \sum_{t \geq 1} \kappa^B_{\chi_{t, 0}}( L_{(1+\Psi_{\ell, X}(b))b} X, \ldots, L_{(1+\Psi_{\ell, X}(b))b}X)  \\
&= \Phi_{\ell, X}((1+\Psi_{\ell, X}(b))b)
\end{align*}
as desired.  The other proof is nearly identical (using $d \mapsto R_d$ is antimultiplicative).
\end{proof}

We now state the following result which describes how the $S$-transforms behave with respect to products of operators.
\begin{thm}[\cite{D2006}*{Theorem 1.1}]
\label{thm:free-S}
Let $(\A, E, \varepsilon)$ be a Banach $B$-$B$-non-commutative probability space, let $(X_1, Y_1)$ and $(X_2, Y_2)$ be bi-free over $B$.  Assume that $E(X_k)$ and $E(Y_k)$ are invertible.  Then
\begin{align*}
S^\ell_{X_1X_2}(b) &= S^\ell_{X_2}(b) S^\ell_{X_1}(S^\ell_{X_2}(b)^{-1} b S_{X_2}(b)) \text{ and}\\
S^r_{Y_1Y_2}(d) &= S^r_{Y_1}(S^r_{Y_2}(d) d S_{Y_2}(d)^{-1})  S^r_{Y_2}(d)
\end{align*}
each on a neighbourhood of zero.
\end{thm}
Note the condition that $X_1$ and $X_2$ are free and $Y_1$ and $Y_2$ are free imply $E(X_1X_2) = E(X_1)E(X_2)$ and $E(Y_1Y_2) = E(Y_2)E(Y_1)$ are invertible.  Furthermore, since the Fr\'{e}chet differential of $\Phi^\inv_{\ell, X_2}$ and $\Phi^\inv_{r, Y_2}$ at zero are $c \mapsto c E(X_2)$ and $c \mapsto E(Y_2)c$ respectively,  one sees that $S^\ell_{X_2}(0)$ and $S^r_{Y_2}(0)$ are invertible so $S^\ell_{X_2}(b)^{-1}$ and $S_{Y_2}(d)^{-1}$ make sense on a neighbourhood of zero (and the compositions in Theorem \ref{thm:free-S} make sense).

To provide a proof of  Theorem \ref{thm:free-S} along the lines of \cite{NS1997}, some additional machinery is required.
Given a non-crossing partition $\pi \in NC(n)$, the \emph{Kreweras complement of $\pi$}, denoted $K(\pi)$, is the non-crossing partition on $\{1, \ldots, n\}$ with non-crossing diagram obtained by drawing $\pi$ via the standard non-crossing diagram on $\{1,\ldots, n\}$, placing nodes $1', 2', \ldots, n'$ with $k'$ directly to the right of $k$, and drawing the largest non-crossing partition on $1', 2', \ldots, n'$ that does not intersect $\pi$, which is then $K(\pi)$.  The following diagram exhibits that if $\pi = \{\{1,6\}, \{2\},  \{3,4, 5\}, \{7\}\}$, then $K(\pi) = \{\{1,2,5\}, \{3\}, \{4\}, \{6,7\}\}$.
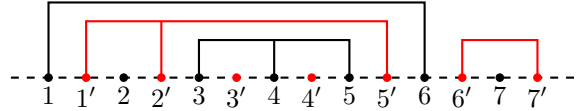
\begin{align*}
	\begin{tikzpicture}[baseline]
	\draw[thick] (1,0) -- (1,1) -- (6,1) -- (6,0);
	\draw[thick] (3,0) -- (3,.5) -- (5,.5) -- (5,0);
	\draw[thick] (4,0) -- (4,.5);
	\draw[thick,red] (1.5,0) -- (1.5,0.75) -- (5.5,0.75)--(5.5, 0);
	\draw[thick,red] (2.5,0) -- (2.5,0.75);
	\draw[thick,red] (6.5,0) -- (6.5,0.5) -- (7.5,0.5)--(7.5, 0);
	\draw[thick, dashed] (0.5,0) -- (8,0);
	\node[below] at (1, 0) {1};
	\draw[fill=black] (1,0) circle (0.05);
	\node[below] at (2, 0) {2};
	\draw[fill=black] (2,0) circle (0.05);
	\node[below] at (3, 0) {3};
	\draw[fill=black] (3,0) circle (0.05);
	\node[below] at (4, 0) {4};
	\draw[fill=black] (4,0) circle (0.05);
	\node[below] at (5, 0) {5};
	\draw[fill=black] (5,0) circle (0.05);
	\node[below] at (6, 0) {6};
	\draw[fill=black] (6,0) circle (0.05);
	\node[below] at (7, 0) {7};
	\draw[fill=black] (7,0) circle (0.05);
	\node[below] at (1.5, 0) {$1'$};
	\draw[red, fill=red] (1.5,0) circle (0.05);
	\node[below] at (2.5, 0) {$2'$};
	\draw[red, fill=red] (2.5,0) circle (0.05);
	\node[below] at (3.5, 0) {$3'$};
	\draw[red, fill=red] (3.5,0) circle (0.05);
	\node[below] at (4.5, 0) {$4'$};
	\draw[red, fill=red] (4.5,0) circle (0.05);
	\node[below] at (5.5, 0) {$5'$};
	\draw[red, fill=red] (5.5,0) circle (0.05);
	\node[below] at (6.5, 0) {$6'$};
	\draw[red, fill=red] (6.5,0) circle (0.05);
	\node[below] at (7.5, 0) {$7'$};
	\draw[red, fill=red] (7.5,0) circle (0.05);
	\end{tikzpicture}
\end{align*}

For $n \geq 1$ let $\sigma_n$ be the partition on $\{1,\ldots, 2n\}$ with blocks $\{1,2\}, \{3, 4\}, \ldots, \{2n-1, 2n\}$, and let
\begin{align*}
BNC'_\ell(n) & = \{ \pi \in BNC(\chi_{2n, 0}) \, \mid \, \pi \vee \sigma_n = 1_{\chi_{2n,0}}, \{1\} \text{ is a block of } \pi, 2p \nsim_{\pi} 2q+1 \text{ for all }p,q \in \bN \cup \{0\}\} \\
BNC'_r(n) & = \{ \pi \in BNC(\chi_{0, 2n}) \, \mid \, \pi \vee \sigma_n = 1_{\chi_{0, 2n}}, \{1\} \text{ is a block of } \pi, 2p \nsim_{\pi} 2q+1 \text{ for all }p,q \in \bN \cup \{0\}\}
\end{align*}
(which are actually the same sets of partitions).  Note if $\pi \in BNC'_k(n)$ for $k \in \{\ell, r\}$, then $\pi|_{\{1, 3, \ldots, 2n-1\}}$ and $\pi|_{\{2, 4, \ldots, 2n\}}$ are Kreweras complements of each other.  Consequently, that there is a bijection between non-crossing partitions $\pi$ on $\{1, 3, \ldots, 2n-1\}$ with $\{1\}$ a block of $\pi$ and $BNC'_k(n)$ via sending $\pi$ to $\pi \cup K(\pi)$ (viewing $K(\pi)$ on $\{2, 4, \ldots, 2n\}$).

If $Z_1, Z_2 \in \A_\ell$ and $Z'_1, Z'_2 \in \A_r$, define
\begin{align*}
\psi_\ell(Z_1, Z_2) &:= \sum_{n\geq 1} \sum_{\pi\in BNC'_\ell(n)} \kappa_\pi(1_\A, Z_1, Z_2, Z_1, Z_2, \ldots, Z_2, Z_1) \text{ and} \\
\psi_r(Z'_1, Z'_2) &:= \sum_{n\geq 1} \sum_{\pi\in BNC'_r(n)} \kappa_\pi(1_\A, Z'_1, Z'_2, Z'_1, Z'_2, \ldots, Z'_2, Z'_1)
\end{align*}
provided these sums converge.  Note that both sums converge absolutely provided $\left\|Z_1\right\|\left\|Z_2\right\| < 1$ and $\left\|Z'_1\right\|\left\|Z'_2\right\| < 1$.  

It is helpful to note
\begin{align}
\label{eq:move-around-B-in-pinched}
\psi_\ell(X_2 L_b, X_1) = \psi_\ell(X_2, L_b X_1) b \qqand \psi_r(Y_2 R_d, Y_1) = d \psi_r(Y_2, R_d Y_1),
\end{align}
which is true by bi-multiplicative properties.

Several Lemmata relating $\psi_\ell$, $\Phi_\ell$, $\psi_r$, and $\Phi_r$ will be required.
\begin{lem}
\label{lem:S-lem-1}
Let $(\A, E, \varepsilon)$ be a Banach $B$-$B$-non-commutative probability space, let $(X_1, Y_1)$ and $(X_2, Y_2)$ be bi-free over $B$.   Then
\[
\Phi_{\ell, X_1X_2}(b) = \psi_\ell(L_b X_1, X_2) \psi_\ell(X_2, L_b X_1) \qqand \Phi_{r, Y_1Y_2}(d) = \psi_r(Y_2, R_d Y_1) \psi_r(R_d Y_1, Y_2)
\]
for $b$ and $d$ sufficiently small.
\end{lem}
\begin{proof}
For $n \geq 1$ let $\sigma_n$ be the partition on $\{1,\ldots, 2n\}$ with blocks $\{1,2\}, \{3, 4\}, \ldots, \{2n-1, 2n\}$.  Since $X_1$ and $X_2$ are free over $B$, Theorem \ref{thm:products} implies that
\begin{align}
\Phi_{\ell, X_1X_2}(b) &= \sum_{n\geq 1} \kappa_{\chi_{n,0}}( L_bX_1X_2, \ldots, L_b X_1X_2)  \nonumber \\
&= \sum_{n\geq 1}  \sum_{\substack{\pi \in BNC(\chi_{2n,0}) \\ \pi \vee \sigma_n = 1_{\chi_{2n,0}}}}   \kappa^B_{\pi}( L_bX_1, X_2, L_b X_1, X_2, \ldots, L_b X_1, X_2) \nonumber \\
&= \sum_{n\geq 1}  \sum_{\substack{\pi \in BNC(\chi_{2n,0}) \\ \pi \vee \sigma_n = 1_{\chi_{2n,0}} \\ 2p \nsim_\pi 2q+1 \forall p, q  }}   \kappa^B_{\pi}( L_bX_1, X_2, L_b X_1, X_2, \ldots, L_b X_1, X_2). \label{eq:sumpsi}
\end{align}

For any fixed $\pi$ in the sum, the conditions $\pi \vee \sigma_n = 1_{\chi_{2n,0}}$ and  $2p \nsim_\pi 2q+1$ for all $p, q$ implies there exists a unique odd number $m$ such that $1 \sim_\pi m$ and $m+1 \sim_\pi 2n$. If $\pi'$ is the partition on $\{1,\ldots, m+1\}$ obtained from $\pi|_{\{1, \ldots, m\}}$ by replacing each number $k$ in each block of $\pi|_{\{1, \ldots, m\}}$ with $k+1$ and adding the block $\{1\}$, then $\pi' \in BNC'_\ell(\frac{m+1}{2})$.  Furthermore, if $\pi''$ is the partition on $\{1, \ldots, 2n - m+1\}$ obtained from $\pi|_{\{m+1, \ldots, 2n\}}$ by replacing each number $k$ in each block of $\pi|_{\{1, \ldots, m\}}$ with $k-m+1$ and adding the block $\{1\}$, then $\pi'' \in BNC'_\ell(\frac{2n - m+1}{2})$.  The diagrams below given an example of such a $\pi$ (with $m = 5$) and how one constructs $\pi'$ and $\pi''$ (the dotted lines represent $\sigma_n$ in the first picture).
\begin{gather*}
\begin{tikzpicture}[baseline]
	\draw[thick, blue, densely dotted] (0,0) -- (.5, 0.25) -- (1,0);
	\draw[thick, blue, densely dotted] (2,0) -- (2.5, 0.25) -- (3,0);
	\draw[thick, blue, densely dotted] (4,0) -- (4.5, 0.25) -- (5,0);
	\draw[thick, blue, densely dotted] (6,0) -- (6.5, 0.25) -- (7,0);
	\draw[thick, blue, densely dotted] (8,0) -- (8.5, 0.25) -- (9,0);
	\draw[thick, blue, densely dotted] (10,0) -- (10.5, 0.25) -- (11,0);
	\draw[thick, dashed] (-.25,0) -- (11.25, 0);
	\node[below] at (0, 0) {$L_b X_1$};
	\draw[black, fill=black] (0,0) circle (0.05);	
	\node[below] at (1, 0) {$X_2$};
	\draw[black, fill=black] (1,0) circle (0.05);	
	\node[below] at (2, 0) {$L_b X_1$};
	\draw[black, fill=black] (2,0) circle (0.05);	
	\node[below] at (3, 0) {$X_2$};
	\draw[black, fill=black] (3,0) circle (0.05);	
	\node[below] at (4, 0) {$L_b X_1$};
	\draw[black, fill=black] (4,0) circle (0.05);	
	\node[below] at (5, 0) {$X_2$};
	\draw[black, fill=black] (5,0) circle (0.05);	
	\node[below] at (6, 0) {$L_b X_1$};
	\draw[black, fill=black] (6,0) circle (0.05);	
	\node[below] at (7, 0) {$X_2$};
	\draw[black, fill=black] (7,0) circle (0.05);	
	\node[below] at (8, 0) {$L_b X_1$};
	\draw[black, fill=black] (8,0) circle (0.05);	
	\node[below] at (9, 0) {$X_2$};
	\draw[black, fill=black] (9,0) circle (0.05);	
	\node[below] at (10, 0) {$L_b X_1$};
	\draw[black, fill=black] (10,0) circle (0.05);	
	\node[below] at (11, 0) {$X_2$};
	\draw[black, fill=black] (11,0) circle (0.05);	
	\draw[thick] (0,0) -- (0,.75) -- (4, .75) -- (4,0);
	\draw[thick] (1,0) -- (1, .5) -- (3, .5) -- (3, 0);
	\draw[thick] (5, 0) -- (5, 1) -- (11, 1) -- (11, 0);
	\draw[thick] (6, 0) -- (6, .75) -- (10, .75) -- (10,0);
	\draw[thick] (7, 0) -- (7, .5) -- (9, .5) -- (9, 0);
\end{tikzpicture}
\\
\Downarrow
\\
\begin{tikzpicture}[baseline]
	\draw[thick, blue, densely dotted] (0,0) -- (-.5, 0.25) -- (-1,0);
	\draw[thick, blue, densely dotted] (2,0) -- (1.5, 0.25) -- (1,0);
	\draw[thick, blue, densely dotted] (4,0) -- (3.5, 0.25) -- (3,0);
	\draw[thick, dashed] (-1.25,0) -- (4.25, 0);
	\node[below] at (-1, 0) {$1_\A$};
	\draw[black, fill=black] (-1,0) circle (0.05);	
	\node[below] at (0, 0) {$L_b X_1$};
	\draw[black, fill=black] (0,0) circle (0.05);	
	\node[below] at (1, 0) {$X_2$};
	\draw[black, fill=black] (1,0) circle (0.05);	
	\node[below] at (2, 0) {$L_b X_1$};
	\draw[black, fill=black] (2,0) circle (0.05);	
	\node[below] at (3, 0) {$X_2$};
	\draw[black, fill=black] (3,0) circle (0.05);	
	\node[below] at (4, 0) {$L_b X_1$};
	\draw[black, fill=black] (4,0) circle (0.05);	
	\draw[thick] (0,0) -- (0,.75) -- (4, .75) -- (4,0);
	\draw[thick] (1,0) -- (1, .5) -- (3, .5) -- (3, 0);
\end{tikzpicture}
\qquad
\begin{tikzpicture}[baseline]
	\draw[thick, blue, densely dotted] (4,0) -- (4.5, 0.25) -- (5,0);
	\draw[thick, blue, densely dotted] (6,0) -- (6.5, 0.25) -- (7,0);
	\draw[thick, blue, densely dotted] (8,0) -- (8.5, 0.25) -- (9,0);
	\draw[thick, blue, densely dotted] (10,0) -- (10.5, 0.25) -- (11,0);
	\draw[thick, dashed] (3.75,0) -- (11.25, 0);	
	\node[below] at (4, 0) {$1_\A$};
	\draw[black, fill=black] (4,0) circle (0.05);	
	\node[below] at (5, 0) {$X_2$};
	\draw[black, fill=black] (5,0) circle (0.05);	
	\node[below] at (6, 0) {$L_b X_1$};
	\draw[black, fill=black] (6,0) circle (0.05);	
	\node[below] at (7, 0) {$X_2$};
	\draw[black, fill=black] (7,0) circle (0.05);	
	\node[below] at (8, 0) {$L_b X_1$};
	\draw[black, fill=black] (8,0) circle (0.05);	
	\node[below] at (9, 0) {$X_2$};
	\draw[black, fill=black] (9,0) circle (0.05);	
	\node[below] at (10, 0) {$L_b X_1$};
	\draw[black, fill=black] (10,0) circle (0.05);	
	\node[below] at (11, 0) {$X_2$};
	\draw[black, fill=black] (11,0) circle (0.05);	
	\draw[thick] (5, 0) -- (5, 1) -- (11, 1) -- (11, 0);
	\draw[thick] (6, 0) -- (6, .75) -- (10, .75) -- (10,0);
	\draw[thick] (7, 0) -- (7, .5) -- (9, .5) -- (9, 0);
\end{tikzpicture}
\end{gather*}

Using bi-multiplicative properties
\begin{align}
\kappa^B_{\pi}&( L_bX_1, X_2, L_b X_1, X_2, \ldots, L_b X_1, X_2) \nonumber \\ &= \kappa^B_{\pi'}(1_\A, L_bX_1, X_2, L_bX_1, X_2, \ldots, X_2, L_bX_1) \kappa^B_{\pi''}(1_\A, X_2, L_bX_1, X_2, L_bX_1, \ldots, L_bX_1, X_2). \label{eq:prod-ofkappa}
\end{align}
Since the map $\pi \mapsto (\pi', \pi'')$ is easily seen to be a bijection from 
\[
\{\pi \, \mid \, n \geq 1, \pi \in BNC(\chi_{2n,0}), \pi \vee \sigma_n = 1_{\chi_{2n,0}},  2p \nsim_\pi 2q+1 \text{ for all } p, q \}
\]
to $\{BNC'_\ell(m) \times BNC'_\ell(k) \, \mid \, m,k\geq 1\}$, the result follows by rearranged sum (\ref{eq:sumpsi}) and using the definition of $\psi_\ell$ (together with the fact that $X_1$ and $X_2$ are free over $B$).

The claim for other equation follows by similar arguments (where the bi-multiplicative properties give the opposite product in (\ref{eq:prod-ofkappa})).
\end{proof}

\begin{lem}
\label{lem:S-lem-2}
Let $(\A, E, \varepsilon)$ be a Banach $B$-$B$-non-commutative probability space, let $(X_1, Y_1)$ and $(X_2, Y_2)$ be bi-free over $B$.   Then
\[
\Phi_{\ell, X_1X_2}(b) = \Phi_{\ell, X_2}(\psi_\ell(L_b X_1, X_2))     \qqand \Phi_{r, Y_1Y_2}(d) = \Psi_{r, Y_2}(\psi_r(R_d Y_1, Y_2))
\]
for $b$ and $d$ sufficiently small.
\end{lem}
\begin{proof}
Note $\Phi_{\ell, X_2}(\psi_\ell(L_b X_1, X_2))$ makes sense for $b$ sufficiently small.  

For $n \geq 1$ let $\sigma_n$ be the partition on $\{1,\ldots, 2n\}$ with blocks $\{1,2\}, \{3, 4\}, \ldots, \{2n-1, 2n\}$.  Since $X_1$ and $X_2$ are free over $B$, Theorem \ref{thm:products} implies that
\begin{align}
\Phi_{\ell, X_1X_2}(b) &= \sum_{n\geq 1}  \sum_{\substack{\pi \in BNC(\chi_{2n,0}) \\ \pi \vee \sigma_n = 1_{\chi_{2n,0}} \\ 2p \nsim_\pi 2q+1 \forall p, q  }}   \kappa^B_{\pi}( L_bX_1, X_2, L_b X_1, X_2, \ldots, L_b X_1, X_2).  \label{eq:sumphipsi}
\end{align}
For each $\pi$ in the sum in (\ref{eq:sumphipsi}), let $V_\pi$ denote the block of $\pi$ containing $2n$.  If $V_\pi = \{1 \leq k_1 < k_2 < \cdots < k_t = 2n\}$ and $k_0 = 0$, then, for each $0 \leq m \leq t-1$, $\pi'_m = \pi|_{\{k_m + 1, k_m + 2, \ldots, k_{m+1}-1\}}$ is an element of $BNC'_\ell(q)$ for some $q$ once the numbers in the blocks of $\pi|_{\{k_m + 1, k_m + 2, \ldots, k_{m+1}-1\}}$ are shifted and the singleton block $\{1\}$ is added to the front.  
Below is an example of such a $\pi$ with $V_\pi = \{3, 10, 12\}$ (again, the dotted lines represent $\sigma_{n}$).
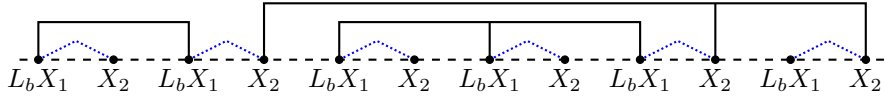
\begin{align*}
\begin{tikzpicture}[baseline]
	\draw[thick, blue, densely dotted] (0,0) -- (.5, 0.25) -- (1,0);
	\draw[thick, blue, densely dotted] (2,0) -- (2.5, 0.25) -- (3,0);
	\draw[thick, blue, densely dotted] (4,0) -- (4.5, 0.25) -- (5,0);
	\draw[thick, blue, densely dotted] (6,0) -- (6.5, 0.25) -- (7,0);
	\draw[thick, blue, densely dotted] (8,0) -- (8.5, 0.25) -- (9,0);
	\draw[thick, blue, densely dotted] (10,0) -- (10.5, 0.25) -- (11,0);
	\draw[thick, dashed] (-.25,0) -- (11.25, 0);
	\node[below] at (0, 0) {$L_b X_1$};
	\draw[black, fill=black] (0,0) circle (0.05);	
	\node[below] at (1, 0) {$X_2$};
	\draw[black, fill=black] (1,0) circle (0.05);	
	\node[below] at (2, 0) {$L_b X_1$};
	\draw[black, fill=black] (2,0) circle (0.05);	
	\node[below] at (3, 0) {$X_2$};
	\draw[black, fill=black] (3,0) circle (0.05);	
	\node[below] at (4, 0) {$L_b X_1$};
	\draw[black, fill=black] (4,0) circle (0.05);	
	\node[below] at (5, 0) {$X_2$};
	\draw[black, fill=black] (5,0) circle (0.05);	
	\node[below] at (6, 0) {$L_b X_1$};
	\draw[black, fill=black] (6,0) circle (0.05);	
	\node[below] at (7, 0) {$X_2$};
	\draw[black, fill=black] (7,0) circle (0.05);	
	\node[below] at (8, 0) {$L_b X_1$};
	\draw[black, fill=black] (8,0) circle (0.05);	
	\node[below] at (9, 0) {$X_2$};
	\draw[black, fill=black] (9,0) circle (0.05);	
	\node[below] at (10, 0) {$L_b X_1$};
	\draw[black, fill=black] (10,0) circle (0.05);	
	\node[below] at (11, 0) {$X_2$};
	\draw[black, fill=black] (11,0) circle (0.05);	
	\draw[thick] (0,0) -- (0,.5) -- (2, .5) -- (2,0);
	\draw[thick] (3, 0) -- (3, .75) -- (11, .75) -- (11, 0);
	\draw[thick] (9, 0) -- (9, .75);
	\draw[thick] (4, 0) -- (4, .5) -- (8, .5) -- (8, 0);
	\draw[thick] (6, 0) -- (6, .5);
\end{tikzpicture}
\end{align*}

As the map from
\[
\{\pi \, \mid \, n \geq 1, \pi \in BNC(\chi_{2n,0}), \pi \vee \sigma_n = 1_{\chi_{2n,0}},  2p \nsim_\pi 2q+1 \text{ for all }p, q \}
\]
to
\[
\{\{  BNC'_\ell(m_1) \times \cdots \times BNC'_\ell(m_q)  \,\mid m_p \geq 1\} \, \mid \, q \geq 1\}
\]
that sends $\pi$ to the element of $(\pi'_1, \ldots, \pi'_t) \in BNC'_\ell(m_1) \times \cdots \times BNC'_\ell(m_t)$, where $t = |V_\pi|$ and $\pi'_1, \ldots, \pi'_q$ are described above, is a bijection, and since the sum in (\ref{eq:sumphipsi}) can be rearranged as it converges absolutely, we obtain via bi-multiplicative properties that 
\[
\Phi_{\ell, X_1X_2}(b) = \sum_{t \geq 1} \kappa_{\chi_{t, 0}}(L_{\psi_\ell(L_b X_1, X_2)} X_2, \ldots, L_{\psi_\ell(L_b X_1, X_2)} X_2) = \Phi_{\ell, X_2}(\psi_\ell(L_b X_1, X_2)).
\]

The proof for $\Phi_{r, Y_1Y_2}(d) = \Psi_{r, Y_2}(\psi_r(R_d Y_1, Y_2))$ is identical.
\end{proof}

Furthermore, if we combine the proofs of Lemmata \ref{lem:S-lem-1} and \ref{lem:S-lem-2}, we obtain the following.
\begin{lem}
\label{lem:S-lem-3}
Let $(\A, E, \varepsilon)$ be a Banach $B$-$B$-non-commutative probability space, let $(X_1, Y_1)$ and $(X_2, Y_2)$ be bi-free over $B$.   Then 
\[
\psi_\ell(X_2, L_b X_1) \psi_\ell(L_b X_1, X_2) = \Phi_{\ell, X_1}(\psi_\ell(X_2 L_b, X_1)) \qand \psi_r(R_d Y_1, Y_2) \psi_r(Y_2, R_d Y_1)  = \Phi_{r, Y_1}(\psi_r(Y_2 R_d, Y_1))
\]
for sufficiently small $b$ and $d$.
\end{lem}
\begin{proof}
Indeed, by the arguments in Lemma \ref{lem:S-lem-1} followed by using bi-multiplicative properties, one obtains that
\begin{align*}
\psi_\ell  (X_2, L_b X_1) \psi_\ell(L_b X_1, X_2) &= \sum_{n\geq 1}  \sum_{\substack{\pi \in BNC(\chi_{2n,0}) \\ \pi \vee \sigma_n = 1_{\chi_{2n,0}} \\ 2p \nsim_\pi 2q+1 \forall p, q  }}   \kappa^B_{\pi}(X_2, L_b X_1, X_2, L_b X_1, \ldots, X_2, L_b X_1) \\
&= \sum_{n\geq 1}  \sum_{\substack{\pi \in BNC(\chi_{2n,0}) \\ \pi \vee \sigma_n = 1_{\chi_{2n,0}} \\ 2p \nsim_\pi 2q+1 \forall p, q  }}   \kappa^B_{\pi}(X_2L_b,  X_1, X_2L_b,  X_1, \ldots, X_2L_b ,X_1).
\end{align*}
However, by the arguments in Lemma \ref{lem:S-lem-2}, if we rearrange the above sum to sum over the block of $\pi$ containing $2n$, we obtain
\[
\psi_\ell  (X_2, L_b X_1) \psi_\ell(L_b X_1, X_2) = \Phi_{\ell, X_1}(\psi_\ell(X_2L_b, X_1))
\]
as desired. 

The right-side again follows via the opposite multiplication need for the bi-multiplicative properties.
\end{proof}

Note Lemmata \ref{lem:S-lem-1} and \ref{lem:S-lem-2} immediately imply other relations.  Indeed Lemma \ref{lem:S-lem-2} implies
\begin{align}
\label{eq:pinched-to-inverse-of-full}
\psi_\ell(L_b X_1, X_2) = \Phi^\inv_{\ell, X_2}(\Phi_{\ell, X_1X_2}(b)) \qqand \psi_r(R_d Y_1, Y_2) = \Phi^\inv_{r, Y_2}(\Phi_{r, Y_1Y_2}(d)).
\end{align}
Then applying Lemma \ref{lem:S-lem-1} to equation (\ref{eq:pinched-to-inverse-of-full}) produces
\begin{align}
\label{eq:inverse-times-pinched}
\Phi^\inv_{\ell, X_2}(\Phi_{\ell, X_1X_2}(b)) \psi_\ell(X_2, L_b X_1) = \Phi_{\ell, X_1X_2}(b) \qand \psi_r(Y_2, R_d Y_1)\Phi^\inv_{r, Y_2}(\Phi_{r, Y_1Y_2}(d) )= \Phi_{r, Y_1Y_2}(d).
\end{align}
Finally, we need the following result.
\begin{lem}
\label{lem:S-lem-4}
Let $(\A, E, \varepsilon)$ be a Banach $B$-$B$-non-commutative probability space, let $(X_1, Y_1)$ and $(X_2, Y_2)$ be bi-free over $B$.  Then
\[
\psi_\ell\left(X_2, L_{\Phi_{\ell, X_1X_2}^\inv(b)} X_1\right) = S^\ell_{X_2}(b)^{-1} \qqand \psi_r\left(Y_2, R_{\Phi^\inv_{Y_1Y_2}(d)} Y_1\right) = S^r_{Y_2}(d)^{-1}
\]
for $b$ and $d$ sufficiently small.
\end{lem}
\begin{proof}
For the left-equation, by replacing $b$ with $\Phi_{\ell, X_1X_2}^\inv(b)$ in (\ref{eq:inverse-times-pinched}), which is a valid operation near $b = 0$, we obtain that
\[
\Phi^\inv_{\ell, X_2}(b) \psi_\ell\left(X_2, L_{\Phi_{\ell, X_1X_2}^\inv(b)} X_1\right)  = b.
\]
By Definition \ref{defn:op-free-S}, we obtain
\[
b S^\ell_{X_2}(b)  \psi_\ell\left(X_2, L_{\Phi_{\ell, X_1X_2}^\inv(b)} X_1\right) = b.
\]
Hence, as both sides are analytic functions near zero, we obtain using the uniqueness of power series expansions that
\[
S^\ell_{X_2}(b)  \psi_\ell\left(X_2, L_{\Phi_{\ell, X_1X_2}^\inv(b)} X_1\right)  = 1
\]
so the left-equation follows.  The right-equation follows by similar arguments.
\end{proof}

\begin{proof}[Proof of Theorem \ref{thm:free-S}]
Only the left $S$-transform will be demonstrated as the right $S$-transform will follow by similar arguments where any two elements of $B$ must be multiplied in the opposite order.

Note
\begin{align*}
\Phi_{\ell, X_1X_2}(b) b &= \psi_\ell(L_b X_1, X_2) \psi_\ell(X_2, L_b X_1) b & & \text{by Lemma } \ref{lem:S-lem-1}\\
&= \psi_\ell(L_b X_1, X_2) \psi_\ell(X_2 L_b,  X_1)  & & \text{by }(\ref{eq:move-around-B-in-pinched})\\
&= \psi_\ell(L_b X_1, X_2) \Phi^\inv_{\ell, X_1}( \Phi_{\ell, X_1}(\psi_\ell(X_2 L_b,  X_1))) \\
&= \psi_\ell(L_b X_1, X_2) \Phi^\inv_{\ell, X_1}\left(\psi_\ell(X_2, L_b X_1) \psi_\ell(L_bX_1, X_2)   \right) & & \text{by Lemma } \ref{lem:S-lem-3} \\
&=  \psi_\ell(L_b X_1, X_2) \Phi^\inv_{\ell, X_1}\left(\psi_\ell(X_2, L_b X_1) \Phi^{\inv}_{\ell, X_2}(\Phi_{\ell, X_1X_2}(b)) \right)  & & \text{by }(\ref{eq:pinched-to-inverse-of-full}).
\end{align*}
By replacing $b$ with $\Phi^{\inv}_{\ell, X_1X_2}(b)$ into the left- and right-most sides of this equation, and by applying Lemma \ref{lem:S-lem-4} and equation (\ref{eq:pinched-to-inverse-of-full}), we obtain that
\[
b \Phi^\inv_{\ell, X_1X_2}(b) = \Phi^\inv_{\ell, X_2}(b) \Phi^{\inv}_{\ell, X_1}\left(S^\ell_{X_2}(b)^{-1} \Phi^{\inv}_{\ell, X_2}(b) \right).
\]
Therefore, by using Definition \ref{defn:op-free-S}, we obtain
\begin{align*}
b^2 S^\ell_{X_1X_2}(b) &= \Phi^\inv_{\ell, X_2}(b) \left(S^\ell_{X_2}(b)^{-1} \Phi^{\inv}_{\ell, X_2}(b) \right)     S^\ell_{X_1}\left(S^\ell_{X_2}(b)^{-1} \Phi^{\inv}_{\ell, X_2}(b) \right) \\
& = b \Phi^{\inv}_{\ell, X_2}(b)   S^\ell_{X_1}\left(S^\ell_{X_2}(b)^{-1} \Phi^{\inv}_{\ell, X_2}(b) \right) \\
& = b^2 S^\ell_{X_2}(b)   S^\ell_{X_1}\left(S^\ell_{X_2}(b)^{-1} b S^\ell_{X_2}(b)\right).
\end{align*}
Therefore, by comparing the power series of the two analytic functions on both sides, the result follows.
\end{proof}

For later purposes, we note the following result whose proof is contained in the proof of Theorem \ref{thm:free-S}.
\begin{lem}
\label{lem:free-S-lem-for-T}
Let $(\A, E, \varepsilon)$ be a Banach $B$-$B$-non-commutative probability space, let $(X_1, Y_1)$ and $(X_2, Y_2)$ be bi-free over $B$.  Then
\begin{align*}
\psi_\ell\left(X_2 L_{\Phi^{\inv}_{\ell, X_1X_2}(b)},  X_1\right) &= \Phi^{\inv}_{\ell, X_1}\left(S^\ell_{X_2}(b)^{-1}  b S^\ell_{X_2}(b) \right) \quad \text{ and } \\ \psi_r\left(Y_2 R_{\Phi^{\inv}_{r, Y_1Y_2}(d)},  Y_1\right) &= \Phi^{\inv}_{r, Y_1}\left(S^r_{Y_2}(d) d S^r_{Y_2}(d)^{-1} \right).
\end{align*}
\end{lem}

\section{The Operator-Valued Partial $T$-Transform}
\label{sec:T}

In this section, the operator-valued bi-free partial transformation that enables additive convolution in one pair and multiplicative convolution in the other is developed.  Using the results of Section \ref{sec:free-S}, the proof is nearly identical to the proof of \cite{S2015-2}*{Theorem 3.5} once bi-multiplicative functions are considered and bookkeeping is done.

There is an additional transformation that we will require to define the operator-valued bi-free partial $S$- and $T$-transforms.  Let $(\A, E, \varepsilon)$ be a Banach $B$-$B$-non-commutative probability space and let $X \in \A_\ell$ and $Y \in \A_r$.  For $b,c,d \in B$, define
\[
K_{X, Y}(b,c,d) := \sum_{n,m\geq 1} \kappa^B_{\chi_{n,m}}(  \underbrace{L_b X, \ldots, L_b X}_{n \text{ entries }}, \underbrace{R_d Y, \ldots, R_d Y}_{m-1 \text{ entries }}, R_d Y R_c).
\]
Again, by Remark \ref{rem:convergence}, $K_{X, Y}(b,c,d)$ converges absolutely for all $c$ in a bounded set provided $b$ and $d$ are sufficiently small.

\begin{defn}
\label{defn:T-Transform}
Let $(\A, E, \varepsilon)$ be a Banach $B$-$B$-non-commutative probability space and let $X \in \A_\ell$ and $Y \in \A_r$ be such that $E(Y)$ is invertible.  The \emph{operator-valued bi-free partial $T$-transform of $(X, Y)$} is the analytic function defined by
\begin{align*}
T_{X, Y}(b,c,d) &:= c + K_{X, Y}\left(b, c, \Phi^\inv_{r,Y}(d)\right) d^{-1} \\
&\,\, = c + K_{X, Y}\left(b, c, S^r_Y(d) d \right) d^{-1}
\end{align*}
for any bounded collection of $c$ provided $b$ and $d$ sufficiently small.
\end{defn}

\begin{rem}
One may be worried about the $d^{-1}$ term in Definition \ref{defn:T-Transform}.  However, notice in the definition of $K_{X, Y}(b,c,d)$ that every term in the infinite sum is multiplied by $d$ on the right due to the properties of bi-multiplicative functions.  Consequently, the $d^{-1}$ really is representing 
\[
\Phi^\inv_{r,Y}(d) d^{-1} = S^r_Y(d)
\]
once one factors out the extraneous $d$-term in $K_{X, Y}(b,c,d)$.  Thus $T_{X, Y}(b,c,d)$ is well-defined.

Note in the case $B = \bC$ that $T_{X, Y}$ agrees with the bi-free partial $T$-transform defined in \cite{S2015-2}*{Proposition 3.2} by letting $b = z$, $d = w$, and $c = 1$.  Furthermore, the following result is the operator-valued analogue of \cite{S2015-2}*{Theorem 3.5} and reduces to \cite{S2015-2}*{Theorem 3.5} when $B = \bC$, $b = z$, $d = w$, and $c = 1$.
\end{rem}

\begin{thm}
\label{thm:T-property}
Let $(\A, E, \varepsilon)$ be a Banach $B$-$B$-non-commutative probability space, let $(X_1, Y_1)$ and $(X_2, Y_2)$ be bi-free over $B$ with $E(Y_1)$ and $E(Y_2)$ invertible.  Then
\[
T_{X_1+ X_2, Y_1Y_2}(b,c,d) = T_{X_1, Y_1}\left(b, \, T_{X_2, Y_2}(b,c,d) S^r_{Y_2}(d)^{-1}, \, S^r_{Y_2}(d) d S^r_{Y_2}(d)^{-1}\right) S^r_{Y_2}(d)
\]
for any bounded collection of $c$ provided $b$ and $d$ sufficiently small.
\end{thm}

Note that since $(X_1, Y_1)$ and $(X_2, Y_2)$ be bi-free over $B$, $Y_1$ and $Y_2$ are free over $B$ so $E(Y_1Y_2) = E(Y_2)E(Y_1)$ is invertible.  Hence $T_{X_1+ X_2, Y_1Y_2}(b,c,d)$ makes sense.  Furthermore, recall $S^r_{Y_2}(d)^{-1}$ exists for $d$ sufficiently close to zero.

Our proof begins by analyzing
\[
K_{X_1+X_2, Y_1Y_2}(b,c,d) = \sum_{n,m\geq 1} \kappa^B_{\chi_{n,m}}(  \underbrace{L_b (X_1+X_2), \ldots, L_b (X_1+X_2)}_{n \text{ entries }}, \underbrace{R_d Y_1Y_2, \ldots, R_d Y_1Y_2}_{m-1 \text{ entries }}, R_d Y_1Y_2 R_c).
\]
For fix $n,m \geq 1$, let $\sigma_{n,m}$ denote the element of $BNC(n,2m)$ with blocks $\{\{k_\ell\}\}_{k=1}^n \cup \{\{(2k-1)_r, (2k)_r\}\}_{k=1}^m$.  Thus Theorem \ref{thm:products} implies that
\begin{align*}
& \kappa^B_{\chi_{n,m}}(  \underbrace{L_b (X_1+X_2), \ldots, L_b (X_1+X_2)}_{n \text{ entries }}, \underbrace{R_d Y_1Y_2, \ldots, R_d Y_1Y_2}_{m-1 \text{ entries}}, R_d Y_1Y_2 R_c) \\
&= \sum_{\substack{ \pi \in BNC(n, 2m) \\ \pi \vee \sigma_{n,m} = 1_{n,2m} }} \kappa^B_\pi(\underbrace{L_b (X_1+X_2), \ldots, L_b (X_1+X_2)}_{n \text{ entries}}, \underbrace{R_d Y_1, Y_2, R_d Y_1, Y_2, \ldots, R_d Y_1, Y_2, R_d Y_1, Y_2 R_c }_{R_d Y_1 \text{ occurs }m \text{ times}}).
\end{align*}
Notice that if $\pi  \in BNC(n, 2m)$ and $\pi \vee \sigma_{n,m} = 1_{n,2m}$, then any block of $\pi$ containing a $k_\ell$ must contain a $j_r$ for some $j$.  Furthermore, if $1 \leq k < j \leq n$ are such that $k_\ell$ and $j_\ell$ are in the same block of $\pi$, then $q_\ell$ must be in the same block as $k_\ell$ for all $k \leq q \leq j$.  Moreover, since $(X_1, Y_1)$ and $(X_2, Y_2)$ are bi-free, we note that 
\[
\kappa^B_\pi(\underbrace{L_b (X_1+X_2), \ldots, L_b (X_1+X_2)}_{n \text{ entries}}, \underbrace{R_d Y_1, Y_2, R_d Y_1, Y_2, \ldots, R_d Y_1, Y_2, R_d Y_1, Y_2 R_c }_{R_d Y_1 \text{ occurs }m \text{ times}})= 0
\]
if $\pi$ contains a block containing a $(2k)_r$ and a $(2j-1)_r$ for some $k, j$.

For $n,m\geq 1$, let $BNC_T(n,m)$ denote all $\pi \in BNC(n, 2m)$ such that $\pi \vee \sigma_{n,m} = 1_{n,2m}$ and $\pi$ contains no blocks containing both a $(2k)_r$ and a $(2j-1)_r$ for some $k, j$.  Consequently, we obtain
\begin{align*}
&K_{X_1+X_2, Y_1Y_2}(b,c,d) \\
&= \sum_{n,m\geq 1}  \sum_{\pi \in BNC_T(n,m)} \kappa^B_\pi(\underbrace{L_b (X_1+X_2), \ldots, L_b (X_1+X_2)}_{n \text{ entries}}, \underbrace{R_d Y_1, Y_2, R_d Y_1, Y_2, \ldots, R_d Y_1, Y_2, R_d Y_1, Y_2 R_c }_{R_d Y_1 \text{ occurs }m \text{ times}}).
\end{align*}
We desire to divide up this sum into two different sums based on types of partitions in $BNC_T(n,m)$ (which is allowed as everything converges absolutely).  Let $BNC_T(n,m)_e$ denote all $\pi \in BNC_T(n,m)$ such that the block containing $1_\ell$ also contains a $(2k)_r$ for some $k$, and let $BNC_T(n,m)_o$ denote all $\pi \in BNC_T(n,m)$ such that the block containing $1_\ell$ also contains a $(2k-1)_r$ for some $k$.  Note $BNC_T(n,m)_e$ and $BNC_T(n,m)_o$ are disjoint and $BNC_T(n,m)_e \cup BNC_T(n,m)_o = BNC_T(n,m)$ by previous discussions.  Therefore, if for $p \in \{o,e\}$ we define
\[
\Psi_p(b,c,d) := \sum_{\substack{ n,m\geq 1 \\ \pi \in BNC_T(n,m)_p}} \kappa^B_\pi(\underbrace{L_b (X_1+X_2), \ldots, L_b (X_1+X_2)}_{n \text{ entries}}, \underbrace{R_d Y_1, Y_2, \ldots, R_d Y_1, Y_2, R_d Y_1, Y_2 R_c }_{R_d Y_1 \text{ occurs }m \text{ times}})
\]
then
\[
K_{X_1+X_2, Y_1Y_2}(b,c,d) = \Psi_e(b,c,d) + \Psi_o(b,c,d).
\]
Expressions for $\Psi_e(b,c,d)$ and $\Psi_o(b,c,d)$ will be derived beginning with $\Psi_e(b,c,d)$.  However, we will not use the same rigour that we did in Section \ref{sec:R-Trans}.

\begin{lem}
\label{lem:T-case-1}
Under the above notation and assumptions,
\[
\Psi_e(b,c,d) = K_{X_2, Y_2}\left(b, c, \psi_r(R_d Y_1, Y_2)\right). 
\]
\end{lem}
\begin{proof}
We desire to rearrange the sum in $\Psi_e(b,c,d)$ (which we can as it converges absolutely) to sum over all $\pi$ with the same block containing $1_\ell$.  The result will then follow by applying bi-multiplicative properties.

Fix $n,m \geq 1$.  If $\pi \in BNC_T(n,m)_e$, then the block $V_\pi$ containing $1_\ell$ must also contain $(2k)_r$ for some $k$ and thus all of $(2m)_r, 1_\ell, 2_\ell, \ldots, n_\ell$ must be in $V_\pi$ in order for $\pi \vee \sigma_{n,m} = 1_{n,2m}$.  Below is an example of such a $\pi$.  Note the dotted lines represent $\sigma_{n,m}$ and we really should draw all of the left nodes above all of the right notes.
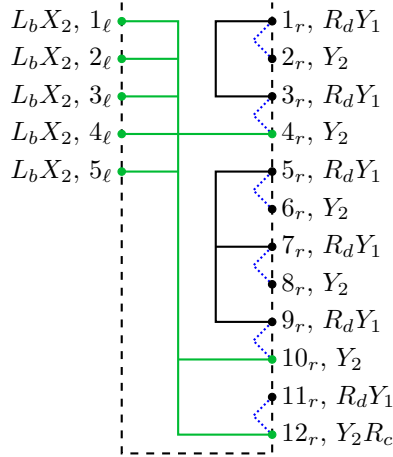
\begin{align*}
	\begin{tikzpicture}[baseline]
	\draw[thick, dashed] (-1,5.75) -- (-1,-.25) -- (1,-.25) -- (1,5.75);
	\draw[thick, blue, densely dotted] (1, 5.5) -- (0.75, 5.25) -- (1, 5);
	\draw[thick, blue, densely dotted] (1, 4.5) -- (0.75, 4.25) -- (1, 4);
	\draw[thick, blue, densely dotted] (1, 3.5) -- (0.75, 3.25) -- (1, 3);
	\draw[thick, blue, densely dotted] (1, 2.5) -- (0.75, 2.25) -- (1, 2);
	\draw[thick, blue, densely dotted] (1, 1.5) -- (0.75, 1.25) -- (1, 1);
	\draw[thick, blue, densely dotted] (1, 0.5) -- (0.75, 0.25) -- (1, 0);
	\draw[ggreen, thick] (-1,5.5) -- (-0.25,5.5) -- (-.25,0) -- (1, 0);
	\draw[thick] (1,1.5) -- (0.25,1.5) -- (0.25,3.5) -- (1, 3.5);
	\draw[thick] (1,4.5) -- (0.25,4.5) -- (0.25,5.5) -- (1, 5.5);
	\draw[thick] (1, 2.5) -- (.25, 2.5);
	\draw[ggreen, thick] (-1, 5) -- (-0.25, 5);
	\draw[ggreen, thick] (-1, 4.5) -- (-0.25, 4.5);
	\draw[ggreen, thick] (-1, 4) -- (-0.25, 4);
	\draw[ggreen, thick] (-1, 3.5) -- (-0.25, 3.5);
	\draw[ggreen, thick] (1, 4) -- (-0.25, 4);
	\draw[ggreen, thick] (1, 1) -- (-0.25, 1);
	\node[left] at (-1, 5.5) {$L_b X_2$, $1_\ell$};
	\draw[ggreen, fill=ggreen] (-1,5.5) circle (0.05);
	\node[left] at (-1, 5) {$L_b X_2$, $2_\ell$};
	\draw[ggreen, fill=ggreen] (-1,5) circle (0.05);
	\node[left] at (-1, 4.5) {$L_b X_2$, $3_\ell$};
	\draw[ggreen, fill=ggreen] (-1,4.5) circle (0.05);
	\node[left] at (-1, 4) {$L_b X_2$, $4_\ell$};
	\draw[ggreen, fill=ggreen] (-1,4) circle (0.05);
	\node[left] at (-1, 3.5) {$L_b X_2$, $5_\ell$};
	\draw[ggreen, fill=ggreen] (-1,3.5) circle (0.05);
	\draw[fill=black] (1,5.5) circle (0.05);
	\node[right] at (1,5.5) {$1_r$, $R_d Y_1$};
	\draw[fill=black] (1,5) circle (0.05);
	\node[right] at (1,5) {$2_r$, $Y_2$};
	\draw[fill=black] (1,4.5) circle (0.05);
	\node[right] at (1,4.5) {$3_r$, $R_d Y_1$};
	\draw[ggreen, fill=ggreen] (1,0) circle (0.05);
	\node[right] at (1,4) {$4_r$, $Y_2$};
	\draw[ggreen, fill=ggreen] (1,4) circle (0.05);
	\node[right] at (1,3.5) {$5_r$, $R_d Y_1$};
	\draw[fill=black] (1,3.5) circle (0.05);
	\node[right] at (1,3) {$6_r$, $Y_2$};
	\draw[fill=black] (1,3) circle (0.05);
	\node[right] at (1,2.5) {$7_r$, $R_d Y_1$};
	\draw[fill=black] (1,2.5) circle (0.05);
	\node[right] at (1,2) {$8_r$, $Y_2$};
	\draw[fill=black] (1,2) circle (0.05);
	\node[right] at (1,1.5) {$9_r$, $R_d Y_1$};
	\draw[fill=black] (1,1.5) circle (0.05);
	\node[right] at (1,1) {$10_r$, $Y_2$};
	\draw[ggreen, fill=ggreen] (1,1) circle (0.05);
	\node[right] at (1,0.5) {$11_r$, $R_d Y_1$};
	\draw[fill=black] (1,0.5) circle (0.05);
	\node[right] at (1,0) {$12_r$, $Y_2 R_c$};
	\end{tikzpicture}
\end{align*}

Let $E = \{(2k)_r\}^{m}_{k=1}$, let $s$ denote the number of elements of $E$ contained in $V_\pi$ (so $s \geq 1$), and let $1 \leq k_1 < k_2 < \cdots < k_s = m$ be such that $(2k_q)_r \in V_\pi$.  Note $V_\pi$ divides the right nodes into $s$ disjoint regions.  For each $1 \leq q \leq s$ let $j_q = k_q - k_{q-1}$, where $k_0 = 0$, and let $\pi_q$ denote the non-crossing partition obtained by restricting $\pi$ to $\{(2k_{q-1} + 1)_r, (2k_{q-1}+2)_r, \ldots, (2k_{q}-1)_r\}$.  Thus $\pi_q$ is naturally an element of $BNC'_r(j_q)$ once a singleton block is added. 

Consequently, if we sum over all possible $n,m\geq 1$, for each $V_\pi$ one obtains $\psi_r(R_d Y_1, Y_2)$ for the $B$-operator in each of the $s$ disjoint regions on the right.  Using bi-multiplicative properties (so the $\psi_r(R_d Y_1, Y_2)$ term attaches to a $Y_2$ from above to obtain $R_{\psi_r(R_d Y_1, Y_2)} Y_2$) and summing over all possible $V_\pi$ yields the result.
\end{proof}

In order to discuss $\Psi_o(b,c,d)$, it is quite helpful to discuss a subcase.  For $n,m \geq 0$, let $\sigma'_{n,m} $ denote the element of $BNC(n, 2m+1)$ with blocks $\{\{k_\ell\}\}_{k=1}^n \cup \{1_r\} \cup \{\{(2k)_r, (2k+1)_r\}\}_{k=1}^m$.  Let $BNC_T(n,m)'_o$ denote the set of all $\pi \in BNC(n,2m+1)$ such that $\pi \vee \sigma'_{n,m} = 1_{n,2m+1}$ and $\pi$ contains no blocks containing both a $(2k)_r$ and a $(2j-1)_r$ for some $k, j$.

\begin{lem}
\label{lem:T-case-2}
Under the above notation and assumptions, if
\begin{align*}
\Psi_{o'}(b,c,d) := \sum_{\substack{n\geq 1,  m\geq 0 \\ \pi \in BNC_T(n,m)'_o}} \kappa^B_\pi(\underbrace{L_b (X_1+X_2), \ldots, L_b (X_1+X_2)}_{n \text{ entries}}, \underbrace{Y_2, R_d Y_1, \ldots, Y_2, R_d Y_1 }_{Y_2, R_d Y_1 \text{ occur }m \text{ times each}} , Y_2 R_c )
\end{align*}
then
\[
\Psi_{o'}(b,c,d) = K_{X_2, Y_2}(b, c, \psi_r(R_d Y_1, Y_2))\psi_r(R_d Y_1, Y_2)^{-1}.
\]
\end{lem}
\begin{proof}
Again we desire to rearrange the sum in $\Psi_{o'}(b,c,d)$ (which we can as it converges absolutely) to sum over all $\pi$ with the same block containing $1_\ell$.  The result will then follow by applying bi-multiplicative properties.

Fix $n \geq 1$ and $m \geq 0$.  If $\pi \in BNC_T(n,m)'_o$, then the block $V_\pi$ containing $1_\ell$ must contain $1_r, (2m+1)_r, 1_\ell, 2_\ell, \ldots, n_\ell$ in order for $\pi \vee \sigma'_{n,m} = 1_{n,2m+1}$.  Below is an example of such a $\pi$. 
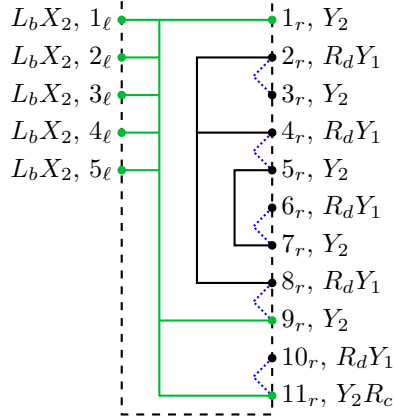
\begin{align*}
	\begin{tikzpicture}[baseline]
	\draw[thick, dashed] (-1,5.75) -- (-1,.25) -- (1,.25) -- (1,5.75);
	\draw[thick, blue, densely dotted] (1, 5) -- (0.75, 4.75) -- (1, 4.5);
	\draw[thick, blue, densely dotted] (1, 4) -- (0.75, 3.75) -- (1, 3.5);
	\draw[thick, blue, densely dotted] (1, 3) -- (0.75, 2.75) -- (1, 2.5);
	\draw[thick, blue, densely dotted] (1, 2) -- (0.75, 1.75) -- (1, 1.5);
	\draw[thick, blue, densely dotted] (1, 1) -- (0.75, 0.75) -- (1, 0.5);
	\draw[ggreen, thick] (-1,5.5) -- (-0.5,5.5) -- (-.5,0.5) -- (1, 0.5);
	\draw[thick] (1,2.5) -- (0.5,2.5) -- (.5,3.5) -- (1, 3.5);
	\draw[thick] (1,5) -- (-0,5) -- (-0,2) -- (1, 2);
	\draw[thick] (1,4) -- (-0,4);
	\draw[ggreen, thick] (-1, 5) -- (-0.5, 5);
	\draw[ggreen, thick] (-1, 4.5) -- (-0.5, 4.5);
	\draw[ggreen, thick] (-1, 4) -- (-0.5, 4);
	\draw[ggreen, thick] (-1, 3.5) -- (-0.5, 3.5);
	\draw[ggreen, thick] (1, 1.5) -- (-0.5, 1.5);
	\draw[ggreen, thick] (1, 5.5) -- (-0.5, 5.5);
	\node[left] at (-1, 5.5) {$L_b X_2$, $1_\ell$};
	\draw[ggreen, fill=ggreen] (-1,5.5) circle (0.05);
	\node[left] at (-1, 5) {$L_b X_2$, $2_\ell$};
	\draw[ggreen, fill=ggreen] (-1,5) circle (0.05);
	\node[left] at (-1, 4.5) {$L_b X_2$, $3_\ell$};
	\draw[ggreen, fill=ggreen] (-1,4.5) circle (0.05);
	\node[left] at (-1, 4) {$L_b X_2$, $4_\ell$};
	\draw[ggreen, fill=ggreen] (-1,4) circle (0.05);
	\node[left] at (-1, 3.5) {$L_b X_2$, $5_\ell$};
	\draw[ggreen, fill=ggreen] (-1,3.5) circle (0.05);
	\draw[ggreen, fill=ggreen] (1,5.5) circle (0.05);
	\node[right] at (1,5.5) {$1_r$, $Y_2$};
	\draw[fill=black] (1,5) circle (0.05);
	\node[right] at (1,5) {$2_r$, $R_d Y_1$};
	\draw[fill=black] (1,4.5) circle (0.05);
	\node[right] at (1,4.5) {$3_r$, $Y_2$};
	\node[right] at (1,4) {$4_r$, $R_d Y_1$};
	\draw[fill=black] (1,4) circle (0.05);
	\node[right] at (1,3.5) {$5_r$, $Y_2$};
	\draw[fill=black] (1,3.5) circle (0.05);
	\node[right] at (1,3) {$6_r$, $R_d Y_1$};
	\draw[fill=black] (1,3) circle (0.05);
	\node[right] at (1,2.5) {$7_r$, $Y_2$};
	\draw[fill=black] (1,2.5) circle (0.05);
	\node[right] at (1,2) {$8_r$, $R_d Y_1$};
	\draw[fill=black] (1,2) circle (0.05);
	\node[right] at (1,1.5) {$9_r$, $Y_2$};
	\draw[ggreen, fill=ggreen] (1,1.5) circle (0.05);
	\node[right] at (1,1) {$10_r$, $R_d Y_1$};
	\draw[fill=black] (1,1) circle (0.05);
	\node[right] at (1,0.5) {$11_r$, $Y_2R_c$};
	\draw[ggreen, fill=ggreen] (1,0.5) circle (0.05);
	\end{tikzpicture}
\end{align*}

Let $O = \{(2k-1)_r\}^{m+1}_{k=1}$, let $s$ denote the number of elements of $O$ contained in $V_\pi$ (so $s \geq 1$), and let $1 = k_1 < k_2 < \cdots < k_s = m+1$ be such that $(2k_q-1)_r \in V_\pi$. Note $V_\pi$ divides the right nodes into $s-1$ disjoint regions.  For each $1 \leq q \leq s-1$ let $j_q = k_{q+1} - k_q$ and let $\pi_q$ denote the non-crossing partition obtained by restricting $\pi$ to $\{(2k_{q} )_r, (2k_{q}+1)_r, \ldots, (2k_{q+1}-2)_r\}$.  Note  $\pi_q$ is naturally an element of $BNC'_r(j_q)$  once a singleton block is added.  

Consequently, if we sum over all possible $n,m\geq 1$, for each $V_\pi$ one obtains $\psi_r(R_d Y_1, Y_2)$ for the $B$-operator in each of the $s-1$ disjoint regions on the right.  Using bi-multiplicative properties (so the $\psi_r(R_d Y_1, Y_2)$ term attaches to a $Y_2$ from above to obtain $R_{\psi_r(R_d Y_1, Y_2)} Y_2$) and summing over all possible $V_\pi$ yields the result.   Note the $\psi_r(R_d Y_1, Y_2)^{-1}$ comes from the fact that the $1_r$ term is always just a $Y_2$ where it must be a $R_{\psi_r(R_d Y_1, Y_2)}Y_2$ in $K_{X_2, Y_2}(b, c, \psi_r(R_d Y_1, Y_2))$ and bi-multiplicative properties correct this.
\end{proof}

\begin{lem}
\label{lem:T-case-3}
Under the above notation and assumptions, 
\[
\Psi_o(b,c,d) =  K_{X_1, Y_1}\left(b, \Psi_{o'}(b,c,d) + c \psi_r(Y_2, R_d Y_1), d \psi_r(Y_2, R_d Y_1)\right) (d \psi_r(Y_2, R_d Y_1))^{-1} d.
\]
\end{lem}
\begin{proof}
Again we desire to rearrange the sum in $\Psi_{o}(b,c,d)$ (which we can as it converges absolutely) to sum over all $\pi$ with the same block containing $1_\ell$.  The result will then follow by applying bi-multiplicative properties.

Fix $n,m \geq 1$, let $O = \{(2k-1)_r\}^{m}_{k=1}$, let $\pi \in BNC_T(n,m)_o$, let $V_\pi$ denote the block of $\pi$ containing $1_\ell$, let $t$ (respectively $s$) denote the number of elements of $\{1_\ell, \ldots, n_\ell\}$ (respectively $O$) contained in $V_\pi$ (so $t, s \geq 1$).  Since $\pi \vee \sigma_{n,m} = 1_{n,2m}$, $V_\pi$ must be of the form $\{k_\ell\}^t_{k=1} \cup \{(2k_q-1)_r\}^s_{q=1}$ for some $1 = k_1 < k_2 < \cdots < k_s \leq m$.  Below is an example of such a $\pi$.
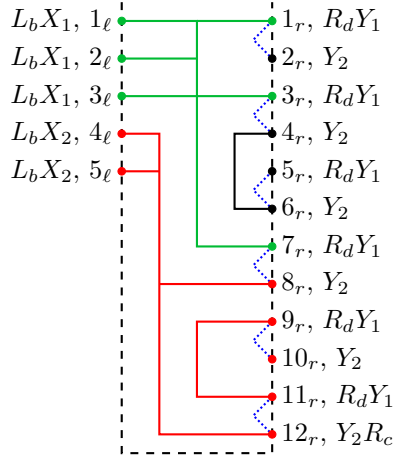
\begin{align*}
	\begin{tikzpicture}[baseline]
	\draw[thick, dashed] (-1,5.75) -- (-1,-.25) -- (1,-.25) -- (1,5.75);
	\draw[thick, blue, densely dotted] (1, 5.5) -- (0.75, 5.25) -- (1, 5);
	\draw[thick, blue, densely dotted] (1, 4.5) -- (0.75, 4.25) -- (1, 4);
	\draw[thick, blue, densely dotted] (1, 3.5) -- (0.75, 3.25) -- (1, 3);
	\draw[thick, blue, densely dotted] (1, 2.5) -- (0.75, 2.25) -- (1, 2);
	\draw[thick, blue, densely dotted] (1, 1.5) -- (0.75, 1.25) -- (1, 1);
	\draw[thick, blue, densely dotted] (1, 0.5) -- (0.75, 0.25) -- (1, 0);
	\draw[ggreen, thick] (-1,5.5) -- (0, 5.5) -- (0, 2.5) -- (1, 2.5);
	\draw[thick] (1,3) -- (.5, 3) -- (.5, 4) -- (1, 4);
	\draw[thick, red] (1,.5) -- (0, .5) -- (0, 1.5) -- (1, 1.5);
	\draw[thick, red] (-1,4) -- (-.5, 4) -- (-.5, 0) -- (1, 0);
	\draw[ggreen, thick] (1, 5.5) -- (0, 5.5);
	\draw[ggreen, thick] (1, 4.5) -- (0, 4.5);
	\draw[ggreen, thick] (-1, 4.5) -- (0, 4.5);
	\draw[ggreen, thick] (-1, 5) -- (0, 5);
	\draw[thick, red] (1, 2) -- (-.5, 2);
	\draw[thick, red] (-1, 3.5) -- (-.5, 3.5);
	\node[left] at (-1, 5.5) {$L_b X_1$, $1_\ell$};
	\draw[ggreen, fill=ggreen] (-1,5.5) circle (0.05);
	\node[left] at (-1, 5) {$L_b X_1$, $2_\ell$};
	\draw[ggreen, fill=ggreen](-1,5) circle (0.05);
	\node[left] at (-1, 4.5) {$L_b X_1$, $3_\ell$};
	\draw[ggreen, fill=ggreen] (-1,4.5) circle (0.05);
	\node[left] at (-1, 4) {$L_b X_2$, $4_\ell$};
	\draw[red, fill=red] (-1,4) circle (0.05);
	\node[left] at (-1, 3.5) {$L_b X_2$, $5_\ell$};
	\draw[red, fill=red] (-1,3.5) circle (0.05);
	\draw[ggreen, fill=ggreen] (1,5.5) circle (0.05);
	\node[right] at (1,5.5) {$1_r$, $R_d Y_1$};
	\draw[fill=black] (1,5) circle (0.05);
	\node[right] at (1,5) {$2_r$, $Y_2$};
	\draw[ggreen, fill=ggreen] (1,4.5) circle (0.05);
	\node[right] at (1,4.5) {$3_r$, $R_d Y_1$};
	\draw[red, fill=red] (1,0) circle (0.05);
	\node[right] at (1,4) {$4_r$, $Y_2$};
	\draw[fill=black] (1,4) circle (0.05);
	\node[right] at (1,3.5) {$5_r$, $R_d Y_1$};
	\draw[black, fill=black] (1,3.5) circle (0.05);
	\node[right] at (1,3) {$6_r$, $Y_2$};
	\draw[fill=black] (1,3) circle (0.05);
	\node[right] at (1,2.5) {$7_r$, $R_d Y_1$};
	\draw[ggreen, fill=ggreen] (1,2.5) circle (0.05);
	\node[right] at (1,2) {$8_r$, $Y_2$};
	\draw[red, fill=red] (1,2) circle (0.05);
	\node[right] at (1,1.5) {$9_r$, $R_d Y_1$};
	\draw[red, fill=red] (1,1.5) circle (0.05);
	\node[right] at (1,1) {$10_r$, $Y_2$};
	\draw[red, fill=red] (1,1) circle (0.05);
	\node[right] at (1,0.5) {$11_r$, $R_d Y_1$};
	\draw[red, fill=red] (1,0.5) circle (0.05);
	\node[right] at (1,0) {$12_r$, $Y_2R_c$};
	\end{tikzpicture}
\end{align*}

Note $V_\pi$ divides the right nodes into $s$ disjoint regions where the bottom region is special as those nodes may connect to left nodes.   For each $1 \leq q \leq s$ let $j_q = k_{q+1} - k_q$ where $k_s = m+1$ and, for $q \neq s$, let $\pi_q$ denote the non-crossing partition obtained by restricting $\pi$ to $\{(2k_{q})_r, (2k_{q}+1)_r, \ldots, (2k_{q+1}-2)_r\}$.  Note $\pi_q$ is naturally an element of $BNC'_r(j_q)$  once a singleton block is added.

Let $\pi'_s$ denote the bi-non-crossing partition obtained by restricting $\pi$ to $\{k_\ell\}_{k=t+1}^n \cup \{(2k_s)_r, (2k_s + 1)_r, \ldots, (2m)_r\}$ (which is shaded differently in the above diagram).  Notice, in order for $\pi \vee \sigma_{n,m} = 1_{2n,2m}$, it must be the case that $\pi_s \in BNC_T(n-t,j_s-1)'_o$.  

Consequently, if we sum over all possible $n,m\geq 1$, for each $V_\pi$ one obtains $\psi_r(Y_2, R_d Y_1)$ for the $B$-operator in each of the $s-1$ disjoint regions on the right (excluding the bottom region).  If one sums over all possible bottom regions, one obtains $c \psi_r(Y_2, R_d Y_1)$ from the $\pi$ with $t = n$ and one obtains $\Psi_{o'}(b,c,d)$ from all other $\pi$.  Using bi-multiplicative properties (so the $\psi_r(Y_2, R_d Y_1)$ term attaches to a $R_dY_1$ from above to obtain $R_{d\psi_r(Y_2, R_d Y_1)} Y_1$) and summing over all possible $V_\pi$ yields the result.   Note the `$d$'  in the $(d \psi_r(Y_2, R_d Y_1))^{-1} d$ term comes from the $R_d Y_1$ in the $1_r$ position whereas the $(d \psi_r(Y_2, R_d Y_1))^{-1}$ comes from the fact that we want the $1_r$ term to be $R_{d \psi_r(Y_2, R_d Y_1) }Y_1$ to match the term in $K_{X_1, Y_1}(b, c, d \psi_r(Y_2, R_d Y_1))$ and bi-multiplicative properties correct this.
\end{proof}

\begin{proof}[Proof of Theorem \ref{thm:T-property}.]
Note
\[
\psi_r\left(R_{\Phi^{\inv}_{r, Y_1Y_2}(d)} Y_1, Y_2\right) = \Phi^{\inv}_{r, Y_2}(d)
\]
by equation (\ref{eq:pinched-to-inverse-of-full}), 
\[
d \psi_r(Y_2, R_d Y_1) = \psi_r(Y_2 R_d, Y_1)
\]
by equation (\ref{eq:move-around-B-in-pinched}), 
\[
\psi_r\left(Y_2 R_{\Phi^{\inv}_{r, Y_1Y_2}(d)},  Y_1\right) = \Phi^{\inv}_{r, Y_1}\left(S^r_{Y_2}(d) d S^r_{Y_2}(d)^{-1} \right)
\]
by Lemma \ref{lem:free-S-lem-for-T}, and
\[
\psi_r\left(Y_2, R_{\Phi^{\inv}_{r, Y_1Y_2}(d)} Y_1\right)= d\Phi^\inv_{r, Y_2}(d )^{-1} = S^r_{Y_2}(d)^{-1}
\]
by equation (\ref{eq:inverse-times-pinched}), 
Consequently, by replacing $d$ with $\Phi^{\inv}_{r, Y_1Y_2}(d)$ in each of the expressions from Lemmata \ref{lem:T-case-1}, \ref{lem:T-case-2}, and \ref{lem:T-case-3}, we obtain that

\begin{align*}
& K_{X_1+X_2, Y_1Y_2}\left(b,c,\Phi^{\inv}_{r, Y_1Y_2}(d)\right) \\
&= K_{X_2, Y_2}\left(b, c, \Psi^{\inv}_{r, Y_2}(d)\right) + K_{X_1, Y_1}\left(b, \Psi_{o'}\left(b,c, \Phi^{\inv}_{r, Y_1Y_2}(d)\right) + c S^r_{Y_2}(d)^{-1}, \Phi^{\inv}_{r, Y_1}\left(S^r_{Y_2}(d) d S^r_{Y_2}(d)^{-1} \right) \right) S^r_{Y_2}(d)
\end{align*}
where
\[
\Psi_{o'}\left(b,c, \Phi^{\inv}_{r, Y_1Y_2}(d)\right) = K_{X_2, Y_2}\left(b, c, \Phi^{\inv}_{r, Y_2}(d)\right) \Phi^\inv_{r, Y_2}(d)^{-1} = K_{X_2, Y_2}\left(b, c, \Phi^{\inv}_{r, Y_2}(d)\right) d^{-1} S^r_{Y_2}(d)^{-1}.
\]
Therefore
\begin{align*}
& T_{X_1, Y_1}\left(b, \, T_{X_2, Y_2}(b,c,d) S^r_{Y_2}(d)^{-1}, \, S^r_{Y_2}(d) d S^r_{Y_2}(d)^{-1}\right) S^r_{Y_2}(d) \\
&= T_{X_2, Y_2}(b,c,d) S^r_{Y_2}(d)^{-1} S^r_{Y_2}(d) \\
&\quad + K_{X_1, Y_1}\left(b, T_{X_2, Y_2}(b,c,d) S^r_{Y_2}(d)^{-1}, \Phi^{\inv}_{r, Y_1}\left(S^r_{Y_2}(d) d S^r_{Y_2}(d)^{-1} \right)\right)  \left(S^r_{Y_2}(d) d S^r_{Y_2}(d)^{-1}\right)^{-1}  S^r_{Y_2}(d) \\
&= T_{X_2, Y_2}(b,c,d) + K_{X_1, Y_1}\left(b, T_{X_2, Y_2}(b,c,d) S^r_{Y_2}(d)^{-1}, \Phi^{\inv}_{r, Y_1}\left(S^r_{Y_2}(d) d S^r_{Y_2}(d)^{-1} \right)\right) S^r_{Y_2}(d) d^{-1} \\
&= c + K_{X_2, Y_2}\left(b, c, \Psi^{\inv}_{r, Y_2}(d)\right) d^{-1}\\
& \quad + K_{X_1, Y_1}\left(b, c S^r_{Y_2}(d)^{-1} +   K_{X_2, Y_2}\left(b, c, \Psi^{\inv}_{r, Y_2}(d)\right)d^{-1}S^r_{Y_2}(d)^{-1}, \Phi^{\inv}_{r, Y_1}\left(S^r_{Y_2}(d) d S^r_{Y_2}(d)^{-1} \right)\right) S^r_{Y_2}(d) d^{-1} \\
&= c + K_{X_2, Y_2}\left(b, c, \Psi^{\inv}_{r, Y_2}(d)\right)d^{-1} \\
& \quad  + K_{X_1, Y_1}\left(b, c S^r_{Y_2}(d)^{-1} +   \Psi_{o'}\left(b,c, \Phi^{\inv}_{Y_1Y_2}(d)\right), \Phi^{\inv}_{r, Y_1}\left(S^r_{Y_2}(d) d S^r_{Y_2}(d)^{-1} \right)\right) S^r_{Y_2}(d) d^{-1} \\
&= c + K_{X_1+X_2, Y_1Y_2}\left(b,c,\Phi^{\inv}_{Y_1Y_2}(d)\right) d^{-1} \\
&= T_{X_1 + X_2, Y_1Y_2}(b,c,d)
\end{align*}
as claimed.
\end{proof}

\section{The Operator-Valued Partial $S$-Transform}
\label{sec:S}

In this section, the operator-valued bi-free partial transformation that enables multiplicative convolution both pairs is developed.  Using the results of Section \ref{sec:free-S}, the proof is near identical to the proof of \cite{S2015-2}*{Theorem 4.5} once bi-multiplicative functions are considered and bookkeeping is done.

\begin{defn}
\label{defn:S-Transform}
Let $(\A, E, \varepsilon)$ be a Banach $B$-$B$-non-commutative probability space and let $X \in \A_\ell$ and $Y \in \A_r$ be such that $E(Y)$ and $E(X)$ are invertible.  The \emph{operator-valued bi-free partial $S$-transform of $(X, Y)$} is the analytic function defined by
\begin{align*}
S_{X, Y}(b,c,d) &= c + b^{-1} \Upsilon_{X, Y}(b,c,d) + \Upsilon_{X, Y}(b,c,d) d^{-1} + b^{-1} \Upsilon_{X, Y}(b,c,d) d^{-1}
\end{align*}
where
\[
\Upsilon_{X, Y}(b,c,d) := K_{X, Y}\left(\Phi^\inv_{\ell, X}(b), c, \Phi^\inv_{r,Y}(d)\right) = K_{X, Y}\left(b S^\ell_X(b), c, S^r_Y(d)d \right)
\]
for any bounded collection of $c$ provided $b$ and $d$ sufficiently small.
\end{defn}

\begin{rem}
One may be worried about the $b^{-1}$ and $d^{-1}$ terms in Definition \ref{defn:S-Transform}.  However, notice that in the definition of $K_{X, Y}(b,c,d)$, every term in the infinite sum is some element of $B$ multiplied by $b$ on the left and $d$ on the right, due to the properties of bi-multiplicative functions.  Consequently, the $b^{-1}$ and $d^{-1}$ are really representing 
\[
b^{-1}\Phi^\inv_{\ell, X}(b) = S^\ell_X(b) \qqand \Phi^\inv_{r,Y}(d) d^{-1} = S^r_Y(d)
\]
once one factors out the extraneous $b$- and $d$-terms in $K_{X, Y}(b,c,d)$.  Thus $S_{X, Y}(b,c,d)$ is well-defined.

Note in the case $B = \bC$ that $S_{X, Y}$ agrees with the bi-free partial $S$-transform defined in \cite{S2015-2}*{Proposition 4.2} by letting $b = z$, $d = w$, and $c = 1$.  Furthermore, the following result is the operator-valued analogue of \cite{S2015-2}*{Theorem 4.5} and reduces to \cite{S2015-2}*{Theorem 4.5} when $B = \bC$, $b = z$, $d = w$, and $c = 1$.
\end{rem}

\begin{thm}
\label{thm:S-property}
Let $(\A, E, \varepsilon)$ be a Banach $B$-$B$-non-commutative probability space, let $(X_1, Y_1)$ and $(X_2, Y_2)$ be bi-free over $B$ with $E(X_k)$ and $E(Y_k)$ invertible for all $k$.  Then
\begin{align*}
S_{X_1 X_2, Y_1Y_2}& (b,c,d) \\ & = S^\ell_{X_2}(b)S_{X_1, Y_1}\left(S^\ell_{X_2}(b)^{-1}bS^\ell_{X_2}(b), \,\, S^\ell_{X_2}(b)^{-1}      S_{X_2, Y_2}(b,c,d) S^r_{Y_2}(d)^{-1}, \,\, S^r_{Y_2}(d) d S^r_{Y_2}(d)^{-1}\right) S^r_{Y_2}(d)
\end{align*}
for any bounded collection of $c$ provided $b$ and $d$ sufficiently small.
\end{thm}

Note that since $(X_1, Y_1)$ and $(X_2, Y_2)$ be bi-free over $B$, we obtain that $E(X_1X_2) = E(X_1)E(X_2)$ and $E(Y_1Y_2) = E(Y_2)E(Y_1)$ are invertible.  Hence $S_{X_1 X_2, Y_1Y_2}(b,c,d)$ makes sense.  Furthermore, the above formula makes sense by the same arguments used to show that the formula in Theorem \ref{thm:T-property} makes sense.

Our proof begins by analyzing
\[
K_{X_1X_2, Y_1Y_2}(b,c,d) = \sum_{n,m\geq 1} \kappa^B_{n,m}(  \underbrace{L_b X_1X_2, \ldots, L_b X_1X_2}_{n \text{ entries }}, \underbrace{R_d Y_1Y_2, \ldots, R_d Y_1Y_2}_{m-1 \text{ entries }}, R_d Y_1Y_2 R_c)  .
\]
For fix $n,m \geq 1$, let $\sigma_{n,m}$ denote the element of $BNC(2n,2m)$ with blocks $\{\{(2k-1)_\ell, (2k)_\ell\}\}_{k=1}^n \cup \{\{(2k-1)_r, (2k)_r\}\}_{k=1}^m$.  Thus Theorem \ref{thm:products} implies that
\begin{align*}
&\kappa^B_{n,m}(  \underbrace{L_b X_1X_2, \ldots, L_b X_1X_2}_{n \text{ entries }}, \underbrace{R_d Y_1Y_2, \ldots, R_d Y_1Y_2}_{m-1 \text{ entries }}, R_d Y_1Y_2 R_c) \\
&= \sum_{\substack{ \pi \in BNC(2n, 2m) \\ \pi \vee \sigma_{n,m} = 1_{2n,2m} }} \kappa^B_\pi(\underbrace{L_b X_1, X_2, L_b X_1, X_2, \ldots, L_b X_1, X_2}_{X_2 \text{ occurs }n \text{ times}}, \underbrace{R_d Y_1, Y_2, R_dY_1, Y_2, \ldots, R_dY_1, Y_2}_{Y_2 \text{ occurs }m-1 \text{ times}}, R_d Y_1, Y_2 R_c).
\end{align*}
Since $(X_1, Y_1)$ and $(X_2, Y_2)$ are bi-free, we note that 
\[
\kappa^B_\pi(\underbrace{L_b X_1, X_2, L_b X_1, X_2, \ldots, L_b X_1, X_2}_{X_2 \text{ occurs }n \text{ times}}, \underbrace{R_d Y_1, Y_2, R_dY_1, Y_2, \ldots, R_dY_1, Y_2}_{Y_2 \text{ occurs }m-1 \text{ times}}, R_d Y_1, Y_2 R_c) = 0
\]
if $\pi$ contains a block containing a $(2k)_{\theta_1}$ and a $(2j-1)_{\theta_2}$ for some $\theta_1, \theta_2 \in \{\ell, r\}$ and for some $k, j$.

For $n,m\geq 1$, let $BNC_S(n,m)$ denote all $\pi \in BNC(2n, 2m)$ such that $\pi \vee \sigma_{n,m} = 1_{2n,2m}$ and $\pi$ contains no blocks with both a $(2k)_{\theta_1}$ and a $(2j-1)_{\theta_2}$ for some $\theta_1, \theta_2 \in \{\ell, r\}$ and for some $k, j$.  Consequently, we obtain
\begin{align*}
&K_{X_1X_2, Y_1Y_2}(b,c,d) \\ 
&= \sum_{n,m\geq 1} \left(\sum_{\pi \in BNC_S(n,m)} \kappa^B_\pi(\underbrace{L_b X_1, X_2, \ldots, L_b X_1, X_2}_{X_2 \text{ occurs }n \text{ times}}, \underbrace{R_d Y_1, Y_2, \ldots, R_dY_1, Y_2}_{Y_2 \text{ occurs }m-1 \text{ times}}, R_d Y_1, Y_2 R_c) \right).
\end{align*}

We desire to divide up this sum into two sums based on types of partitions in $BNC_S(n,m)$.  Notice that if $\pi \in BNC_S(n,m)$, then $\pi$ must contain a block with both a $k_\ell$ and a $j_r$ for some $k,j$ so that $\pi \vee \sigma_{n,m} = 1_{2n, 2m}$.

Let $BNC_S(n,m)_e$ denote all $\pi \in BNC_S(n,m)$ such that the first block of $\pi$, as measured from the top in the bi-non-crossing diagram, that has both left and right nodes has nodes of even index.  Similarly let $BNC_S(n,m)_o$ denote all $\pi \in BNC_T(n,m)$ such that the first block of $\pi$, as measured from the top in the bi-non-crossing diagram, that has both left and right nodes has nodes of odd index. Note $BNC_S(n,m)_e$ and $BNC_S(n,m)_o$ are disjoint and $BNC_S(n,m)_e \cup BNC_S(n,m)_o = BNC_S(n,m)$ (as the top block can either have nodes of even index or nodes of odd index).  Therefore, if for $p \in \{o,e\}$ we define
\[
\Psi_p(b,c,d) := \sum_{n,m\geq 1} \left(\sum_{\pi \in BNC_S(n,m)_p} \kappa^B_\pi(\underbrace{L_b X_1, X_2, \ldots, L_b X_1, X_2}_{X_2 \text{ occurs }n \text{ times}}, \underbrace{R_d Y_1, Y_2, \ldots, R_dY_1, Y_2}_{Y_2 \text{ occurs }m-1 \text{ times}}, R_d Y_1, Y_2 R_c) \right),
\]
then
\[
K_{X_1X_2, Y_1Y_2}(b,c,d) = \Psi_e(b,c,d) + \Psi_o(b,c,d).
\]
Expressions for $\Psi_e(b,c,d)$ and $\Psi_o(b,c,d)$ will be derived beginning with $\Psi_e(b,c,d)$.  Again, we will not use the same rigour as we did in Section \ref{sec:R-Trans}.

\begin{lem}
\label{lem:S-case-1}
Under the above notation and assumptions,
\[
\Psi_e(b,c,d) = K_{X_2, Y_2}\left(\psi_\ell(L_b X_1, X_2), c, \psi_r(R_d Y_1, Y_2) \right). 
\]
\end{lem}
\begin{proof}
Fix $n,m \geq 1$.  If $\pi \in BNC_S(n,m)_e$, let $V_\pi$ denote the first (and, as it happens, only) block of $\pi$, as measured from the top of $\pi$'s bi-non-crossing diagram, that has both left and right nodes.  Since $\pi \vee \sigma_{n,m} = 1_{2n,2m}$, there exist $t,s\geq 1$, $1\leq l_1 < l_2 < \cdots < l_t = n$, and $1 \leq k_1 < k_2 < \cdots < k_s = m$ such that
\[
V_\pi = \{(2l_p)_\ell\}^t_{p=1} \cup \{(2k_q)_r\}^s_{q=1}.
\]
Note $V_\pi$ divides the remaining left nodes into $t$ disjoint regions and the remaining right nodes into $s$ disjoint regions. Moreover, each block of $\pi$ can only contain nodes in one such region.  Below is an example of such a $\pi$.
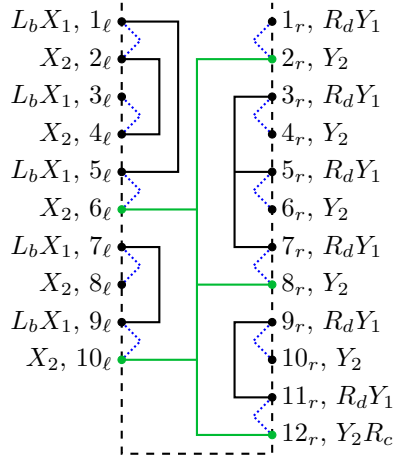
\begin{align*}
	\begin{tikzpicture}[baseline]
	\draw[thick, dashed] (-1,5.75) -- (-1,-.25) -- (1,-.25) -- (1,5.75);
	\draw[thick, blue, densely dotted] (1, 5.5) -- (0.75, 5.25) -- (1, 5);
	\draw[thick, blue, densely dotted] (1, 4.5) -- (0.75, 4.25) -- (1, 4);
	\draw[thick, blue, densely dotted] (1, 3.5) -- (0.75, 3.25) -- (1, 3);
	\draw[thick, blue, densely dotted] (1, 2.5) -- (0.75, 2.25) -- (1, 2);
	\draw[thick, blue, densely dotted] (1, 1.5) -- (0.75, 1.25) -- (1, 1);
	\draw[thick, blue, densely dotted] (1, 0.5) -- (0.75, 0.25) -- (1, 0);
	\draw[thick, blue, densely dotted] (-1, 5.5) -- (-0.75, 5.25) -- (-1, 5);
	\draw[thick, blue, densely dotted] (-1, 4.5) -- (-0.75, 4.25) -- (-1, 4);
	\draw[thick, blue, densely dotted] (-1, 3.5) -- (-0.75, 3.25) -- (-1, 3);
	\draw[thick, blue, densely dotted] (-1, 2.5) -- (-0.75, 2.25) -- (-1, 2);
	\draw[thick, blue, densely dotted] (-1, 1.5) -- (-0.75, 1.25) -- (-1, 1);
	\draw[thick, ggreen] (1,5) -- (0,5) -- (0,0) -- (1,0);
	\draw[thick] (-1,5.5) -- (-.25, 5.5) -- (-.25, 3.5) -- (-1, 3.5);
	\draw[thick] (1, 4.5) -- (.5, 4.5) -- (.5, 2.5) -- (1, 2.5);
	\draw[thick] (1, 3.5) -- (.5, 3.5);
	\draw[thick] (1,1.5) -- (.5, 1.5) -- (.5, 0.5) -- (1, 0.5);
	\draw[thick] (-1, 1.5) -- (-.5, 1.5) -- (-.5, 2.5) -- (-1, 2.5);
	\draw[thick] (-1, 4) -- (-.5, 4) -- (-.5, 5) -- (-1, 5);
	\draw[thick, ggreen] (-1, 3) -- (0, 3);
	\draw[thick, ggreen] (-1, 1) -- (0, 1);
	\draw[thick, ggreen] (1, 2) -- (0, 2);
	\node[left] at (-1, 5.5) {$L_b X_1$, $1_\ell$};
	\draw[fill=black] (-1,5.5) circle (0.05);
	\node[left] at (-1, 5) {$X_2$, $2_\ell$};
	\draw[fill=black] (-1,5) circle (0.05);
	\node[left] at (-1, 4.5) {$L_b X_1$, $3_\ell$};
	\draw[fill=black] (-1,4.5) circle (0.05);
	\node[left] at (-1, 4) {$X_2$, $4_\ell$};
	\draw[fill=black] (-1,4) circle (0.05);
	\node[left] at (-1, 3.5) {$L_b X_1$, $5_\ell$};
	\draw[fill=black] (-1,3.5) circle (0.05);
	\node[left] at (-1, 3) {$X_2$, $6_\ell$};
	\draw[ggreen, fill=ggreen] (-1,3) circle (0.05);
	\node[left] at (-1, 2.5) {$L_b X_1$, $7_\ell$};
	\draw[fill=black] (-1,2.5) circle (0.05);
	\node[left] at (-1, 2) {$X_2$, $8_\ell$};
	\draw[fill=black] (-1,2) circle (0.05);
	\node[left] at (-1, 1.5) {$L_b X_1$, $9_\ell$};
	\draw[fill=black] (-1,1.5) circle (0.05);
	\node[left] at (-1, 1) {$X_2$, $10_\ell$};
	\draw[ggreen, fill=ggreen] (-1,1) circle (0.05);
	\draw[fill=black] (1,5.5) circle (0.05);
	\node[right] at (1,5.5) {$1_r$, $R_d Y_1$};
	\draw[ggreen, fill=ggreen] (1,5) circle (0.05);
	\node[right] at (1,5) {$2_r$, $Y_2$};
	\draw[fill=black] (1,4.5) circle (0.05);
	\node[right] at (1,4.5) {$3_r$, $R_d Y_1$};
	\draw[ggreen, fill=ggreen] (1,0) circle (0.05);
	\node[right] at (1,4) {$4_r$, $Y_2$};
	\draw[fill=black] (1,4) circle (0.05);
	\node[right] at (1,3.5) {$5_r$, $R_d Y_1$};
	\draw[fill=black] (1,3.5) circle (0.05);
	\node[right] at (1,3) {$6_r$, $Y_2$};
	\draw[fill=black] (1,3) circle (0.05);
	\node[right] at (1,2.5) {$7_r$, $R_d Y_1$};
	\draw[fill=black] (1,2.5) circle (0.05);
	\node[right] at (1,2) {$8_r$, $Y_2$};
	\draw[ggreen, fill=ggreen] (1,2) circle (0.05);
	\node[right] at (1,1.5) {$9_r$, $R_d Y_1$};
	\draw[fill=black] (1,1.5) circle (0.05);
	\node[right] at (1,1) {$10_r$, $Y_2$};
	\draw[fill=black] (1,1) circle (0.05);
	\node[right] at (1,0.5) {$11_r$, $R_d Y_1$};
	\draw[fill=black] (1,0.5) circle (0.05);
	\node[right] at (1,0) {$12_r$, $Y_2 R_c$};
	\end{tikzpicture}
\end{align*}

For each $1 \leq p \leq t$, let $i_p = l_p - l_{p-1}$, where $l_0 = 0$, and let $\pi_{\ell, p}$ denote the non-crossing partition obtained by restricting $\pi$ to $\{(2l_{p-1} + 1)_\ell, (2l_{p-1}+2)_\ell, \ldots, (2l_{p}-1)_\ell\}$.  Note $\pi_{\ell, p}$ is naturally an element of $BNC'_\ell(i_p)$ once a singleton block is added. 

Similarly, for each $1 \leq q \leq s$, let $j_q = k_q - k_{q-1}$, where $k_0 = 0$, and let $\pi_{r, q}$ denote the non-crossing partition obtained by restricting $\pi$ to $\{(2k_{q-1} + 1)_r, (2k_{q-1}+2)_r, \ldots, (2k_{q}-1)_r\}$.  Note $\pi_{r, q}$ is naturally an element of $BNC'_r(j_q)$  once a singleton block is added.

Consequently, if we sum over all possible $n,m\geq 1$, for each $V_\pi$ one obtains $\psi_\ell(L_b X_1, X_2)$ for the $B$-operator in each of the $t$ disjoint regions on the left and $\psi_r(R_d Y_1, Y_2)$ for the $B$-operator in each of the $s$ disjoint regions on the right.  Using bi-multiplicative properties (so the $\psi_\ell(L_b X_1, X_2)$ term attaches to a $X_2$ from above to obtain $L_{\psi_\ell(L_b X_1, X_2)} X_2$ and so the $\psi_r(R_d Y_1, Y_2)$ term attaches to a $Y_2$ from above to obtain $R_{\psi_r(R_d Y_1, Y_2)} Y_2$) and summing over all possible $V_\pi$ yields the result.
\end{proof}

In order to discuss $\Psi_o(b,c,d)$, it is quite helpful to discuss subcases.  For $n,m \geq 0$, let $\sigma'_{n,m} $ denote the element of $BNC(2n+1, 2m+1)$ with blocks $\{\{1_\ell, 1_r\}\} \cup \{\{(2l)_\ell, (2l+1)_\ell\}\}_{l=1}^n  \cup \{\{(2k)_r, (2k+1)_r\}\}_{k=1}^m$.  Let $BNC_S(n,m)'_o$ denote the set of all $\pi \in BNC(2n+1,2m+1)$ such that $\pi \vee \sigma'_{n,m} = 1_{2n+1,2m+1}$ and contains no blocks with both a $(2k)_{\theta_1}$ and a $(2j-1)_{\theta_2}$ any $\theta_1, \theta_2 \in \{\ell, r\}$ and any $k, j$.  We desired to divide up $BNC_S(n,m)'_o$ further.  For $\pi \in BNC_S(n,m)'_o$, let $V_{\pi, \ell}$ denote the block of $\pi$ containing $1_\ell$ and let $V_{\pi, r}$ denote the block of $\pi$ containing $1_r$.  Then
\begin{align*}
BNC_S(n,m)_{o,0} &= \{\pi \in BNC_S(n,m)'_o \, \mid \,  V_{\pi, \ell} \text{ has no right nodes and } V_{\pi, r} \text{ has no left nodes}\},\\
BNC_S(n,m)_{o,r} &= \{\pi \in BNC_S(n,m)'_o \, \mid \,  V_{\pi, \ell} \text{ has no right nodes but } V_{\pi, r} \text{ has left nodes}\}, \\
BNC_S(n,m)_{o,\ell} &= \{\pi \in BNC_S(n,m)'_o \, \mid \,  V_{\pi, \ell} \text{ has right nodes but } V_{\pi, r} \text{ has no left nodes}\}, \text{ and} \\
BNC_S(n,m)_{o, \ell r} &= \{\pi \in BNC_S(n,m)'_o \, \mid \, V_{\pi, \ell} = V_{\pi, r}\}.
\end{align*}
Due to the nature of bi-non-crossing partitions, the above sets are disjoint and have union $BNC_S(n,m)'_o$.  

For $p \in \{0, r, \ell, \ell r\}$, define
\begin{align*}
\Psi_{o, p}(b,c,d)  := \sum_{\substack{ n,m\geq 0 \\ \pi \in BNC_S(n,m)_{o,p}}} \kappa^B_\pi(X_2, \underbrace{L_{b}X_1, X_2, \ldots, L_b X_1, X_2}_{X_2, L_b X_1 \text{ each occur }n \text{ times}}, \underbrace{ Y_2, R_d Y_1, Y_2 \ldots, R_dY_1, Y_2, R_d Y_1}_{Y_2, R_d Y_1 \text{ occurs }m \text{ times}}, Y_2 R_c).
\end{align*}

\begin{lem}
\label{lem:S-case-2}
Under the above notation and assumptions, 
\[
\Psi_{o,0}(b,c,d) =  \psi_\ell(L_b X_1, X_2)^{-1} \Phi_{\ell, X_2}(\psi_\ell(L_b X_1, X_2)) c \Phi_{r, Y_2}(\psi_r(R_d Y_1, Y_2)) \psi_r(R_d Y_1, Y_2)^{-1}.
\]
\end{lem}
\begin{proof}
Fix $n,m \geq 0$.  If $\pi \in BNC_S(n,m)_{o,0}$, then, since $\pi \vee \sigma'_{n,m} = 1_{2n+1,2m+1}$, there exist $t,s\geq 1$, $1 = l_1 < l_2 < \cdots < l_t = n+1$, and $1 = k_1 < k_2 < \cdots < k_s = m+1$ such that
\[
V_{\pi, \ell} = \{(2l_p-1)_\ell\}^t_{p=1} \qqand  V_{\pi, r} = \{(2k_q-1)_r\}^s_{q=1}.
\]
Note $V_{\pi, \ell}$ divides the remaining left nodes into $t-1$ disjoint regions and $V_{\pi, r}$ divides the remaining right nodes into $s-1$ disjoint regions. Moreover, each block of $\pi$ can only contain nodes in one such region.  The following is an example of such a $\pi$.
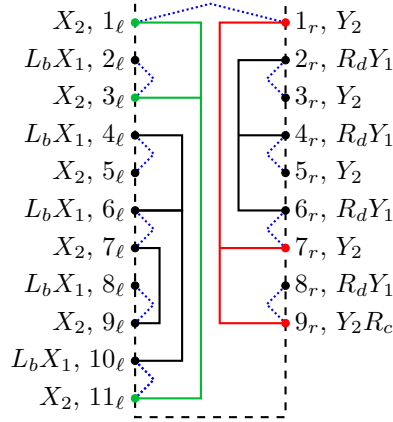
\begin{align*}
	\begin{tikzpicture}[baseline]
	\draw[thick, dashed] (-1,5.75) -- (-1,.25) -- (1,.25) -- (1,5.75);
	\draw[thick, blue, densely dotted] (1, 5.5) -- (0, 5.75) -- (-1, 5.5);
	\draw[thick, blue, densely dotted] (1, 5) -- (0.75, 4.75) -- (1, 4.5);
	\draw[thick, blue, densely dotted] (1, 4) -- (0.75, 3.75) -- (1, 3.5);
	\draw[thick, blue, densely dotted] (1, 3) -- (0.75, 2.75) -- (1, 2.5);
	\draw[thick, blue, densely dotted] (1, 2) -- (0.75, 1.75) -- (1, 1.5);
	\draw[thick, blue, densely dotted] (-1, 1) -- (-0.75, 0.75) -- (-1, 0.5);
	\draw[thick, blue, densely dotted] (-1, 5) -- (-0.75, 4.75) -- (-1, 4.5);
	\draw[thick, blue, densely dotted] (-1, 4) -- (-0.75, 3.75) -- (-1, 3.5);
	\draw[thick, blue, densely dotted] (-1, 3) -- (-0.75, 2.75) -- (-1, 2.5);
	\draw[thick, blue, densely dotted] (-1, 2) -- (-0.75, 1.75) -- (-1, 1.5);
	\draw[thick, blue, densely dotted] (-1, 1) -- (-0.75, 0.75) -- (-1, 0.5);
	\draw[thick] (1, 5) -- (.375, 5) -- (.375, 3) -- (1, 3);	
	\draw[thick] (1, 4) -- (.375, 4);
	\draw[thick, red] (1, 5.5) -- (.125, 5.5) -- (.125, 1.5) -- (1, 1.5);
	\draw[thick, ggreen] (-1, 5.5) -- (-.125, 5.5) -- (-.125, 0.5) -- (-1, 0.5);
	\draw[thick, ggreen] (-1, 4.5) -- (-.125, 4.5);
	\draw[thick] (-1, 2.5) -- (-.675, 2.5) -- (-.675, 1.5) -- (-1, 1.5);
	\draw[thick, red] (1, 2.5) -- (.125, 2.5);
	\draw[thick] (-1, 4) -- (-.375, 4) -- (-.375, 3) -- (-1, 3);
	\draw[thick] (-1, 3) -- (-.375, 3) -- (-.375, 1) -- (-1, 1);
	\node[left] at (-1, 5.5) {$X_2$, $1_\ell$};
	\draw[ggreen, fill=ggreen] (-1,5.5) circle (0.05);
	\node[left] at (-1, 5) {$L_b X_1$, $2_\ell$};
	\draw[fill=black] (-1,5) circle (0.05);
	\node[left] at (-1, 4.5) {$X_2$, $3_\ell$};
	\draw[ggreen, fill=ggreen] (-1,4.5) circle (0.05);
	\node[left] at (-1, 4) {$L_b X_1$, $4_\ell$};
	\draw[fill=black] (-1,4) circle (0.05);
	\node[left] at (-1, 3.5) {$X_2$, $5_\ell$};
	\draw[fill=black] (-1,3.5) circle (0.05);
	\node[left] at (-1, 3) {$L_b X_1$, $6_\ell$};
	\draw[fill=black] (-1,3) circle (0.05);
	\node[left] at (-1, 2.5) {$X_2$, $7_\ell$};
	\draw[black, fill=black] (-1,2.5) circle (0.05);
	\node[left] at (-1, 2) {$L_b X_1$, $8_\ell$};
	\draw[fill=black] (-1,2) circle (0.05);
	\node[left] at (-1, 1.5) {$X_2$, $9_\ell$};
	\draw[fill=black] (-1,1.5) circle (0.05);
	\node[left] at (-1, 1) {$L_b X_1$, $10_\ell$};
	\draw[fill=black] (-1,1) circle (0.05);
	\node[left] at (-1, .5) {$X_2$, $11_\ell$};
	\draw[ggreen, fill=ggreen] (-1,.5) circle (0.05);
	\draw[red, fill=red] (1,5.5) circle (0.05);
	\node[right] at (1,5.5) {$1_r$, $Y_2$};
	\draw[fill=black] (1,5) circle (0.05);
	\node[right] at (1,5) {$2_r$, $R_d Y_1$};
	\draw[fill=black] (1,4.5) circle (0.05);
	\node[right] at (1,4.5) {$3_r$, $Y_2$};
	\node[right] at (1,4) {$4_r$, $R_d Y_1$};
	\draw[fill=black] (1,4) circle (0.05);
	\node[right] at (1,3.5) {$5_r$, $Y_2$};
	\draw[fill=black] (1,3.5) circle (0.05);
	\node[right] at (1,3) {$6_r$, $R_d Y_1$};
	\draw[fill=black] (1,3) circle (0.05);
	\node[right] at (1,2.5) {$7_r$, $Y_2$};
	\draw[red, fill=red] (1,2.5) circle (0.05);
	\node[right] at (1,2) {$8_r$, $R_d Y_1$};
	\draw[fill=black] (1,2) circle (0.05);
	\node[right] at (1,1.5) {$9_r$, $Y_2 R_c$};
	\draw[red, fill=red] (1,1.5) circle (0.05);
	\end{tikzpicture}
\end{align*}

Consequently, if we sum over all possible $n,m\geq 1$, for each $V_\pi$ one obtains $\psi_\ell(L_b X_1, X_2)$ for the $B$-operator in each of the $t-1$ disjoint regions on the left and $\psi_r(R_d Y_1, Y_2)$ for the $B$-operator in each of the $s-1$ disjoint regions on the right.  Using bi-multiplicative properties (so the $\psi_\ell(L_b X_1, X_2)$ term attaches to a $X_2$ from above to obtain $L_{\psi_\ell(L_b X_1, X_2)} X_2$ and so the $\psi_r(R_d Y_1, Y_2)$ term attaches to a $Y_2$ from above to obtain $R_{\psi_r(R_d Y_1, Y_2)} Y_2$) and summing over all possible $V_\pi$ yields the result.  Note the $\psi_\ell(L_b X_1, X_2)^{-1}$ (respectively $\psi_r(R_d Y_1, Y_2)^{-1}$) occur since one copy of $\psi_\ell(L_b X_1, X_2)$ (respectively $\psi_r(R_d Y_1, Y_2)$) is missing on $1_\ell$ (respectively $1_r$) $X_2$-term (respectively $Y_2$-term) of the product in order to obtain $\Phi_{\ell, X_2}(\psi_\ell(L_b X_1, X_2))$ (respectively $\Phi_{r, Y_2}(\psi_r(R_d Y_1, Y_2))$).  Note this division is valid by Lemma \ref{lem:divide}.
\end{proof}

\begin{lem}
\label{lem:S-case-3}
Under the above notation and assumptions, 
\[
\Psi_{o,r}(b,c,d) = \psi_\ell(X_2, L_b X_1) K_{X_2, Y_2}( \psi_\ell(L_b X_1, X_2), c, \psi_r(R_d Y_1, Y_2)) \psi_r(R_d Y_1, Y_2)^{-1} .
\]
\end{lem}
\begin{proof}
Fix $n,m \geq 0$.  Note $BNC_S(0,m)_{o,r} = \emptyset$   by definition.

If $\pi \in BNC_S(n,m)_{o,r}$, then, since $\pi \vee \sigma'_{n,m} = 1_{2n+1,2m+1}$, there exist $t,s\geq 1$, $1 < l_1 < l_2 < \cdots < l_t = n+1$, and $1 = k_1 < k_2 < \cdots < k_s = m+1$ such that
\[
V_{\pi, r} = \{(2l_p-1)_\ell\}^t_{p=1}  \cup  \{(2k_q-1)_r\}^s_{q=1}.
\]
Note $V_{\pi, r}$ divides the remaining right nodes into $s-1$ disjoint regions and the remaining left nodes into $t$ regions.  However, the top region is special.  If $l_0$ is the largest natural number such that $(2l_0 -1)_\ell \in V_{\pi, \ell}$, then $l_0$ further divides the top region on the left into two regions. Note each block of $\pi$ can only contain nodes in one such region.
The following is an example of such a $\pi$ for which $l_0 = 3$, with one part of the special region ($1_\ell, \ldots, 5_\ell$) shaded differently.
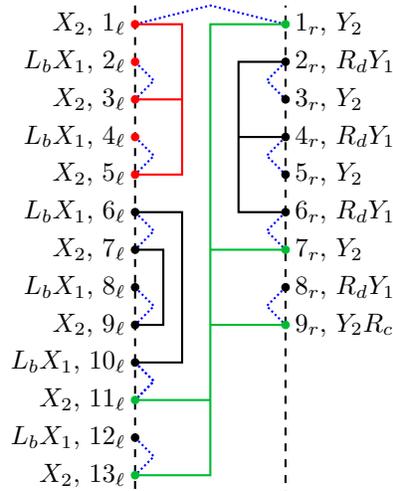
\begin{align*}
	\begin{tikzpicture}[baseline]
	\draw[thick, dashed] (-1,5.75) -- (-1,-.75) -- (1,-.75) -- (1,5.75);
	\draw[thick, blue, densely dotted] (1, 5.5) -- (0, 5.75) -- (-1, 5.5);
	\draw[thick, blue, densely dotted] (1, 5) -- (0.75, 4.75) -- (1, 4.5);
	\draw[thick, blue, densely dotted] (1, 4) -- (0.75, 3.75) -- (1, 3.5);
	\draw[thick, blue, densely dotted] (1, 3) -- (0.75, 2.75) -- (1, 2.5);
	\draw[thick, blue, densely dotted] (1, 2) -- (0.75, 1.75) -- (1, 1.5);
	\draw[thick, blue, densely dotted] (-1, 1) -- (-0.75, 0.75) -- (-1, 0.5);
	\draw[thick, blue, densely dotted] (-1, 5) -- (-0.75, 4.75) -- (-1, 4.5);
	\draw[thick, blue, densely dotted] (-1, 4) -- (-0.75, 3.75) -- (-1, 3.5);
	\draw[thick, blue, densely dotted] (-1, 3) -- (-0.75, 2.75) -- (-1, 2.5);
	\draw[thick, blue, densely dotted] (-1, 2) -- (-0.75, 1.75) -- (-1, 1.5);
	\draw[thick, blue, densely dotted] (-1, 1) -- (-0.75, 0.75) -- (-1, 0.5);
	\draw[thick, blue, densely dotted] (-1, 0) -- (-0.75, -0.25) -- (-1, -0.5);
	\draw[thick] (1, 5) -- (.375, 5) -- (.375, 3) -- (1, 3);	
	\draw[thick] (1, 4) -- (.375, 4);
	\draw[thick, ggreen] (1, 5.5) -- (0, 5.5) -- (0, -.5) -- (-1, -.5);
	\draw[thick, ggreen] (-1, .5) -- (0, .5);
	\draw[thick, ggreen] (0, 1.5) -- (1, 1.5);
	\draw[thick, ggreen] (1, 2.5) -- (0, 2.5);
	\draw[thick] (-1,1) -- (-.375, 1) -- (-.375, 3) -- (-1,3);
	\draw[thick] (-1,1.5) -- (-.625, 1.5) -- (-.625, 2.5) -- (-1,2.5);
	\draw[thick, red] (-1, 5.5) -- (-.375, 5.5) -- (-.375, 3.5) -- (-1, 3.5);
	\draw[thick, red] (-1, 4.5) -- (-.375, 4.5);
	\node[left] at (-1, 5.5) {$X_2$, $1_\ell$};
	\draw[red, fill=red] (-1,5.5) circle (0.05);
	\node[left] at (-1, 5) {$L_b X_1$, $2_\ell$};
	\draw[red, fill=red] (-1,5) circle (0.05);
	\node[left] at (-1, 4.5) {$X_2$, $3_\ell$};
	\draw[red, fill=red] (-1,4.5) circle (0.05);
	\node[left] at (-1, 4) {$L_b X_1$, $4_\ell$};
	\draw[red, fill=red] (-1,4) circle (0.05);
	\node[left] at (-1, 3.5) {$X_2$, $5_\ell$};
	\draw[red, fill=red] (-1,3.5) circle (0.05);
	\node[left] at (-1, 3) {$L_b X_1$, $6_\ell$};
	\draw[fill=black] (-1,3) circle (0.05);
	\node[left] at (-1, 2.5) {$X_2$, $7_\ell$};
	\draw[fill=black] (-1,2.5) circle (0.05);
	\node[left] at (-1, 2) {$L_b X_1$, $8_\ell$};
	\draw[fill=black] (-1,2) circle (0.05);
	\node[left] at (-1, 1.5) {$X_2$, $9_\ell$};
	\draw[ggreen, fill=ggreen] (-1,.5) circle (0.05);
	\node[left] at (-1, 1) {$L_b X_1$, $10_\ell$};
	\draw[fill=black] (-1,1) circle (0.05);
	\node[left] at (-1, .5) {$X_2$, $11_\ell$};
	\draw[fill=black] (-1,1.5) circle (0.05);
	\node[left] at (-1, 0) {$L_b X_1$, $12_\ell$};
	\draw[fill=black] (-1,0) circle (0.05);
	\node[left] at (-1, -.5) {$X_2$, $13_\ell$};
	\draw[ggreen, fill=ggreen] (-1,-.5) circle (0.05);
	\draw[ggreen, fill=ggreen] (1,5.5) circle (0.05);
	\node[right] at (1,5.5) {$1_r$, $Y_2$};
	\draw[fill=black] (1,5) circle (0.05);
	\node[right] at (1,5) {$2_r$, $R_d Y_1$};
	\draw[fill=black] (1,4.5) circle (0.05);
	\node[right] at (1,4.5) {$3_r$, $Y_2$};
	\node[right] at (1,4) {$4_r$, $R_d Y_1$};
	\draw[fill=black] (1,4) circle (0.05);
	\node[right] at (1,3.5) {$5_r$, $Y_2$};
	\draw[fill=black] (1,3.5) circle (0.05);
	\node[right] at (1,3) {$6_r$, $R_d Y_1$};
	\draw[fill=black] (1,3) circle (0.05);
	\node[right] at (1,2.5) {$7_r$, $Y_2$};
	\draw[ggreen, fill=ggreen] (1,2.5) circle (0.05);
	\node[right] at (1,2) {$8_r$, $R_d Y_1$};
	\draw[fill=black] (1,2) circle (0.05);
	\node[right] at (1,1.5) {$9_r$, $Y_2 R_c$};
	\draw[ggreen, fill=ggreen] (1,1.5) circle (0.05);
	\end{tikzpicture}
\end{align*}

Consequently, if we sum over all possible $n,m\geq 1$, for each $V_\pi$ one obtains $\psi_r(R_d Y_1, Y_2)$ for the $B$-operator in each of the $s-1$ disjoint regions on the right and $\psi_\ell(L_b X_1, X_2)$ for the $B$-operator in each of the $t-1$ disjoint regions on the left, excluding the top region where one obtains $\psi_\ell(X_2, L_b X_1)\psi_\ell(L_b X_1, X_2)$.  Using bi-multiplicative properties (so the $\psi_\ell(L_b X_1, X_2)$ term attaches to a $X_2$ from above to obtain $L_{\psi_\ell(L_b X_1, X_2)} X_2$ and so the $\psi_r(R_d Y_1, Y_2)$ term attaches to a $Y_2$ from above to obtain $R_{\psi_r(R_d Y_1, Y_2)} Y_2$) and summing over all possible $V_\pi$ yields the result.  Note the $\psi_\ell(X_2, L_b X_1)$ remains on the left of the product since it is not needed in the $K_{X_2, Y_2}$-term and $\psi_r(R_d Y_1, Y_2)^{-1}$ occurs since one copy of $\psi_r(R_d Y_1, Y_2)$ is missing on $1_r$  $Y_2$-term of the product.
\end{proof}

\begin{lem}
\label{lem:S-case-4}
Under the above notation and assumptions,
\[
\Psi_{o,\ell}(b,c,d) = \psi_\ell(L_b X_1, X_2)^{-1} K_{X_2, Y_2}(\psi_\ell(L_b X_1, X_2), c, \psi_r(R_d Y_1, Y_2)) \psi_r(Y_2, R_d Y_1).
\]
\end{lem}
\begin{proof}
The proof of this result can be obtained by applying a mirror to Lemma \ref{lem:S-case-3}.
\end{proof}

\begin{lem}
\label{lem:S-case-5}
Under the above notation and assumptions,
\[
\Psi_{o, \ell r}(b,c,d) = \psi_\ell(L_b X_1, X_2)^{-1} K_{X_2, Y_2}(\psi_\ell(L_b X_1, X_2), c, \psi_r(R_d Y_1, Y_2)) \psi_r(R_d Y_1, Y_2)^{-1}.
\]
\end{lem}
\begin{proof}
The proof of this result follows from the proof of Lemma \ref{lem:S-case-2}.  Indeed there is a bijection from $BNC_S(n,m)_{o, 0}$ to $BNC_S(n,m)_{o,\ell r}$ where given $\pi \in BNC_S(n,m)_{o,0}$ we produce $\pi'\in BNC_S(n,m)_{o,\ell r}$ by joining $V_{\pi, \ell}$ and $V_{\pi, r}$ into a single block which, under the same arguments, yields the result. 
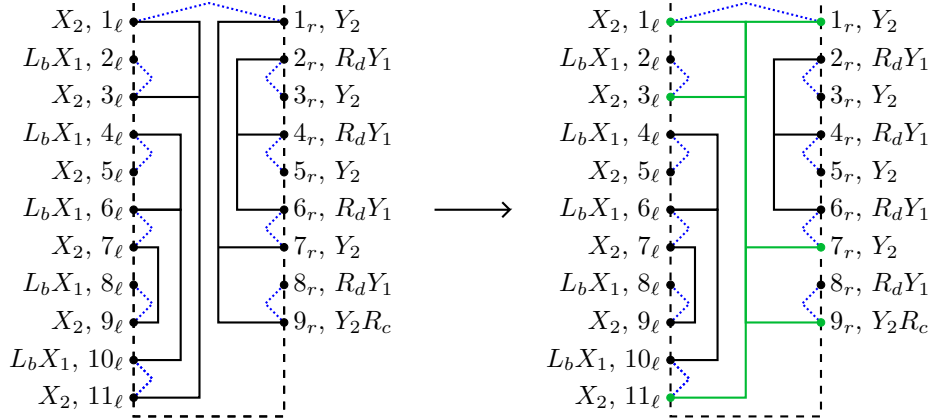
\begin{align*}
	\begin{tikzpicture}[baseline]
	\draw[thick, dashed] (-1,5.75) -- (-1,.25) -- (1,.25) -- (1,5.75);
	\draw[thick, blue, densely dotted] (1, 5.5) -- (0, 5.75) -- (-1, 5.5);
	\draw[thick, blue, densely dotted] (1, 5) -- (0.75, 4.75) -- (1, 4.5);
	\draw[thick, blue, densely dotted] (1, 4) -- (0.75, 3.75) -- (1, 3.5);
	\draw[thick, blue, densely dotted] (1, 3) -- (0.75, 2.75) -- (1, 2.5);
	\draw[thick, blue, densely dotted] (1, 2) -- (0.75, 1.75) -- (1, 1.5);
	\draw[thick, blue, densely dotted] (-1, 1) -- (-0.75, 0.75) -- (-1, 0.5);
	\draw[thick, blue, densely dotted] (-1, 5) -- (-0.75, 4.75) -- (-1, 4.5);
	\draw[thick, blue, densely dotted] (-1, 4) -- (-0.75, 3.75) -- (-1, 3.5);
	\draw[thick, blue, densely dotted] (-1, 3) -- (-0.75, 2.75) -- (-1, 2.5);
	\draw[thick, blue, densely dotted] (-1, 2) -- (-0.75, 1.75) -- (-1, 1.5);
	\draw[thick, blue, densely dotted] (-1, 1) -- (-0.75, 0.75) -- (-1, 0.5);
	\draw[thick] (1, 5) -- (.375, 5) -- (.375, 3) -- (1, 3);	
	\draw[thick] (1, 4) -- (.375, 4);
	\draw[thick] (1, 5.5) -- (.125, 5.5) -- (.125, 1.5) -- (1, 1.5);
	\draw[thick] (-1, 5.5) -- (-.125, 5.5) -- (-.125, 0.5) -- (-1, 0.5);
	\draw[thick] (-1, 4.5) -- (-.125, 4.5);
	\draw[thick] (-1, 2.5) -- (-.675, 2.5) -- (-.675, 1.5) -- (-1, 1.5);
	\draw[thick] (1, 2.5) -- (.125, 2.5);
	\draw[thick] (-1, 4) -- (-.375, 4) -- (-.375, 3) -- (-1, 3);
	\draw[thick] (-1, 3) -- (-.375, 3) -- (-.375, 1) -- (-1, 1);
	\node[left] at (-1, 5.5) {$X_2$, $1_\ell$};
	\draw[black, fill=black] (-1,5.5) circle (0.05);
	\node[left] at (-1, 5) {$L_b X_1$, $2_\ell$};
	\draw[fill=black] (-1,5) circle (0.05);
	\node[left] at (-1, 4.5) {$X_2$, $3_\ell$};
	\draw[black, fill=black] (-1,4.5) circle (0.05);
	\node[left] at (-1, 4) {$L_b X_1$, $4_\ell$};
	\draw[fill=black] (-1,4) circle (0.05);
	\node[left] at (-1, 3.5) {$X_2$, $5_\ell$};
	\draw[fill=black] (-1,3.5) circle (0.05);
	\node[left] at (-1, 3) {$L_b X_1$, $6_\ell$};
	\draw[fill=black] (-1,3) circle (0.05);
	\node[left] at (-1, 2.5) {$X_2$, $7_\ell$};
	\draw[black, fill=black] (-1,2.5) circle (0.05);
	\node[left] at (-1, 2) {$L_b X_1$, $8_\ell$};
	\draw[fill=black] (-1,2) circle (0.05);
	\node[left] at (-1, 1.5) {$X_2$, $9_\ell$};
	\draw[fill=black] (-1,1.5) circle (0.05);
	\node[left] at (-1, 1) {$L_b X_1$, $10_\ell$};
	\draw[fill=black] (-1,1) circle (0.05);
	\node[left] at (-1, .5) {$X_2$, $11_\ell$};
	\draw[black, fill=black] (-1,.5) circle (0.05);
	\draw[black, fill=black] (1,5.5) circle (0.05);
	\node[right] at (1,5.5) {$1_r$, $Y_2$};
	\draw[fill=black] (1,5) circle (0.05);
	\node[right] at (1,5) {$2_r$, $R_d Y_1$};
	\draw[fill=black] (1,4.5) circle (0.05);
	\node[right] at (1,4.5) {$3_r$, $Y_2$};
	\node[right] at (1,4) {$4_r$, $R_d Y_1$};
	\draw[fill=black] (1,4) circle (0.05);
	\node[right] at (1,3.5) {$5_r$, $Y_2$};
	\draw[fill=black] (1,3.5) circle (0.05);
	\node[right] at (1,3) {$6_r$, $R_d Y_1$};
	\draw[fill=black] (1,3) circle (0.05);
	\node[right] at (1,2.5) {$7_r$, $Y_2$};
	\draw[black, fill=black] (1,2.5) circle (0.05);
	\node[right] at (1,2) {$8_r$, $R_d Y_1$};
	\draw[fill=black] (1,2) circle (0.05);
	\node[right] at (1,1.5) {$9_r$, $Y_2R_c$};
	\draw[black, fill=black] (1,1.5) circle (0.05);
	\draw[thick, dashed] (-1,5.75) -- (-1,.25) -- (1,.25) -- (1,5.75);
	\draw[thick] (3, 3) -- (4,3) -- (3.90, 2.90);
	\draw[thick] (4,3) -- (3.90, 3.1);
	\end{tikzpicture}
	\quad	
	\begin{tikzpicture}[baseline]
	\draw[thick, dashed] (-1,5.75) -- (-1,.25) -- (1,.25) -- (1,5.75);
	\draw[thick, blue, densely dotted] (1, 5.5) -- (0, 5.75) -- (-1, 5.5);
	\draw[thick, blue, densely dotted] (1, 5) -- (0.75, 4.75) -- (1, 4.5);
	\draw[thick, blue, densely dotted] (1, 4) -- (0.75, 3.75) -- (1, 3.5);
	\draw[thick, blue, densely dotted] (1, 3) -- (0.75, 2.75) -- (1, 2.5);
	\draw[thick, blue, densely dotted] (1, 2) -- (0.75, 1.75) -- (1, 1.5);
	\draw[thick, blue, densely dotted] (-1, 1) -- (-0.75, 0.75) -- (-1, 0.5);
	\draw[thick, blue, densely dotted] (-1, 5) -- (-0.75, 4.75) -- (-1, 4.5);
	\draw[thick, blue, densely dotted] (-1, 4) -- (-0.75, 3.75) -- (-1, 3.5);
	\draw[thick, blue, densely dotted] (-1, 3) -- (-0.75, 2.75) -- (-1, 2.5);
	\draw[thick, blue, densely dotted] (-1, 2) -- (-0.75, 1.75) -- (-1, 1.5);
	\draw[thick, blue, densely dotted] (-1, 1) -- (-0.75, 0.75) -- (-1, 0.5);
	\draw[thick] (1, 5) -- (.375, 5) -- (.375, 3) -- (1, 3);	
	\draw[thick] (1, 4) -- (.375, 4);
	\draw[thick, ggreen] (1, 5.5) -- (0, 5.5) -- (0, 1.5) -- (1, 1.5);
	\draw[thick, ggreen] (-1, 5.5) -- (0, 5.5) -- (0, 0.5) -- (-1, 0.5);
	\draw[thick, ggreen] (-1, 4.5) -- (0, 4.5);
	\draw[thick] (-1, 2.5) -- (-.675, 2.5) -- (-.675, 1.5) -- (-1, 1.5);
	\draw[thick, ggreen] (1, 2.5) -- (0, 2.5);
	\draw[thick] (-1, 4) -- (-.375, 4) -- (-.375, 3) -- (-1, 3);
	\draw[thick] (-1, 3) -- (-.375, 3) -- (-.375, 1) -- (-1, 1);
	\node[left] at (-1, 5.5) {$X_2$, $1_\ell$};
	\draw[ggreen, fill=ggreen] (-1,5.5) circle (0.05);
	\node[left] at (-1, 5) {$L_b X_1$, $2_\ell$};
	\draw[fill=black] (-1,5) circle (0.05);
	\node[left] at (-1, 4.5) {$X_2$, $3_\ell$};
	\draw[ggreen, fill=ggreen] (-1,4.5) circle (0.05);
	\node[left] at (-1, 4) {$L_b X_1$, $4_\ell$};
	\draw[fill=black] (-1,4) circle (0.05);
	\node[left] at (-1, 3.5) {$X_2$, $5_\ell$};
	\draw[fill=black] (-1,3.5) circle (0.05);
	\node[left] at (-1, 3) {$L_b X_1$, $6_\ell$};
	\draw[fill=black] (-1,3) circle (0.05);
	\node[left] at (-1, 2.5) {$X_2$, $7_\ell$};
	\draw[black, fill=black] (-1,2.5) circle (0.05);
	\node[left] at (-1, 2) {$L_b X_1$, $8_\ell$};
	\draw[fill=black] (-1,2) circle (0.05);
	\node[left] at (-1, 1.5) {$X_2$, $9_\ell$};
	\draw[fill=black] (-1,1.5) circle (0.05);
	\node[left] at (-1, 1) {$L_b X_1$, $10_\ell$};
	\draw[fill=black] (-1,1) circle (0.05);
	\node[left] at (-1, .5) {$X_2$, $11_\ell$};
	\draw[ggreen, fill=ggreen] (-1,.5) circle (0.05);
	\draw[ggreen, fill=ggreen] (1,5.5) circle (0.05);
	\node[right] at (1,5.5) {$1_r$, $Y_2$};
	\draw[fill=black] (1,5) circle (0.05);
	\node[right] at (1,5) {$2_r$, $R_d Y_1$};
	\draw[fill=black] (1,4.5) circle (0.05);
	\node[right] at (1,4.5) {$3_r$, $Y_2$};
	\node[right] at (1,4) {$4_r$, $R_d Y_1$};
	\draw[fill=black] (1,4) circle (0.05);
	\node[right] at (1,3.5) {$5_r$, $Y_2$};
	\draw[fill=black] (1,3.5) circle (0.05);
	\node[right] at (1,3) {$6_r$, $R_d Y_1$};
	\draw[fill=black] (1,3) circle (0.05);
	\node[right] at (1,2.5) {$7_r$, $Y_2$};
	\draw[ggreen, fill=ggreen] (1,2.5) circle (0.05);
	\node[right] at (1,2) {$8_r$, $R_d Y_1$};
	\draw[fill=black] (1,2) circle (0.05);
	\node[right] at (1,1.5) {$9_r$, $Y_2R_c$};
	\draw[ggreen, fill=ggreen] (1,1.5) circle (0.05);
	\end{tikzpicture}
\end{align*}
\end{proof}

\begin{lem}
\label{lem:S-case-6}
Under the above notation and assumptions, 
\[
\Psi_{o}(b,c,d) =   b (\psi_\ell(X_2, L_b X_1)b)^{-1} K_{X_2, Y_2}(\psi_\ell(X_2, L_b X_1)b, \Psi_{o'}(b,c,d), d \psi_r(Y_2, R_d Y_1))(d \psi_r(Y_2, R_d Y_1))^{-1} d
\]
where
\[
\Psi_{o'}(b,c,d) := \Psi_{o,0}(b,c,d) + \Psi_{o,r}(b,c,d) + \Psi_{o,\ell}(b,c,d) + \Psi_{o,\ell r}(b,c,d).
\]
\end{lem}
\begin{proof}

Fix $n,m \geq 1$.  If $\pi \in BNC_S(n,m)_o$, let $V_\pi$ denote the first block of $\pi$, as measured from the top of $\pi$'s bi-non-crossing diagram, that has both left and right nodes.  Since $\pi \in BNC_S(n,m)_o$, there exist $t,s\geq 1$, $1= l_1 < l_2 < \cdots < l_t \leq n$, and $1 = k_1 < k_2 < \cdots < k_s \leq m$ such that
\[
V_\pi = \{(2l_p-1)_\ell\}^t_{p=1} \cup \{(2k_q-1)_r\}^s_{q=1}.
\]
Note $V_\pi$ divides the remaining left nodes and right nodes into $t-1$ disjoint regions on the left, $s-1$ disjoint regions on the right, and one region on the bottom. Below is an example of such a $\pi$.
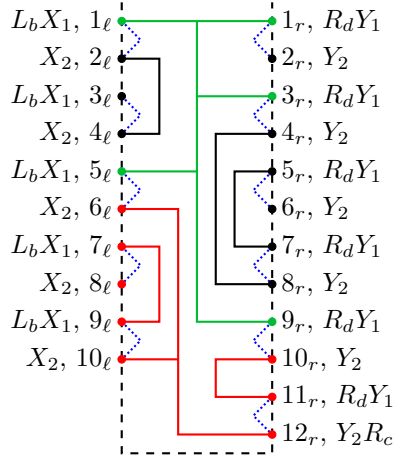
\begin{align*}
	\begin{tikzpicture}[baseline]
	\draw[thick, dashed] (-1,5.75) -- (-1,-.25) -- (1,-.25) -- (1,5.75);
	\draw[thick, blue, densely dotted] (1, 5.5) -- (0.75, 5.25) -- (1, 5);
	\draw[thick, blue, densely dotted] (1, 4.5) -- (0.75, 4.25) -- (1, 4);
	\draw[thick, blue, densely dotted] (1, 3.5) -- (0.75, 3.25) -- (1, 3);
	\draw[thick, blue, densely dotted] (1, 2.5) -- (0.75, 2.25) -- (1, 2);
	\draw[thick, blue, densely dotted] (1, 1.5) -- (0.75, 1.25) -- (1, 1);
	\draw[thick, blue, densely dotted] (1, 0.5) -- (0.75, 0.25) -- (1, 0);
	\draw[thick, blue, densely dotted] (-1, 5.5) -- (-0.75, 5.25) -- (-1, 5);
	\draw[thick, blue, densely dotted] (-1, 4.5) -- (-0.75, 4.25) -- (-1, 4);
	\draw[thick, blue, densely dotted] (-1, 3.5) -- (-0.75, 3.25) -- (-1, 3);
	\draw[thick, blue, densely dotted] (-1, 2.5) -- (-0.75, 2.25) -- (-1, 2);
	\draw[thick, blue, densely dotted] (-1, 1.5) -- (-0.75, 1.25) -- (-1, 1);
	\draw[thick, ggreen] (-1,5.5) -- (1,5.5);
	\draw[thick, ggreen] (0, 5.5) -- (0, 1.5) -- (1, 1.5);
	\draw[thick, ggreen] (1, 4.5) -- (0, 4.5);
	\draw[thick, ggreen] (-1, 3.5) -- (0, 3.5);
	\draw[thick] (-1, 5) -- (-0.5, 5) -- (-0.5, 4) -- (-1, 4);
	\draw[thick] (1, 4) -- (.25, 4) -- (.25, 2) -- (1, 2);
	\draw[thick] (1, 3.5) -- (.5, 3.5) -- (.5, 2.5) -- (1, 2.5);
	\draw[thick, red] (1, 1) -- (.25, 1) -- (.25, .5) -- (1, .5);
	\draw[thick, red] (-1, 3) -- (-.25, 3) -- (-.25, 0) -- (1, 0);
	\draw[thick, red] (-1, 1.5) -- (-.5, 1.5) -- (-.5, 2.5) -- (-1, 2.5);
	\draw[thick, red] (-1, 1) -- (-.25, 1);
	\node[left] at (-1, 5.5) {$L_b X_1$, $1_\ell$};
	\draw[ggreen, fill=ggreen] (-1,5.5) circle (0.05);
	\node[left] at (-1, 5) {$X_2$, $2_\ell$};
	\draw[fill=black] (-1,5) circle (0.05);
	\node[left] at (-1, 4.5) {$L_b X_1$, $3_\ell$};
	\draw[fill=black] (-1,4.5) circle (0.05);
	\node[left] at (-1, 4) {$X_2$, $4_\ell$};
	\draw[fill=black] (-1,4) circle (0.05);
	\node[left] at (-1, 3.5) {$L_b X_1$, $5_\ell$};
	\draw[ggreen, fill=ggreen] (-1,3.5) circle (0.05);
	\node[left] at (-1, 3) {$X_2$, $6_\ell$};
	\draw[red, fill=red] (-1,3) circle (0.05);
	\node[left] at (-1, 2.5) {$L_b X_1$, $7_\ell$};
	\draw[red, fill=red] (-1,2.5) circle (0.05);
	\node[left] at (-1, 2) {$X_2$, $8_\ell$};
	\draw[red, fill=red] (-1,2) circle (0.05);
	\node[left] at (-1, 1.5) {$L_b X_1$, $9_\ell$};
	\draw[red, fill=red] (-1,1.5) circle (0.05);
	\node[left] at (-1, 1) {$X_2$, $10_\ell$};
	\draw[red, fill=red] (-1,1) circle (0.05);
	\draw[ggreen, fill=ggreen] (1,5.5) circle (0.05);
	\node[right] at (1,5.5) {$1_r$, $R_d Y_1$};
	\draw[fill=black] (1,5) circle (0.05);
	\node[right] at (1,5) {$2_r$, $Y_2$};
	\draw[ggreen, fill=ggreen] (1,4.5) circle (0.05);
	\node[right] at (1,4.5) {$3_r$, $R_d Y_1$};
	\draw[red, fill=red] (1,0) circle (0.05);
	\node[right] at (1,4) {$4_r$, $Y_2$};
	\draw[fill=black] (1,4) circle (0.05);
	\node[right] at (1,3.5) {$5_r$, $R_d Y_1$};
	\draw[fill=black] (1,3.5) circle (0.05);
	\node[right] at (1,3) {$6_r$, $Y_2$};
	\draw[fill=black] (1,3) circle (0.05);
	\node[right] at (1,2.5) {$7_r$, $R_d Y_1$};
	\draw[black, fill=black] (1,2.5) circle (0.05);
	\node[right] at (1,2) {$8_r$, $Y_2$};
	\draw[fill=black] (1,2) circle (0.05);
	\node[right] at (1,1.5) {$9_r$, $R_d Y_1$};
	\draw[ggreen, fill=ggreen] (1,1.5) circle (0.05);
	\node[right] at (1,1) {$10_r$, $Y_2$};
	\draw[red, fill=red] (1,1) circle (0.05);
	\node[right] at (1,0.5) {$11_r$, $R_d Y_1$};
	\draw[red, fill=red] (1,0.5) circle (0.05);
	\node[right] at (1,0) {$12_r$, $Y_2R_c$};
	\end{tikzpicture}
\end{align*}

For each $1 \leq p \leq t$, let $i_p = l_{p+1} - l_{p}$, where $l_{t+1} = n+1$, and, for $p \neq t$, let $\pi_{\ell, p}$ denote the non-crossing partition obtained by restricting $\pi$ to $\{(2l_{p})_\ell, (2l_{p}+1)_\ell, \ldots, (2l_{p+1}-2)_\ell\}$.  Note $\pi_{\ell, p}$  is naturally an element of $BNC'_\ell(i_p)$ once a singleton block is added. 

Similarly, for each $1 \leq q \leq s$, let $j_q = k_{q+1} - k_{q}$, where $k_{s+1} = m+1$, and, for $q \neq s$, let $\pi_{r, q}$ denote the non-crossing partition obtained by restricting $\pi$ to $\{(2k_{q})_r, (2k_{q}+1)_r, \ldots, (2k_{q+1}-2)_r\}$.  Note $\pi_{r, q}$ is naturally an element of $BNC'_r(j_q)$ once a singleton block is added.

Finally, if $\pi'$ is the bi-non-crossing partition obtained by restricting $\pi$ to 
\[
\{(2l_{t})_\ell, (2l_{t}+1)_\ell, \ldots, (2n)_\ell, (2k_{s})_r, (2k_{s}+1)_r, \ldots, (2m)_r\}
\]
(which is shaded differently in the above diagram), then $\pi' \in BNC_S(i_t-1,j_s-1)'_o$.

Consequently, if we sum over all possible $n,m\geq 1$, for each $V_\pi$ one obtains $\psi_\ell(X_2, L_b X_1)$ for the $B$-operator in each of the $t-1$ disjoint regions on the left and $\psi_r(Y_2, R_d Y_1)$ for the $B$-operator in each of the $s-1$ disjoint regions on the right, and $\Psi_{o'}(b,c,d)$ for the bottom region.  Using bi-multiplicative properties (so the $\psi_\ell(X_2, L_b X_1)$ term attaches to a $L_b X_1$ from above to obtain $L_{\psi_\ell(X_2, L_b X_1)b} X_1$ and so the $\psi_r(Y_2, R_d Y_1)$ term attaches to a $R_d Y_1$ from above to obtain $R_{d \psi_r(Y_2, R_d Y_1)} Y_1$) and summing over all possible $V_\pi$ yields the result.  Note the `$b$'  in the $b (\psi_\ell(X_2, L_b X_1)b)^{-1}$ term (respectively the `$d$'  in the $(d \psi_r(Y_2, R_d Y_1))^{-1} d$) term comes from the $L_b X_1$ (respectively $R_d Y_1$) in the $1_\ell$ (respectively $1_r$) position whereas the $(\psi_\ell(X_2, L_b X_1)b)^{-1}$ (respectively $(d \psi_r(Y_2, R_d Y_1))^{-1}$) comes from the fact that we want the $1_\ell$ (respectively $1_r$) term to be $L_{\psi_\ell(X_2, L_b X_1)b} X_1$ (respectively $R_{d \psi_r(Y_2, R_d Y_1) }X_1$) to match the term in $K_{X_2, Y_2}(\psi_\ell(X_2, L_b X_1)b, \Psi_{o'}(b,c,d), d \psi_r(Y_2, R_d Y_1))$ and bi-multiplicative properties correct this.
\end{proof}

\begin{proof}[Proof of Theorem \ref{thm:S-property}]

We desired to replace $b$ with $\Phi^\inv_{\ell, X_1X_2}(b)$ and replace $d$ with $\Phi^{\inv}_{r, Y_1Y_2}(d)$ in each expression from Lemmata \ref{lem:S-case-1}, \ref{lem:S-case-2}, \ref{lem:S-case-3}, \ref{lem:S-case-4}, \ref{lem:S-case-5}, and \ref{lem:S-case-6}.  Note
\[
\psi_\ell\left( L_{\Phi^\inv_{\ell, X_1X_2}(b)} X_1, X_2\right) = \Phi^\inv_{\ell, X_2}(b) \qqand  \psi_r\left(R_{\Phi^{\inv}_{r, Y_1Y_2}(d)} Y_1, Y_2\right) = \Phi^{\inv}_{r, Y_2}(d)
\]
by equation (\ref{eq:pinched-to-inverse-of-full}), 
\[
\psi_\ell(X_2, L_b Y_1)b = \psi_\ell (X_2 L_b, X_1) \qqand d \psi_r(Y_2, R_d Y_1) = \psi_r(Y_2 R_d, Y_1)
\]
by equation (\ref{eq:move-around-B-in-pinched}), 
\begin{align*}
\psi_\ell\left(X_2 L_{\Phi^{\inv}_{\ell, X_1X_2}(b)},  X_1\right) &= \Phi^{\inv}_{r, x_1}\left(S^\ell_{X_2}(b)^{-1} b S^\ell_{X_2}(b) \right) \quad \text{ and} \\ \psi_r\left(Y_2 R_{\Phi^{\inv}_{r, Y_1Y_2}(d)},  Y_1\right) &= \Phi^{\inv}_{r, Y_1}\left(S^r_{Y_2}(d) d S^r_{Y_2}(d)^{-1} \right)
\end{align*}
by Lemma \ref{lem:free-S-lem-for-T}, and
\begin{align*}
\psi_\ell\left(X_2, L_{\Phi^{\inv}_{\ell, X_1X_2}(b)} X_1 \right)&= \Phi^\inv_{\ell, X_2}(b )^{-1}b = S^\ell_{X_2}(b)^{-1} \quad \text{and} \\ \psi_r\left(Y_2, R_{\Phi^{\inv}_{r, Y_1Y_2}(d)} Y_1 \right) &= d\Phi^\inv_{r, Y_2}(d )^{-1} = S^r_{Y_2}(d)^{-1}
\end{align*}
by equation (\ref{eq:inverse-times-pinched}).  Hence
\begin{align*}
\Psi_e\left(\Phi^\inv_{\ell, X_1X_2}(b), c, \Phi^{\inv}_{r, Y_1Y_2}(d)\right)&=  K_{X_2, Y_2}\left(\Phi^\inv_{\ell, X_2}(b), c, \Phi^{-1}_{r, Y_2}(d)\right), \\
\Psi_{o,0}\left(\Phi^\inv_{\ell, X_1X_2}(b), c, \Phi^{\inv}_{r, Y_1Y_2}(d)\right)&= S^\ell_{X_2}(b)^{-1} c S^r_{Y_2}(d)^{-1}  ,\\
\Psi_{o,r}\left(\Phi^\inv_{\ell, X_1X_2}(b), c, \Phi^{\inv}_{r, Y_1Y_2}(d) \right)&= S^\ell_{X_2}(b)^{-1}K_{X_2, Y_2}\left(\Phi^\inv_{\ell, X_2}(b), c, \Phi^{-1}_{r, Y_2}(d)\right) d^{-1} S^r_{Y_2}(d)^{-1} \\
\Psi_{o,\ell}\left(\Phi^\inv_{\ell, X_1X_2}(b), c, \Phi^{\inv}_{r, Y_1Y_2}(d) \right)&=  S^\ell_{X_2}(b)^{-1} b^{-1} K_{X_2, Y_2}\left(\Phi^\inv_{\ell, X_2}(b), c, \Phi^{-1}_{r, Y_2}(d)\right) S^r_{Y_2}(d)^{-1}       \\
\Psi_{o,\ell r}\left(\Phi^\inv_{\ell, X_1X_2}(b), c, \Phi^{\inv}_{r, Y_1Y_2}(d) \right) &= S^\ell_{X_2}(b)^{-1} b^{-1} K_{X_2, Y_2}\left(\Phi^\inv_{\ell, X_2}(b), c, \Phi^{-1}_{r, Y_2}(d)\right) d^{-1} S^r_{Y_2}(d)^{-1},
\end{align*}
and $\Psi_0\left(\Phi^\inv_{\ell, X_1X_2}(b), c, \Phi^{\inv}_{r, Y_1Y_2}(d)\right)$ equals
\[
 S^\ell_{X_2}(b) K_{X_1, Y_1}\left(\Phi^\inv_{\ell, X_1}\left(S^\ell_{X_2}(b)^{-1} b S^\ell_{X_2}(b) \right), \Psi_{o'}\left(\Phi^\inv_{\ell, X_1X_2}(b), c, \Phi^{\inv}_{r, Y_1Y_2}(d) \right), \Phi^{\inv}_{r, Y_1}\left(S^r_{Y_2}(d) d S^r_{Y_2}(d)^{-1}\right)    \right) S^r_{Y_2}(d).
\]
By using the fact that
\[
\Upsilon_{X_1X_2, Y_1Y_2}(b,c,d) := \Psi_e\left(\Phi^\inv_{\ell, X_1X_2}(b), c, \Phi^{\inv}_{r, Y_1Y_2}(d)\right) + \Psi_0\left(\Phi^\inv_{\ell, X_1X_2}(b), c, \Phi^{\inv}_{r, Y_1Y_2}(d)\right)
\]
and by using Definition \ref{defn:S-Transform}, one may expand
\[
S^\ell_{X_2}(b)S_{X_1, Y_1}\left(S^\ell_{X_2}(b)^{-1}bS^\ell_{X_2}(b),  \, \, S^\ell_{X_2}(b)^{-1}      S_{X_2, Y_2}(b,c,d) S^r_{Y_2}(d)^{-1}, \, \, S^r_{Y_2}(d) d S^r_{Y_2}(d)^{-1}\right) S^r_{Y_2}(d)
\]
(to obtain 16 terms) and use the above equations to obtain $S_{X_1X_2,Y_1Y_2}(b,c,d)$.
\end{proof}

\end{document}